\documentclass[11pt]{amsart}
\usepackage{amsthm,amssymb,mathrsfs,mathtools,empheq}
\usepackage[shortlabels]{enumitem}
\usepackage[T1]{fontenc}
\usepackage{hyperref}
\usepackage{dsfont}
\usepackage[toc,page]{appendix}
\usepackage{setspace}

\usepackage{pgfplots}
\pgfplotsset{compat=newest}
\usepgfplotslibrary{polar}
\usepgflibrary{shapes.geometric}
\usetikzlibrary{calc}

\newcommand{\eps}{\ensuremath{\varepsilon}}

\mathtoolsset{showonlyrefs}
\usepackage{fullpage,color}

\usepackage{xcolor}
\usepackage{verbatim}
\usepackage{todonotes}

\allowdisplaybreaks

\newtheorem{theorem}{Theorem}[section]
\newtheorem*{theorem*}{Theorem}
\newtheorem{corollary}[theorem]{Corollary}

\newtheorem{lemma}[theorem]{Lemma}
\newtheorem{proposition}[theorem]{Proposition}

\newtheorem*{proposition*}{Proposition}

\newtheorem*{conjecture*}{Conjecture}

\theoremstyle{definition}
\newtheorem{definition}[theorem]{Definition}
\newtheorem{remark}[theorem]{Remark}

\newtheorem{assumption}[theorem]{Assumption}

\numberwithin{equation}{section}

\usepackage{scalerel}
\usepackage{stackengine,wasysym}

\newcommand{\tr}{\mathrm{tr}}
\newcommand{\toup}{\nearrow}

\newcommand{\id}{\mathrm{id}}

\newcommand{\UP}{V}

\newcommand{\rK}{\mathrm{K}}
\newcommand{\oK}{\mathrm{K_0}}
\newcommand{\loc}{\mathrm{loc}}
\newcommand{\Leb}{\mathrm{Leb}}
\newcommand{\then}{\Rightarrow}

\def\bN {\mathbb{N}}

\newcommand \bR {\mathbb{R}}
\newcommand \lL {\mathbf{L}}

\def\bE {\mathbb{E}}

\newcommand \bK {\mathbb{K}}

\def\bB {\mathbb{B}}
\def\bD {\mathbb{D}}
\newcommand \bU {\mathbf{U}}
\newcommand \bV {\mathbf{V}}
\newcommand \bF {\mathbf{F}}

\def\cL {\mathcal{L}}
\def\cM{\mathrm{M}}
\def\cTM{\mathrm{TM}}

\def\cR {\mathcal{R}}

\newcommand{\rb}{)]}

\newcommand{\tx}[1]{\mathrm{#1}}

\newcommand{\sh}[1]{#1^\sharp}

\newcommand{\supp}{\operatorname{supp}}

\newcommand{\ud}{\mathrm{\,d}}
\newcommand{\ut}{\mathrm{t}}

\newcommand{\lin}{_{\textsc{l}}}

\newcommand{\threepartdef}[6]
{
	\left\{
	\begin{array}{lll}
		#1 & \mbox{if } #2 \\
		#3 & \mbox{if } #4 \\
		#5 & \mbox{if } #6
	\end{array}
	\right.
}

\begin{document}

	\title{Wave maps in dimension $1+1$ with an external forcing}
	
	\author{Zdzis{\l}aw Brze\'zniak}
	\address{Department of Mathematics\\
		University of York\\
		Heslington, York YO10 5DD\\
		UK}
	\email{zdzislaw.brzezniak@york.ac.uk}
	
	\author{Jacek Jendrej}
	\address{CNRS \& LAGA, Universit\'e Sorbonne Paris Nord}
	\email{jendrej@math.univ-paris13.fr}
	
	\author{Nimit Rana}
	\address{Department of Mathematics\\
		University of York\\
		Heslington, York YO10 5DD\\
		UK}
	\email{nimit.rana@york.ac.uk}
	
	\date{\today}

	\begin{abstract}
		This paper aims to establish the local and global well-posedness theory in $L^1$, inspired by the approach of Keel and Tao \cite{KT98}, for the forced wave map equation in the ``external'' formalism. In this context, the target manifold is treated as a submanifold of a Euclidean space. As a corollary, we reprove Zhou's uniqueness result from \cite{Zhou99}, leading to the uniqueness of weak solutions with locally finite energy. Additionally, we achieve the scattering of such solutions through a conformal compactification argument.
	\end{abstract}
	
	\maketitle
	\tableofcontents
	

	\section{Introduction}\label{sec-Introduction}

	In this paper we consider the following forced wave  map equation (forced WME) on a trapezoid $K \subset \bR^{1+1}$ with values in a compact manifold $\cM$ (isometrically embedded into some Euclidean space $\bR^n$) with a general forcing term
	\begin{equation}
		\label{eqn-wave map forced-intro}
		\partial_t^2 u_i - \partial_x^2 u_i = \sum_{j,k=1}^n \Gamma_{ijk}(u)(\partial_t u_j\partial_t u_k - \partial_x u_j \partial_x u_k)
		+ \sum_{ j=1}^n P_{ij}(u)f_{j},\;\; 1 \leq i \leq n,
	\end{equation}
	where $\Gamma_{ijk}, P_{ij}: \cM \to \bR$ and $f_j: K \to \bR$, for $1 \leq i, j, k \leq n$, sufficiently regular functions for the equation above to make sense.
	When the forcing  $(f_{j})_{ j=1}^n$ is equal to $0$, equation \eqref{eqn-wave map forced-intro} is called wave  map equation (WME). A forced wave map is  a function $u: K \to \cM$ that is regular enough and satisfies the forced WME.  Similarly, a wave map is a function $u: K \to \cM$  that is sufficiently regular and satisfies the WME.
	Wave maps  have been a subject of extensive research in the past 50 years due to several reasons. Firstly, because the equation \eqref{eqn-wave map forced-intro} without the forcing term represents the most basic geometric  wave equation  providing  an excellent setup to study the non-linear wave interactions within a manifold setting.  Additionally, wave maps become pertinent in the analysis of more intricate hyperbolic Yang-Mills equations \cite{CSTZ98, Chiang13}, where they appear as special cases.
	
	\medskip

	The most natural question associated to forced WME \eqref{eqn-wave map forced-intro} is to study the well-posedness of the associated Cauchy problem with initial data
	\begin{equation}\label{eqn-init-data-intro}
		(u,\partial_t u)_{t=0} = (u_0,v_0)
	\end{equation}
	where $(u_0,v_0): \bR \to \cTM$ with  $\cTM$ being  the tangent bundle of the manifold $\cM$.
	
	\medskip

	One  of the most frequently  studied questions  has been  to find the range of $s \in \mathbb{R}$ such that the Cauchy problem for WME   with initial data $(u_0,v_0) \in H^s_{\loc}\times H^{s-1}_{\loc}$ is locally or globally well posed.
	It has been  established,  see \cite{KT98} and \cite{MNT10}, that the Cauchy problem for WME is locally well-posed  if  $s > \frac{3}{4}$.  It has also been proved, see  \cite{Tao00} for details,
	that  it  is ill-posed in the critical space $\dot{H}^{\frac{1}{2}} \times \dot{H}^{-\frac{1}{2}}$.  For $s \geq 1$,  the combination of local solution theory with energy conservation immediately yields  global well-posedness. This global solution theory extends to $\frac{3}{4} < s<1$ in the scenario when the target manifold is a sphere, see \cite[Theorem 1.3]{KT98}. Independently Zhou in \cite{Zhou99} confirms the global well-posedness result of Keel and Tao \cite{KT98} for $s \geq 1$. Zhou demonstrates  that finite energy weak solutions to the Cauchy problem for the WME  are unique and, in fact, are global strong solutions. For a more comprehensive understanding of the wave map equation, we recommend referring to the insightful surveys \cite{ShSt00, Strauss89, Tataru04}.
	
	\medskip

	The primary objective of the present  paper is to establish the existence and uniqueness of global-in-time solutions for the manifold-valued forced geometric wave equation \eqref{eqn-wave map forced-intro}. The solutions are considered with an arbitrary initial data \eqref{eqn-init-data-intro} within the scale-invariant $L^{1,1}$-space. The analysis is conducted on the Minkowski $1+1$-space and includes the external forcing $f$ from the $L_t^1 L^1_x$-space. For the detailed existence results, we  refer to Theorems \ref{thm-M-globalwave-BndTrap} and \ref{thm-M-global wave-unBndTrap} in Section \ref{sec:M valued cauchy}. The paper also provides a complete proof of the uniqueness of such solutions, as stated in Theorems \ref{thm-uniqueness-wave-map} and \ref{thm-uniqueness-wave-map-unBndTrap}. Notably, our approach to proving the uniqueness draws inspiration from Zhou's methodology in \cite{Zhou99}, but offers an independent and self-contained proof of his results. Additionally, through a conformal compactification argument from \cite[Section 9.4]{KT98}, we establish a scattering result, as presented in Theorem \ref{thm:scat-largedata}. This scattering result can be somewhat surprising in view of the fact that
	solutions do not decay in time. Nevertheless, the special null-structure of the non-linearity
	leads to scattering for all initial data.

	\medskip

	The motivation for considering forced  wave  map equations stems from the necessity to analyse the solution theory to the associated deterministic controlled equation, known as \textit{skeleton equation} in the literature, while establishing the Large Deviations Principle (LDP) for a stochastic version of these equations through a weak convergence approach developed in \cite{BD00} and \cite{BDM08}. These methods rely on variational representations of infinite-dimensional Wiener processes. As an initial step, defining an appropriate action functional with compact level sets is imperative. This functional is derived based on the minimum energy required to generate a solution of the skeleton equation.  The analysis of large deviations for stochastic non-linear wave equations for manifold-valued processes remains relatively unexplored. To the best of  our knowledge, the LDP for solutions to the stochastic wave map equation (SWME) has only been established in \cite{BGOR22}. In this work, the authors considered the Cauchy problem for both  forced and stochastic WME with $s=2$. The existence of a unique global strong solution to stochastic WME has been  established in \cite{BO07}.
	
	\medskip
	
	Another direction of research is to study WME for initial data given by the Gibbs measure, which are too rough to be covered by the deterministic theory resumed above. 	We mention that the existence of (appropriately defined) global weak solutions of WME with values in the sphere was established in \cite{BJ22}, see also \cite{BLS24} for a local well-posedness result and \cite{B23} for global well-posedness in the case of the target being a Lie group.
	
	\medskip
	
	Since we plan to establish  a  stochastic version  of the current paper in the near future,  we refrain ourselves to talk  further about this topic.  This forthcoming work will not only encompass a comprehensive proof of the existence and uniqueness of WMEs subjected to stochastic forcing but also present results regarding the LDP. An earlier study by the first author and Ondrej\'{a}t in \cite{BO11} addresses the existence of solutions to SWMEs within a framework akin to ours. However, it employs Sobolev spaces based on the Lebesgue space $L^2$, while our future work, like the present study, will employ Sobolev spaces based on the Lebesgue space $L^1$.
	
    \medskip
	
    Let us point out that Section 2.2 of the paper \cite{BJ22} by the first two named authors,  contain results related to the present paper but for the so called Goursat problem, which can be seen as the wave map equation \eqref{eqn-wave map forced} in the null coordinates:
    \begin{equation}
        \label{eq:wm-cauchy}
    \begin{gathered}
        \partial_u\partial_v \phi(u, v) = -(\partial_u\phi(u, v)\cdot\partial_v\phi(u, v))\phi(u, v), \qquad\text{for all }(u, v) \in (0, \infty)^2, \\
        \phi(u, 0) = \phi_+(u), \quad \phi(0, v) = \phi_-(v), \qquad\text{for all }u, v \in [0, \infty),
        \end{gathered}
    \end{equation}
where the boundary data  $\bigl(\phi_+,\phi_-\bigr)$, is assumed to be continuous and
the  compatibility condition $\phi_+(0) = \phi_-(0)$ is given along the  characteristic.

	\medskip

	The organisation of the paper is as follows.
	In Section \ref{sec:notation}, we present the essential notation and define the function spaces used throughout the whole paper. Section \ref{sec:linear-wave} delves into the classical properties of the linear wave equation with $L^1_{\loc}$-forcing, along with pertinent properties of solutions to transport equations that become relevant later in the paper.
	Section \ref{sec:framework} provides an in-depth discussion on the framework adopted for the well-posedness of manifold-valued wave equation  in this paper.
	In Section \ref{sec:Rn-valued-WE-Cauchy}, we establish the local well-posedness of the $\bR^n$-valued wave equation \eqref{eqn-wave map forced-intro} with small initial data $(u_0,v_0)$ and forcing $f$, see  Theorems \ref{thm-nonlin-cauchy-bigData-bndTrap} and \ref{thm-nonlin-cauchy-bigData-UnBndTrap}. The proof relies on a continuity Lemma \ref{lem-continuity}, presented in Appendix \ref{sec:continuity-lem}. Despite the absence of a well-established global solution theory for large data $\bR^n$-valued solutions to \eqref{eqn-wave map forced-intro}, Theorem \ref{thm:nonlin-scattering} suggests that if such a global solution exists, scattering of $\bR^n$-valued solutions holds true.
	Section \ref{sec:M valued cauchy} focuses on the global well-posedness  theory for the forced WME \eqref{eqn-wave map forced-intro},  with arbitrary  initial data and forcing. The proofs of the key results, i.e. Theorems \ref{thm-nonlin-cauchy-L1}, \ref{thm-M-globalwave-BndTrap} and \ref{thm-M-global wave-unBndTrap},   rely on a specific ``energy'' method described in Proposition \ref{prop-noncon} and a special  density argument, see Theorem \ref{lem-density}  in Appendix \ref{app-geom}. Furthermore, Theorem \ref{thm:scat-largedata} establishes a scattering result.
	In Section \ref{sec:generalisation-L1loc} we extend the findings from subsections  \ref{sec:wave-map-M-global-soln-existence} and \ref{sec:wave-map-M-global-soln-uniq} to accommodate weaker assumptions on the data. The paper concludes in Section \ref{sec-Zhou}, where we deduce the Zhou uniqueness result \cite[Theorem 1.3]{Zhou99} from Theorems~ \ref{thm-M-global wave-unBndTrap} and \ref{thm-uniqueness-wave-map-unBndTrap}. 
	
	\medskip

	\section{Notation}\label{sec:notation}
	When we write $L>0$ we mean that $L$ is a positive real number.
	All functions considered in this article are $\bR^n$-valued. On $\bR^n$ we always consider the $\ell^1$-norm defined by
	\begin{equation}\label{eqn-R^n-l^1-norm}
		\vert z \vert\coloneqq  \sum_{i=1}^n \vert z_i\vert, \;\; z=(z_1,\ldots,z_n) \in \bR^n.
	\end{equation}
	For a (Lebesgue) measurable function $f=(f_1,\ldots,f_n):\Omega \to \bR^n$, where $\Omega \subset \bR^m$ is a Lebesgue measurable set,  we put
	\begin{equation}\label{eqn-L^1-norm}
		\| f \|_{L^1(\Omega)}\coloneqq  \int_\Omega \sum_{i=1}^n \vert f_i(x) \vert\, \ud x,
	\end{equation}
	where $\ud x$ denotes the integration w.r.t. the Lebesgue measure on $\bR^m$.
	
	By $L^1_{\loc}(\Omega;\bR^n)$ we mean the metrizable topological vector space equipped with a natural countable family of seminorms defined by
	\begin{equation}
		p_j (u) \coloneqq  \| u\|_{L^1(B_j \cap \Omega)}, \qquad u \in  L^1_{\loc}(\Omega;\bR^n), j \in \bN,
	\end{equation}
	where $(B_j)_{j \in \bN}$ are open balls in $\bR^m$ such that $\cup_{j \in \bN} \overline{B_j \cap \Omega} = \Omega$.
	
	Let us put, for $T \in (0,\infty]$,
	\begin{equation}\label{eqn-interval}
		[0,T\rb\coloneqq \begin{cases}[0,T], & \mbox{ if } T \in (0,\infty), \\
			[0,\infty), & \mbox{ if } T =\infty,
		\end{cases} \quad \textrm{ and } \quad
		(0,T\rb\coloneqq \begin{cases}(0,T], & \mbox{ if } T \in (0,\infty), \\
			(0,\infty), & \mbox{ if } T =\infty.
		\end{cases}
	\end{equation}
	For $T \in (0,\infty]$ and a Banach space $E$, we write $C_b([0,T)];E)$ for the space of all $E$-valued bounded and continuous function defined on interval $[0,T)]$. Moreover, for an open subset $\Omega \subset \bR^m$ and $k \in \bN$ we write $C^k(\Omega;E)$ for the collection of all $E$-valued $k$-times differentiable functions such that the $k$-th derivative is continuous. If $f \in C^k(\Omega;E)$ for every $k \in \bN$, then we say that  $f \in C^\infty(\Omega;E)$ or $f$ is $E$-valued $C^\infty$-class function. We will write $C_0^\infty(\Omega;E)$ to denote the space of $E$-valued smooth functions with compact support in $\Omega$.
	
	\begin{remark}
		When the target space of a function is clear from the context we will not write it explicitly. For example we write  $L^1(\Omega)$ instead $L^1(\Omega;\bR^n)$.
	\end{remark}
	
	\begin{proposition}\label{prop-local-integrable}
		If $T\in (0,\infty]$, the  function
		\begin{equation*}
			u: [0,T\rb\times \bR  \to \bR^n
		\end{equation*}
		is  locally  integrable  if and only if  for every compact trapezoid $K \subset [0,T\rb\times \bR$, the restriction of $u$ to $K$ is integrable.
	\end{proposition}	
	\begin{definition} \label{def-W11-notation}
		Let $I \subset \bR$ be an open interval. We say that a measurable function $v: I \to \bR^n$ belongs to  $\dot{W}^{1,1} (I;\bR^n)$  if and only if  it is  weakly differentiable and its weak derivative $Dv$ is integrable over $I$.\\
		Equivalently, a function $v$ belongs to  $\dot{W}^{1,1} (I;\bR^n)$ if and only if  it is locally absolutely continuous, so that it is almost everywhere differentiable,
		and the derivative function $v^\prime$   is integrable over $I$ and $Dv=v^\prime$ almost everywhere on $I$.
	\end{definition}
	\begin{remark}
		Assume that $v\in \dot{W}^{1,1} (I;\bR^n)$ for some $I = (a,b)$. We can take $(a,b) = (-\infty, \infty)$ for example. In view of \cite[Theorem 7.20]{Rudin74RCA}, $v(y)-v(x)=\int_x^y Dv(s)\, ds$, for all $x<y \in I$. In particular, $v$ is bounded and the following two limits
		\begin{align}
			v(a)&\coloneqq  \lim_{x \to a} v(x), \\
			v(b)&\coloneqq  \lim_{x \to b} v(x),
		\end{align}
		exist.  Moreover,  the function
		$\tilde{v}(x)=\int_a^x Dv(s)\, \ud s$  differs from $v$ by a constant (and $ \lim_{x \to a}\tilde{v}(x)=0$).
	\end{remark}
	\noindent The space $\dot{W}^{1,1} (I; \bR^n)$ is a vector space endowed with a seminorm
	\begin{equation}\label{eqn-dotW^11}
		\Vert v  \Vert_{\dot{W}^{1,1}(I)}\coloneqq  \Vert Dv \Vert_{L^1(I)}.
	\end{equation}
	Moreover, the space
	\begin{equation}\label{eqn-L^infty cap W^11}
		(L^\infty \cap \dot{W}^{1,1})(I;\bR^n) \coloneqq L^\infty(I;\bR^n) \cap \dot{W}^{1,1} (I;\bR^n)
	\end{equation}
	is a Banach space endowed with the norm
	\begin{equation}\label{eqn-L^infty cap W^11-norm}
		\Vert u \Vert_{(L^\infty \cap \dot{W}^{1,1})(I)}=\Vert u \Vert_{ L^\infty(I)}+ \Vert Du \Vert_{L^{1} (I)}.
	\end{equation}
	Let us note that
	\begin{equation}\label{eqn-L^infty cap W^11-1}
		(L^\infty \cap \dot{W}^{1,1})(I;\bR^n) = C_b(I;\bR^n) \cap \dot{W}^{1,1} (I;\bR^n),
	\end{equation}
	where $C_b(I;\bR^n)$ is the space of all $\bR^n$-valued bounded and continuous function defined on $I$.

	\subsection{Trapezoids}
	In this subsection we provide definition of an abstract trapezoid, on which we consider the wave maps equation later, and some examples of it. The following definition is motivated from the notion of ``initial set'' defined in \cite{Norris95}.
	\begin{definition}\label{def-abstract trapezoid} A closed,  convex and of positive Lebesgue measure set $K \subseteq [0,\infty) \times \bR$ is said to be a (-n abstract) trapezoid if and only if for every $(t,x) \in K$, the closed triangle $T_{(t,x)}$ with vertices $(t,x), (0,x-t)$ and $(0,x+t)$ is a subset of $K$. \\
		In this setting we further define, for any $t \in (0,\infty)$,
		\begin{equation}\label{eqn-K_t}
			K_t \coloneqq  \{x \in \bR: (t,x) \in K \} \cong  K \cap (\{t\} \times \bR).
		\end{equation}
	\end{definition}	

	Obviously, every set $K_t$ is convex and, possibly empty. If $K_t \not= \emptyset$ for some $t>0$ and $s\in [0,t]$, then $K_s \not= \emptyset$. Note that $\{t\} \times K_t \subset K$.
	If the set $\{ t\in [0,\infty): K_t \not= \emptyset\}$ is unbounded above, then we define the height of $K$ to be equal to $\infty$. Otherwise, we define the height of $K$ to be equal the maximum of that set. Note that if $K^1\subset K^2$ are trapezoids, then   for every $t \geq 0$, $K^1_t\subset K^2_t$.

	Let us see the following three natural examples of an abstract trapezoid.
	\begin{definition}\label{def-trapezoid}
		A compact  trapezoid is a subset  $K$ of $[0,\infty)\times \bR$  such that there exist $x_0 \in \bR$, $L>0$ and $t_0 \in (0,L]$, see Figure \ref{TrapK}, which give
		\begin{equation}\label{eqn-trapezoid}
			K =  \{(t,x) \in [0,t_0]\times \bR: x \in [  x_0 - L + t,x_0 + L - t] \}.
		\end{equation}
		The points  $(0, x_0 - L)$, $(0, x_0 + L)$, $(t_0, x_0 + L - t_0)$ and $(t_0, x_0 - L + t_0)$ are called the vertices of $K$.
		The number  $L$ is equal to the \textbf{half-length} of $K$ and $t_0$ will be called the \textbf{height} of $K$.
		The set
		\begin{equation}\label{eqn-K_0}
			K_0=[x_0 - L, x_0+L]
		\end{equation}
		will be called 	the \textbf{base} of $K$.

		Note that  $x_0$ the middle point of $K_0$ and that   $K_0$ is the ``left bottom'' boundary of $K$, i.e.
		the set  $\{0\}\times K_0$ is the projection of $K$ onto line $\{0\}\times \bR$.
		By the left boundary of $K$ we will understand the set $\{0\} \times K_0$.
		
	\end{definition}

	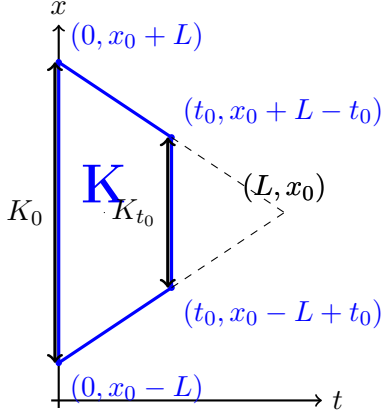
\begin{figure}
		\begin{tikzpicture}[scale=0.50]
			\draw[->,  thick] (-0.2,0)--(7,0) node[right]{$t$};	
			\draw[->, thick] (0,-0.2)--(0,10) node[above]{$x$};
			\draw[blue,very thick] (0,1) node[anchor=north west]{$(0,x_0-L)$}
			-- (0,9) node[anchor=south west]{$(0,x_0+L)$}
			-- (3,7) node[anchor=south west]{$(t_0,x_0+L-t_0)$}
			-- (3,3) node[anchor=north west]{$(t_0,x_0-L+t_0)$}
			-- cycle;
			\draw[dashed]  (1.2,5)  -- (1.2,5)node[anchor=south]{\textbf{\textcolor{blue}{\huge{K}}}};
			\filldraw[blue] (0,1) circle (2pt);
			\filldraw[blue] (0,9) circle (2pt);
			\filldraw[blue] (3,7) circle (2pt);
			\filldraw[blue] (3,3) circle (2pt);
			\draw[dashed]  (3,3)  -- (6,5)node[anchor=south]{$(L,x_0)$};
			\draw[dashed]  (3,7)  -- (6,5)node[anchor=south]{$(L,x_0)$};
			\coordinate (v1) at (-0.1,1);
			\coordinate (v2) at (-0.1,9); \draw[<->,very thick] (v1) -- node[left] {$K_0$} (v2);
			\coordinate (v1) at (2.9,3);
			\coordinate (v2) at (2.9,7); \draw[<->,very thick] (v1) -- node[left] {$K_{t_0}$} (v2);
		\end{tikzpicture}
		\caption{A bounded  trapezoid $K$ with vertices $(0, x_0 - L)$, $(0, x_0 + L)$, $(t_0, x_0 + L - t_0)$ and $(t_0, x_0 - L + t_0)$ for some $x_0>0$ and $0 < t_0 \leq L$.}
		\label{TrapK}
	\end{figure}

	\begin{remark}\label{rem-max-trapezoid}
		If $t_0=L$, then $K$, the trapezoid in the Definition \ref{def-trapezoid},   is a triangle with vertices $(0, x_0 - L)$, $(0, x_0 + L)$, $(L, x_0)$,
		i.e. see Figure \ref{TrapK},
		\begin{equation*}
			K= \{(t,x) \in [0,L]\times \bR: x \in [x_0 - L + t, x_0 + L - t] \}.
		\end{equation*}
		Such a triangle $K$ will be called the ``maximal trapezoid'' with base $K_0$ and height $L$.	
	\end{remark}

	The next two examples consist of unbounded and semi-bounded closed trapezoids.
	\begin{definition}\label{def-trapezoid unbounded}
		An unbounded closed trapezoid is a set $K$ of the form $[0,t_0\rb \times \bR$ for some $t_0 \in (0,\infty]$. The height and base of such a trapezoid is $t_0$ and $\bR$, respectively. In this setup, the half-length $L$ of $K$ is, by definition, equal to $\infty$.
	\end{definition}

	\begin{definition}\label{def-trapezoid semi bounded}
		
		A semi-bounded closed trapezoid is a set of the following form
		\begin{align}\label{eqn-semi-bounded-up}
			K^+(b,t_0)\coloneqq  & \{ (t,x)  \in [0,t_0\rb \times \bR: x \geq b+t\},
			\\
			\label{eqn-semi-bounded-down}
			K^-(a,t_0)\coloneqq  & \{ (t,x)  \in [0,t_0\rb \times \bR: x \leq a-t\},
		\end{align}
		for some $a,b \in \bR$ and $t_0 \in (0,\infty]$.
		The number $t_0$ is equal to the height of the trapezoids $K^+(b,t_0)$ and $K^-(a,t_0)$, while the base of $K^+(b,t_0)$ is the set $K_0=[b,\infty)$ and
		the base of $K^-(a,t_0)$ is the set $K_0=(-\infty,a]$. In both cases, the half-length $L$ of $K$ is, by definition, equal to $\infty$.
	\end{definition}

	\begin{remark}
		If a statement is valid for any abstract trapezoid in the sense of Definition \ref{def-abstract trapezoid}, then we will just write trapezoid without mentioning the words abstract or bounded or unbounded or semi-bounded.
	\end{remark}

	The following invariance property is well known. It plays a fundamental (somehow hidden) role in our considerations.
	\begin{proposition}\label{prop-trapezoid}
		Assume that $K$ is a trapezoid with base $K_0$. Then the following holds.
		\begin{equation}\label{eqn-trapezoid-1}
			\mbox{ If } (t,x) \in K \mbox{ then } (s,x+t-s), (s,x-t+s) \in K \mbox{ for all } s\in [0,t].
		\end{equation}
		In particular, if $(t,x) \in K$ then $x+t,x-t\in K_0$.
	\end{proposition}

	In what follows we will introduce  somehow unusual notation.
	\begin{definition}\label{def-C0inftyK}
		Let $K$ be a trapezoid. Define the double of $K$ to be
		$$ \tilde{K} \coloneqq  K \cup (-K) \subset \bR\times \bR,$$
		where $-K= \{(-t,x): (t,x) \in K  \}$. Then  the space $C_0^\infty((K))$ is defined as
		$$C_0^\infty((K))\coloneqq  \{ \varphi|_K : \varphi \in C_0^\infty(\textrm{int}(\tilde{K})) \}, $$
		where $\textrm{int}(A)$ denotes the interior of set $A$. For $T \in (0,\infty)$, we say that a $C^\infty$-class function $\phi: [0,T]\times \bR \to  \bR^n$  belongs to a class
		$C_0^\infty([0,T]\times \bR;\bR^n)$  if and only if
		there exist compact trapezoids $\widetilde{K}$ and $\widetilde{\widetilde{K}}$ such that $\widetilde{K} \subset \textrm{int}(\widetilde{\widetilde{K}})$, $\supp(\phi) \subset \widetilde{K}$  and $\phi|_{\widetilde{\widetilde{K}}}$ is of $C^\infty$-class.
	\end{definition}		
	
	\begin{remark}
		Note that if $K$ is a compact trapezoid, as in Definition \ref{def-trapezoid}, and $\phi \in  C_0^\infty((K))$, then $\phi$ equal to $0$ in a neighbourhood of every point $(t,x) \in \partial(K) \setminus (\{0\} \times K_0)$ but not for points belonging to the left boundary, i.e.
		$(t,x) \in \{0\} \times K_0$.
	\end{remark}

	\begin{definition}\label{def-space H(K)}
		If $K$ is a trapezoid with basis $K_0$ of height $T_0 \in (0,\infty]$,
		then  by $\mathscr{H}(K)$ we denote  the collection of all  continuous  functions $u:K \to \bR^n$
		for which  the weak derivatives $\partial_t u$ and $\partial_x u$ in $L^1(\textrm{int}(K))$ exist and satisfy
		\begin{align}
			& \partial_x u \in C_b([0,T_0)]; L^1(K_\cdot)), \label{eqn-H(K)} \\
			& \partial_t u \in C_b([0,T_0)]; L^1(K_\cdot)). \label{eqn-H(K)-2}
		\end{align}

		Let us recall that a measurable function $v: K \to \bR^n$ belongs to $C_b([0,T_0)], L^1(K_\cdot))$  if and only if  the function $\tilde{v}$ which is the extension of $v$
		to the rectangle $[0,T_0)] \times K_0$ by $0$ in $([0,T_0)] \times K_0) \setminus K$,
		belongs to $C_b([0,T_0)];L^1(K_0))$.
	\end{definition}
	
	Let us point out that the continuity property of function $u$ in Definition \ref{def-space H(K)} is not needed as it follows from \eqref{eqn-H(K)}-\eqref{eqn-H(K)-2}. Indeed, the conditions \eqref{eqn-H(K)}-\eqref{eqn-H(K)-2} imply that
	$$ u \in C([0,T_0)]; W^{1,1}(K_\cdot)) \cap W^{1,\infty}([0,T_0); L^1(K_\cdot)).$$
	It is easy to prove that $\mathscr{H}(K)$ is a separable Banach space with norm
	\begin{equation}\label{eqn-H(K)-norm}
		\Vert u \Vert_{\mathscr{H}(K)}\coloneqq \sup_{(t,x)\in K}  |u(t,x)| +\sup_{t \in [0,T_0)]}  \int_{K_t}  \left(|\partial_x u(t,x)| +|\partial_t u(t,x)| \right) \ud x.
	\end{equation}
	Let us point out that if $T_0$ is finite, then
	$$C_b([0,T_0)]; L^1(K_\cdot))=C([0,T_0]; L^1(K_\cdot)). $$
	
	\begin{remark}\label{rem-H spaces}
		In the two special cases of a trapezoid $[0,\infty)\times \bR$ or $[0,T]\times \bR$, for $T>0$, the spaces $\mathscr{H}(K)$ defined in the Definition \eqref{def-space H(K)} can characterized in the following simpler way:
		\begin{align}\label{eqn-H_infty}
			\mathscr{H}([0,\infty)\times \bR)
			&\coloneqq  \Bigl\{u: [0,\infty)\times \bR \to \bR^n:
			\\
			& \qquad u \in  C_b([0,\infty); L^\infty(\bR \cap \dot{W}^{1,1}(\bR)),  \partial_t u \in C_b([0,\infty);L^1(\bR) \bigr\},
		\end{align}
		and
		\begin{align}\label{eqn-H_T}
			\mathscr{H}([0,T]\times \bR)& \coloneqq  \bigl\{ u: [0,T]\times \bR \to \bR^n:
			\\
			&  \qquad
			u \in C([0,T]; L^\infty(\bR) \cap \dot{W}^{1,1}(\bR)), \partial_t u \in C([0,T],L^1(\bR)) \bigr\}.
		\end{align}
		Note that  $\mathscr{H}([0,\infty)\times \bR)$  and $\mathscr{H}([0,T]\times \bR)$ are  Banach spaces.
	\end{remark}
	
	The definition above implies the  following simple version of a trace theorem whose proof follows by standard arguments.
	\begin{theorem}
		\label{thm-trace}
		In the framework of Definition \ref{def-space H(K)}, if $u \in \mathscr{H}(K)$, then for every $s \in [0,T_0)]$,
		the ``restrictions'' $\bigl(  u(s, \cdot),\partial_t u(s,\cdot)\bigr)$ exist and belong to the space $  (L^\infty \cap \dot{W}^{1,1}(K_s))\times L^1(K_s)$. Moreover, the  corresponding linear map
		\begin{equation}\label{eqn-trace maps}
			\begin{split}
				\tr_s&:\mathscr{H}(K) \ni u \mapsto  \bigl( u(s, \cdot), \partial_t u(s,\cdot) \bigr)   \in  (L^\infty \cap \dot{W}^{1,1}(K_s))\times L^1(K_s)
			\end{split}
		\end{equation}
		is a contraction.
		Moreover, if $\phi \in C_0^\infty(\textrm{int}(K_0))$ then the map
		\begin{align}\label{eqn-trace-maps-integral}
			\Upsilon_{s,\phi}&:\mathscr{H}(K) \ni u \mapsto 
\int_{K_s} \varphi(x) ~ u(s,x) \cdot \partial_t u(s,x) \, \ud x
		\end{align}
		is continuous. 
	\end{theorem}

	\section{Some classical properties of solutions of the linear wave equations}
	\label{sec:linear-wave}
	
	In this section we study some classical properties of linear wave equation with $L^1_{\loc}$-forcing.
	
	\subsection{Weak/Mild solutions and their equivalence}
	We start with the following definitions of weak/mild solution to problem \eqref{eqn-linear-wave-nonhom}.

	\begin{definition}
		\label{def:nonhom-wave-weak}
		Assume that $K$ is a trapezoid with base $K_0$. Assume that  $h \in L^1_{\loc}(K;\bR^n)$ and $u_0, v_0 \in L^1_{\loc}(K_0;\bR^n)$.
		A locally integrable function $u: K \to \bR^n$ is called a \emph{weak solution} in $K$  of the following initial value problem
		\begin{equation}
			\label{eqn-linear-wave-nonhom}
			\begin{gathered}
				\partial_t^2 u - \partial_x^2 u = h, \\
				u(0, \cdot) = u_0, \quad \partial_t u(0, \cdot) = v_0,
			\end{gathered}
		\end{equation}
		if and only if  for every $\phi \in  C_0^\infty((K))$
		\begin{equation}\label{eq:wave-sol-weak-in K}
			\iint_{K}	u(\partial_t^2 \phi - \partial_x^2 \phi)\ud x\ud t = \iint_{K} \phi h\ud x\ud t
			- \int_{K_0} u_0(x)\partial_t \phi(0,x)\ud x + \int_{K_0} v_0(x)\phi(0,x)\ud x.
		\end{equation}
	\end{definition}
	
	We have the following trivial consequence of our definition.
	\begin{proposition}\label{prop-weak solutions}
		Assume that $K$ is a trapezoid with base $K_0$.  Assume that  $u_0, v_0: K_0 \to \bR^n$ and $h: K \to \bR^n$ are locally integrable functions.
		If  a locally  integrable function $u: K  \to \bR^n$
		is  a weak solution  of the problem \eqref{eqn-linear-wave-nonhom} in $K$, then for every
		bounded trapezoid $\widetilde{K} \subset K$ with base $\widetilde{K}_0 \subset K_0$,  the restriction $u_{\vert \widetilde{K}}$
		is  a weak solution in $\widetilde{K}$   of  problem
		\eqref{eqn-linear-wave-nonhom} with the initial data and forcing being the restrictions of $(u_0, v_0)$ and $ h$ to $\widetilde{K}_0$.
	\end{proposition}
	\begin{proof}[Proof of Proposition \ref{prop-weak solutions}]
		Trivial because $C_0^\infty((\widetilde{K})) \subset C_0^\infty((K))$.
	\end{proof}

	\begin{definition}
		\label{def:nonhom-wave-mild}
		Assume that $K$ is a trapezoid with base $K_0$. Assume that  $h \in L^1_{\loc}(K;\bR^n)$ and $u_0, v_0 \in L^1_{\loc}(K_0;\bR^n)$.
		A locally integrable function  $u:K \to \bR^n$   is a \emph{mild solution}  of the problem \eqref{eqn-linear-wave-nonhom}
		if and only if  for every $(t,x) \in K$
		\begin{equation}
			\label{eq:wave-sol-mild}
			u(t, x)= \frac 12 \big(u_0(x + t) + u_0(x - t)\big) + \frac 12\int_{x-t}^{x+t}v_0(y)\ud y + \frac 12 \int_0^t\int_{x - t + \tau}^{x + t - \tau}h(\tau, y)\ud y\ud\tau.
		\end{equation}
	\end{definition}

	In order to show that the  Definition \ref{def:nonhom-wave-mild} is meaningful we claim that under the conditions on the data $(u_0,v_0,h)$ listed above, the function $u$ defined by formula
	\eqref{eq:wave-sol-mild} is locally integrable in $K$. First note  that $u(t,x)$ is well defined for every point $(t,x) \in K$. Next, we establish that $u$ is continuous on $K$ and thereby establishing the claim.
	For this aim, without loss of generality it is sufficient  assume  that $K$ is not only closed but also bounded.  Let us choose and fix $\eps>0$. Then by \cite[Theorem 6.11]{Rudin74RCA}, we can find $\eta>0$ such that for every Lebesgue measurable  subset $A \subset K$,
	$\int_A \vert h\vert \ud y\ud\tau\leq \eps$, provided $\Leb(A) \leq \eta$. It is obvious geometrically that we can find $\delta>0$ such that if $(t_i,x_i) \in K$, $i=1,2$ and $\vert (t_2,x_2)-(t_1,x_1)\vert_1\leq \delta$, then $\Leb (K_{t_2x_2}\diamond K_{t_1x_1})\leq \eta$, where $A\diamond B=(A\setminus B)\cup (B\setminus A)$ is the symmetric difference between $A$ and $B$.
	This implies that $\vert u(t_2,x_2)-u(t_1,x_1)\vert \leq \eps $  provided $ (t_i,x_i) \in K$, $i=1,2$ and $\vert (t_2,x_2)-(t_1,x_1)\vert_1\leq \delta$.

	\begin{remark}
		For every $i \in \{1, \ldots, n\}$, the $i$-th component $u_i$ of $u$ satisfies
		\begin{equation}
			\label{eq:wave-sol-comp}
			u_i(t, x)= \frac 12 \big(u_{0,i}(x + t) + u_{0,i}(x - t)\big) + \frac 12\int_{x-t}^{x+t}v_{0,i}(y)\ud y + \frac 12 \int_0^t\int_{x - t + \tau}^{x + t - \tau}h_i(\tau, y)\ud y\ud\tau.
		\end{equation}
	\end{remark}

	The following is a version of Proposition \ref{prop-weak solutions} in the framework of mild solutions.
	
	\begin{proposition}\label{prop-mild solutions}
		Assume that trapezoids $K$ and data $(u_0,v_0,h)$ are as in Proposition \ref{prop-weak solutions}. If  $u: K  \to \bR^n$ is  a mild solution  of the problem \eqref{eqn-linear-wave-nonhom} in $K$, then for every
		bounded trapezoid $\widetilde{K} \subset K$ with base $\widetilde{K}_0 \subset K_0$,  the restriction $u_{\vert \widetilde{K}}$
		is  a mild solution in $\widetilde{K}$   of  problem
		\eqref{eqn-linear-wave-nonhom} with data $(u_0|_{\widetilde{K}_0}, v_0|_{\widetilde{K}_0}, h|_{\widetilde{K}_0})$.
	\end{proposition}
	\begin{proof}[Proof of Proposition \ref{prop-mild solutions}]  This follows easily from Proposition \ref{prop-trapezoid}.
	\end{proof}
	
	Next result shows that a mild solution to problem \eqref{eqn-linear-wave-nonhom}, in the sense of Definition \ref{def:nonhom-wave-mild},  is also a weak solution to it, in the sense of Definition \ref{def:nonhom-wave-weak}.
	\begin{theorem}\label{thm-mild solution is weak}
		Assume that $K$ is a trapezoid with base $K_0$. Assume that  $h \in L^1_{\loc}(K;\bR^n)$ and $u_0, v_0 \in L^1_{\loc}(K_0;\bR^n)$. A mild solution $u: K \to \bR^n$ to the initial value problem \eqref{eqn-linear-wave-nonhom} is also a  weak solution in $K$.
	\end{theorem}
	\begin{proof}[Proof of Theorem \ref{thm-mild solution is weak}]
		Let $\varphi_1 \in C_0^\infty(K_0)$  and $\varphi_2 \in C_0^\infty((K))$ are two mollifiers. Let us also set $\varphi_1^\eps \coloneqq  \eps^{-1}\varphi(x/\eps)$ and $\varphi_2^\eps \coloneqq  \eps^{-2}\varphi(t/\eps,x/\eps)$,  for $\eps>0$. Then  consider the following approximations,
		$$u_0^\eps \coloneqq  \varphi_1^\eps \ast u_0, \quad v_0^\eps \coloneqq  \varphi_1^\eps \ast v_0, \quad h^\eps \coloneqq  \varphi_2^\eps \ast h. $$
		Then it is well-known that $u_0^\eps \to u_0, v_0^\eps \to v_0$ and $h^\eps \to h$ as $\eps \to 0$, respectively, in $L^1_{\loc}(K_0)$ and $L^1_{\loc}(K)$, see for e.g.  \cite{Rudin91FA}.  Let us further denote the smooth mild solution of problem \eqref{eqn-linear-wave-nonhom}, which is given by \eqref{eq:wave-sol-mild}, with initial data $(u_0^\eps,u_1^\eps)$ and external force $h^\eps$, by $u^\eps $.
		
		Let us take an arbitrary $\phi \in C_0^\infty((K))$.  Then by applying the integration by parts formula twice  \cite[Theorem 2 in Section C.2]{Evans10} we get
		\begin{align}\label{eq:IBP-wave}
			\iint_K u^\eps (\partial^2_t \phi - \partial_x^2\phi) \ud x \ud t & = \iint_K \phi  (\partial^2_t u^\eps - \partial_x^2u^\eps) \ud x \ud t  + \int_{\partial K}  u^\eps  (\nu\cdot (\partial_t \phi,-\partial_x \phi)) \ud S \nonumber \\
			& \quad + \int_{\partial K} \phi  (\nu\cdot (-\partial_t u^\eps,\partial_x u^\eps)) \ud S,
		\end{align}
		where $\nu$ is the outward pointing normal vector field on $\partial K$ and $\ud S$ denotes the integration w.r.t the surface measure. Since $\phi \in C_0^\infty((K))$,
		\begin{align}
			\int_{\partial K}  & u^\eps  (\nu\cdot (\partial_t \phi,-\partial_x \phi)) \ud S \quad + \int_{\partial K} \phi  (\nu\cdot (-\partial_t u^\eps,\partial_x u^\eps)) \ud S \nonumber\\
			& = -  \int_{ K_0}  u_0^\eps(x)  \partial_t \phi(0,x) \ud x + \int_{ K_0}   \phi(0,x)  v_0^\eps(x) \ud x.
		\end{align}
		Substituting back into \eqref{eq:IBP-wave} together with $\partial_t^2 u^\eps - \partial_x^2 u^\eps = h^\eps$ gives
		\begin{align}\label{eq:weak-u-eps}
			\iint_K u^\eps (\partial^2_t \phi - \partial_x^2\phi) \ud x \ud t & = \iint_K \phi  h \ud x \ud t  -  \int_{ K_0}  u_0^\eps  \partial_t \phi(0) \ud x + \int_{ K_0}   \phi(0)  v_0^\eps \ud x.
		\end{align}

		Now let any $(t,x) \in K$ and consider the associated bounded and closed triangle $T(t,x) \subset K$ with vertices $(t,x), (0,x-t), (0,x+t)$. Since $T(t,x)$ is compact, $u^\eps, u \in L^1_{\loc}(K)$  and $u_0^\eps \to u_0, v_0^\eps \to v_0$ in $L_{\loc}^1(K_0)$ and $h^\eps \to h$ in $L_{\loc}^1(K)$ as $\eps \to 0$, by the Lebesgue Dominated Convergence Theorem we can pass to the limit $\eps \to 0$ in \eqref{eq:weak-u-eps} and get \eqref{eq:wave-sol-weak-in K}. Since the point $(t,x) \in K$ was chosen arbitrary, we deduce the claim and complete the proof  of Theorem \ref{thm-mild solution is weak}.
	\end{proof}

	The next theorem is the main result of this subsection which demonstrates the equivalence between the notions of weak and mild solutions for the linear non-homogeneous initial value problem \eqref{eqn-linear-wave-nonhom}.

	\begin{theorem}\label{thm-weak soluntions uniqueness}
		Assume that $K$ is a trapezoid with base $K_0$. Assume that  $h \in L^1_{\loc}(K;\bR^n)$ and $u_0, v_0 \in L^1_{\loc}(K_0;\bR^n)$. Then
		\begin{trivlist}
			\item[(i)]  Problem \eqref{eqn-linear-wave-nonhom}
			has a unique weak solution which  is given by \eqref{eq:wave-sol-mild}. In particular, a weak solution  is a mild solution.
			\item[(ii)]
			Moreover, if $\widetilde{K} \subset K$ is a trapezoid with base $\widetilde{K_0}$ and two triples of
			locally integrable functions $(u_0, v_0,h)$ and $(\tilde{u}_0, \tilde{v}_0,\tilde{h})$ coincide on $\widetilde{K_0}$, then
			the corresponding solutions $u$ and $\tilde{u}$ coincide on $\widetilde{K}$.
		\end{trivlist}
	\end{theorem}
	\begin{proof}[Proof of Theorem \ref{thm-weak soluntions uniqueness}] Without loss of generality, we can assume $n = 1$.
		\begin{proof}[Proof of item (i)] The existence follows by approximating $u_0, v_0$ and $h$ by smooth functions and passing to the limit in the formula defining weak solutions.
			In order to prove uniqueness,  let $u\coloneqq  u_1-u_2$ be the difference of two weak solutions $u_1$ and $u_2$ of \eqref{eqn-linear-wave-nonhom} with same initial and external data.  Then $u$ satisfies
			\begin{equation}\label{eqn-de Bois Raymond-assumption}
				\iint_{K} u(\partial_t^2 \phi - \partial_x^2 \phi)\ud x\ud t =0,\qquad\text{for all }\phi \in C_0^\infty((K)).
			\end{equation}
			We will to show that $u = 0$ Lebesgue  almost everywhere  on $K$. Such an assertion can be called a generalization of the du Bois-Raymond Lemma.
			
			It suffices to prove that
			\begin{equation}\label{eqn-de Bois Raymond-assertion}
				\iint_{K} u \psi \ud x\ud t =0,\qquad\text{for every } \psi \in C_0^\infty(\bR\times \bR): \supp \psi \subset \textrm{int}(K).
			\end{equation}
			For this aim,  let us choose and fix $\psi \in C_0^\infty(\bR\times \bR)$ such that  $\supp \psi \subset \textrm{int}(K)$. Then we can find
			$T_0>0$ such that \[ \supp~{\psi} \subset ((-\infty,T_0)\times \bR) \cap \textrm{int}(K).\]
			Now set $\widetilde{\phi} \in C^\infty(\bR\times \bR)$ be the unique classical solution to the following boundary value problem on $\bR\times \bR$:
			\begin{equation}
				\label{eqn-linear-wave-nonhom-bdry}
				\begin{gathered}
					\partial_t^2  \widetilde{\phi}(t, x) - \partial_x^2  \widetilde{\phi}(t, x) = \psi(t, x), \\
					\widetilde{\phi}(T_0, \cdot) = 0, \quad \partial_t  \widetilde{\phi}(T_0, \cdot) = 0.
				\end{gathered}
			\end{equation}
			In other words, $\widetilde{\phi}$ is given by
			\[
			\widetilde{\phi}(t,x)  =\begin{cases} \frac{1}{2}\int_{T_0}^{t} \int_{x-t+\tau}^{x+t-\tau}  \psi(\tau,y)\, dy\, d\tau, & \mbox{ if } (t,x) \in [T_0,\infty)\times \bR, \\
				\frac{1}{2}\int_{t}^{T_0} \int_{x+t-\tau}^{x-t+\tau}  \psi(\tau,y)\, dy\, d\tau, & \mbox{ if }(t,x) \in (-\infty, T_0] \times \bR.
			\end{cases}
			\]
			Since $ \psi=0$ on $[T_0,\infty)\times\bR$,  we deduce that $ \widetilde{\phi}=0$ on $[T_0,\infty)\times\bR$. Moreover, function $\phi\coloneqq  \widetilde{\phi}|_{K}$ belongs to $C_0^\infty((K))$.   Consequently, due to \eqref{eqn-de Bois Raymond-assumption}, we have
			\[\iint_{K}  u\psi\ud x\ud t = \iint_{K}  u \bigl( \partial_t^2  \widetilde{\phi} - \partial_x^2  \widetilde{\phi} \bigr) \ud x\ud t
			=\iint_{K}  u \bigl( \partial_t^2  \phi - \partial_x^2  \phi \bigr) \ud x\ud t=0.
			\]
			By arbitrariness of $\psi$ we infer that $u = 0$ Lebesgue  almost everywhere  in $K$ and   the proof of the  uniqueness part is complete.
			Lastly,  since from Theorem \ref{thm-mild solution is weak} a (locally) integrable function $u$ given by \eqref{eq:wave-sol-mild} is a weak solution,  the last claim of part (i) follows by uniqueness.
		\end{proof}
		
		\begin{proof}[Proof of item (ii)]	
			
			The part (ii) follows from the formula \eqref{eq:wave-sol-mild} and the uniqueness assertion from part (i).  
		\end{proof}
            Hence we complete the proof of Theorem \ref{thm-weak soluntions uniqueness}. 
	\end{proof}

	\begin{corollary}\label{cor-weak=mild sol}
		In the framework of Theorem  \ref{thm-weak soluntions uniqueness}  the notions of  weak and mild solutions to problem \eqref{eqn-linear-wave-nonhom} are equivalent.
	\end{corollary}

	We can now formulate the following classical apriori estimate result about the solution to problem
	\eqref{eqn-linear-wave-nonhom} on trapezoids.
	\begin{proposition} \label{prop-W^1,1 estimates}
		Assume that $K$ is a  trapezoid  of height $T \in (0,\infty]$ and base $K_0$.
		Let $h \in L^1_{\loc}(K;\bR^n)$ and $u_0, v_0 \in L^1_{\loc}(K_0;\bR^n)$ such that $D u_0 $ exists and is locally integrable.  Assume that a  locally  integrable function  $u:K \to \bR^n$ is a weak solution of \eqref{eqn-linear-wave-nonhom} on $K$.
		If  subinterval $\widehat{K}_0 \Subset K_0$, with $\widehat{K}_t$ and $\widehat{K}$ defined appropriately, is such that  $v_0, D u_0 \in L^1(K_0)$ and $h \in L^1(\widehat{K})$
		then  $u$ satisfies the following inequality,
		\begin{align}\label{eq:transport-bound-trapezoid}
			\nonumber \max & \Bigl\{
			\sup_{ t\in [0,T)] } \Bigl[  \int_{\widehat{K}_t}|\partial_x u(t, x)|\ud x\Bigr],\;\; \sup_{ t\in [0,T)] }  \Bigl[\int_{\widehat{K}_t}|\partial_t u(t, x)|\ud x \Bigr]\Bigr\} \\
			\qquad & \leq \int_{\widehat{K}_0}|D u_0(x)|\ud x + \int_{\widehat{K}_0}|v_0(x)|\ud x +\int_{\widehat{K}}|h(s, x)|\ud x\ud s.
		\end{align}
		
		In particular, if $v_0, D u_0 \in L^1(K_0)$ and $h \in L^1({K})$, then $u \in \mathscr{H}(K)$.
	\end{proposition}
	\begin{proof}[Proof of Proposition \ref{prop-W^1,1 estimates}]
		This is a classical result. One can also give another proof by  observing that $\partial_t u$, respectively $\partial_x u$,  satisfies  integral  equation \eqref{eqn-u_t}, respectively \eqref{eqn-u_x}, below  and then applying Lemma \ref{lem:transport} formulated in the next section.
	\end{proof}

	\subsection{Zhou's estimates}\label{sec:Zhou-trick}
	In \cite{Zhou99}, Zhou's proofs  are based on properties of solutions of transport equations, which can be solved in a manner
	somewhat similar to formula \eqref{eq:wave-sol-mild}.
	\begin{lemma}
		\label{lem:transport}
		Assume $g: \bR \to \bR^n$ and $f: [0, \infty)\times \bR \to \bR^n$ are locally integrable functions. Let a function $v : [0, \infty) \times \bR \to \bR^n$ be defined by
		\begin{equation}
			\label{eq:transport}
			v(t, x) = g(x - t) + \int_0^t f(\tau, x - t + \tau)\ud \tau,
		\end{equation}
		for all $t \geq 0$ and almost all (``almost'' depending on $t$) $x \in \bR$. \\
		Then for all   $x_0 \in \bR$, $L>0$ and $t_0  \in [0,L]$
		\begin{equation}
			\label{eq:transport-bound}
			\int_{x_0 - L + t_0}^{x_0 + L - t_0}|v(t_0, x)|\ud x \leq \int_{x_0 - L}^{x_0 + L}|g(x)|\ud x + \int_0^{t_0}\int_{x_0 - L + t}^{x_0 + L - t}|f(t, x)|\ud x\ud t.
		\end{equation}
		The same estimate holds if $v$ is defined by
		\begin{equation}
			\label{eq:transport-neg}
			v(t, x) = g(x + t) + \int_0^t f(\tau, x + t - \tau)\ud \tau.
		\end{equation}
	\end{lemma}
	We can easily extend the above  result to the case $L=\infty$ but we don't do it here because we never use such extension in the article.
	\begin{remark}\label{rem-transport equation}
		The functions $v$, defined in \eqref{eq:transport} and \eqref{eq:transport-neg}, respectively, are the unique weak/mild solution of the following linear evolution equation on the Banach space $L^1(\bR)$ or Fr\'echet space $L_{\loc}^1(\bR)$:
		\begin{align}\label{eqn-transport equation}
			\partial_t v +\partial_x v & = f(t,x),\;\;t>0,\;x \in \bR,
			\\
			v(0,x)&=g(x), \;\; x \in \bR,
			\label{eqn-transport equation-IC}
		\end{align}
		and
		\begin{align}\label{eqn-transport equation-neg}
			\partial_t v - \partial_x v & = f(t,x),\;\;t>0,\;x \in \bR,
			\\
			v(0,x)&=g(x), \;\; x \in \bR.
			\label{eqn-transport equation-neg-IC}
		\end{align}
		We ask the reader to refer \cite{Ball77} for more details.
	\end{remark}

	\begin{remark}\label{rem-wave equation}
		The  above mentioned transport equations appear in the context of wave equation in the following way.  Suppose that $u$ solves the linear non-homogeneous wave equation \eqref{eqn-linear-wave-nonhom} with $h=f$ and  initial data $(u_0,u_1)$.  Define two new functions
		\begin{align}
			\label{eqn-v+}
			v_+(t,x) & = \partial_t u(t,x)-\partial_x u(t,x), \\
			v_-(t,x) & = \partial_t u(t,x)+\partial_x u(t,x).
			\label{eqn-v-}
		\end{align}
		Since by definition, $u \in \mathscr{H}(K)$ for some trapezoid $K$, we infer that  $v_+$ and $v_-$ are  locally integrable and are  weak solutions to the transport equation  \eqref{eqn-transport equation}, respectively \eqref{eqn-transport equation-neg},  with initial conditions
		\begin{align*}
			v_+(0,x)&=u_1(x)-Du_0(x),
		\end{align*}
		respectively,
		\begin{align*}
			v_-(0,x)&=u_1(x)+Du_0(x).
		\end{align*}
		Similar assertion holds for function defined by wave map equation \eqref{eq:transport-neg}, see Remark \ref{rem-v+ and v- for wave maps}.
	\end{remark}

	\begin{remark}
		Note that by the Fubini Theorem, for all $t \geq 0$ equation \eqref{eq:transport} defines $v(t, \cdot) \in L^1_\tx{loc}(\bR)$.
		Again by the Fubini Theorem and the change of measure theorem, applied to the locally integrable function $F(t, x, \tau) \coloneqq  f(\tau, x - t + \tau)$,
		the function $v$ is locally integrable. We also have $v \in C([0, \infty); L^1_\tx{loc}(\bR))$,
		by the Lebesgue Dominated Convergence Theorem.
		
		We notice that $v$ is a weak solution of the transport equation $(\partial_t + \partial_x)v = f$ in $[0,\infty) \times \bR$
		if and only if there exists $g$ such that \eqref{eq:transport} holds.
		Similarly, $v$ is a weak solution of $(\partial_t - \partial_x)v = f$ in $[0,\infty) \times \bR$
		if and only if there exists $g$ such that \eqref{eq:transport-neg} holds. This  follows from the uniqueness of weak solutions to the transport equations.
	\end{remark}
	\begin{proof}[Proof of Lemma~\ref{lem:transport}]
		Without loss of generality, we can assume that $x_0 = 0$. Let us choose and fix $L>0$ and $t_0 \in [0,L]$. The formula \eqref{eq:transport} yields
		\begin{equation}
			\begin{aligned}
				\int_{- L + t_0}^{L - t_0}|v(t_0, x)|\ud x &= \int_{- L + t_0}^{L - t_0}\Big|g(x - t_0) + \int_0^{t_0} f(\tau, x - t_0 + \tau)\ud \tau\Big|\ud x \\
				& \leq \int_{- L}^{L - 2t_0}|g(x)|\ud x + \int_0^{t_0}\int_{- L + t_0}^{L - t_0}|f(\tau, x - t_0 + \tau)|\ud x\ud\tau \\
				&= \int_{- L}^{L - 2t_0}|g(x)|\ud x + \int_0^{t_0}\int_{- L + \tau}^{L - 2t_0 + \tau}|f(\tau, x)|\ud x\ud\tau.
				\label{eq:transport-bound-stronger}
			\end{aligned}
		\end{equation}
		Since $-L \leq L - 2t_0 \leq L$ and $-L+\tau \leq L -2t_0 +\tau \leq L -\tau$, we obtain \eqref{eq:transport-bound}.
		The proof of \eqref{eq:transport-neg}  is similar.
	\end{proof}

	\begin{lemma}
		\label{lem:zhou-lemma}
		Let $g_+, g_-: \bR \to \bR^n$ and $f_+, f_-: \bR_+ \times \bR\to\bR^n$ be locally integrable functions, and let $v_+$ and $v_-$ be given for $t \geq 0$ by, for almost all $x \in \bR$,
		\begin{equation}
			\label{eq:zhou-lemma-hyp}
			\begin{aligned}
				v_+(t, x) &= g_+(x - t) + \int_0^t f_+(\tau, x - t + \tau)\ud \tau, \\
				v_-(t, x) &= g_-(x + t) + \int_0^t f_-(\tau, x + t - \tau)\ud\tau.
			\end{aligned}
		\end{equation}
		Then for all $x_0 \in \bR$, $L>0$ and $t_0  \in [0,L]$,
		\begin{equation}
			\label{eq:zhou-lemma}
			\begin{aligned}
				& \int_0^{t_0}\int_{x_0 - L + t}^{x_0 + L - t}|v_+(t, x)||v_-(t, x)|\ud x\ud t  \\
				&\leq \frac{1}{2}
				\bigg(\int_{x_0 - L}^{x_0 + L}|g_+(x)|\ud x+\int_0^{t_0}\int_{x_0 - L + t}^{x_0 + L - t}|f_+(t, x)|\ud x\ud t\bigg)\times \\
				&\qquad\times\bigg(\int_{x_0 - L}^{x_0 + L}|g_-(x)|\ud x+\int_0^{t_0}\int_{x_0 - L + t}^{x_0 + L - t}|f_-(t, x)|\ud x\ud t\bigg).
			\end{aligned}
		\end{equation}
	\end{lemma}
	\begin{remark}
		In particular, the left hand side of \eqref{eq:zhou-lemma} is finite, which is not obvious apriori. We do not claim that \eqref{eq:zhou-lemma} holds when $|v_+(t, x)||v_-(t, x)|$ is replaced by $|v_+(t, x)|^2$.
	\end{remark}
	\begin{proof}[Proof of Lemma~\ref{lem:zhou-lemma}]
		First observe that, by triangle inequality we have
		\begin{equation}
			\label{eq:zhou-lemma-hyp-triangle}
			\begin{aligned}
				|v_+(t, x)| &\leq |g_+(x - t)| + \int_0^t |f_+(\tau, x - t + \tau)|\ud \tau, \\
				|v_-(t, x)| &\leq  |g_-(x + t)| + \int_0^t |f_-(\tau, x + t - \tau)|\ud\tau.
			\end{aligned}
		\end{equation}
		Thus it is enough to prove \eqref{eq:zhou-lemma} in the case when $n=1$ and all the functions $g_+,g_-, f_+$ and $f_-$ are being non-negative.
		
		Without loss of generality, let us choose and fix $x_0 = 0$.  Let us extend $g_+,g_-$ by $0$ for $|x|\geq L$ and $f_+,f_-$ by $0$ for $t<0$ or $t>t_0$ or $|x| > L-t$. Then observe that the product
		$$ v_+(t,x) v_-(t,x) =0 \textrm{ on }  ([0,\infty) \times \bR)\setminus \{(t,x): t \in [0,L], x \leq L-t, x \geq -L+t \}$$
		and it suffices to prove
		\begin{equation}
			\label{eq:zhou-lemma-R}
			\begin{aligned}
				\int_0^{t_0}\int_{\bR} & v_+(t, x)v_-(t, x)\ud x\ud t  \\
				&\leq \frac{1}{2}
				\bigg(\int_{\bR}g_+(x) \ud x+\int_0^{t_0}\int_{\bR}f_+(t, x)\ud x\ud t\bigg)\bigg(\int_{\bR}g_-(x)\ud x+\int_0^{t_0}\int_{\bR}f_-(t, x)\ud x\ud t\bigg).
			\end{aligned}
		\end{equation}
		By substituting the expression from \eqref{eq:zhou-lemma-hyp}, LHS in \eqref{eq:zhou-lemma-R} can be rewritten as
		\begin{align*}
			\int_0^{t_0}\int_{\bR} & v_+(t, x)v_-(t, x)\ud x\ud t = \int_0^{t_0}\int_{\bR}g_+(x-t)g_-(x+t)\ud x\ud t \\
			& \quad + \int_0^{t_0}\int_{\bR}g_+(x-t) \int_0^t f_-(\tau, x + t - \tau)\ud \tau  \ud x\ud t \\
			& \quad + \int_0^{t_0}\int_{\bR}g_-(x+t) \int_0^t f_+(\tau, x - t + \tau)\ud \tau  \ud x\ud t \\
			& \quad + \int_0^{t_0}\int_{\bR}\int_0^t \int_0^t f_+(\tau, x - t + \tau) f_-(\sigma, x + t - \sigma) \ud \tau \ud \sigma  \ud x\ud t \\
			& \quad =: I + II + III + IV.
		\end{align*}
		To estimate these four terms we will work with the null coordinates $(a,b)$ defined by
		\begin{equation}\label{eq:defn-null-coordinates}
			a \coloneqq x+t,\quad b = x-t.
		\end{equation}
		Using this transformation we have for $0 \leq s \leq t_0$ we have
		\begin{align}\label{eq:zhou-gg-est}
			\int_s^{t_0}\int_{\bR} & g_+(x-t)g_-(x+t)\ud x\ud t = \frac{1}{2} \int_{\bR} \int_{b+2s}^{b+2t_0} g_-(a)g_+(b) \ud a \ud b \nonumber\\
			& \leq \frac{1}{2} \left(\int_{\bR} g_-(a) \ud a \right)\left(\int_{\bR} g_+(b) \ud b \right) .
		\end{align}
		Now we estimate the most difficult term $IV$ and leave the terms $II$ and $III$ to a reader. Let us set $g_+^\tau(x-t)\coloneqq f_+(\tau, x - t + \tau) $ and $g_-^\sigma(x+t)\coloneqq f_-(\sigma, x + t -\sigma) $ and observe that by \eqref{eq:zhou-gg-est} we get
		\begin{align}
			& \int_0^{t_0}\int_{\bR}\int_0^t \int_0^t f_+(\tau, x - t + \tau) f_-(\sigma, x + t - \sigma) \ud \tau \ud \sigma  \ud x\ud t \nonumber\\
			&\qquad = \int_0^{t_0} \int_0^{t_0} \left[ \int_{\tau \vee \sigma}^{t_0} \int_{\bR}  g_+^\tau(x-t) g_-^\sigma(x+t) \ud x\ud t  \right] \ud \tau \ud \sigma   \nonumber \\
			&\qquad \leq \frac{1}{2} \int_0^{t_0} \int_0^{t_0} \left(\int_{\bR} g_-^\sigma(a) \ud a \right)\left(\int_{\bR} g_+^\tau(b) \ud b \right)   \ud \tau \ud \sigma   \nonumber \\
			&\qquad \leq \frac{1}{2} \left( \int_0^{t_0}\int_{\bR}f_-(\sigma, a - \sigma) \ud a \ud \sigma \right)\left( \int_0^{t_0}\int_{\bR}f_+(\tau, b + \tau) \ud b \ud \tau \right) .
		\end{align}
		Hence the result follows.
	\end{proof}
	
	\begin{definition}\label{def-Q form}
		Let $K$ be any trapezoid. If $u:K \to \bR^n$ and $ \sh u:K \to \bR^n$, then we define  \[Q(u,\sh u) \coloneqq   [Q_{j,k}(u,\sh u)]_{j,k=1}^n: K \to \bR^{n \times n},\] where
		\begin{equation}\label{eqn-Q}
			Q_{jk}(u, \sh u) \coloneqq  \partial_t u_j\partial_t \sh u_k - \partial_x u_j \partial_x \sh u_k, \;\; j,k=1,\ldots,n,
		\end{equation}
		and $\bR^{n\times n}$ is  the vector space of all real valued $n\times n$ matrices endowed with the $\ell^1$-norm.
	\end{definition}

	\begin{proposition}\label{prop:zhou-estimates}
		Let $h \in L^1_{\loc}(K;\bR^n)$ and $u_0, v_0, L^1_{\loc}(K_0;\bR^n)$. Assume that a locally integrable function  $u:K \to \bR^n$ is a mild solution of linear wave equation \eqref{eqn-linear-wave-nonhom} on $K$ such that  $D u_0 \in L_{\loc}^1(K_0;\bR^n)$.	
		Then $u \in \mathscr{H}(K)$ and
		$Q_{jk}(u, u) \in L_{\loc}^1(K;\bR^n)$, for all $ j,k=1,\ldots,n$,  and, for every $(t_0,x_0) \in K$ we have
		\begin{equation}\label{eqn-Q(u,u)-estimate}
			\begin{split}		
				\iint_{T_{(t_0,x_0)}} |Q(u, u)|\ud x\ud t \leq
				\bigg(\int_{T_0} |v_0(x)-Du_0(x)| \ud x + \iint_{T_{(t_0,x_0)}} |h|\ud x\ud t\bigg)
				\\
				 \times \bigg(\int_{T_0}|v_0(x)+Du_0(x)| \ud x + \iint_{T_{(t_0,x_0)}} |h|\ud x\ud t\bigg),
		\end{split}		\end{equation}
		where $T_{(t_0,x_0)} \subset K$ is the triangle with vertices $(t_0,x_0), (0,x_0-t_0)$ and $(0,x_0+t_0)$,  and  $T_0 =[x_0-t_0, x_0+t_0] $.
		
		More generally, if $\sh u$ is another  solution (with initial data $\sh u_0, \sh v_0: K_0 \to \bR^n$ and forcing $\sh h: K \to \bR^n$ satisfying the above assumptions), then
		\begin{equation}\label{eq:Qestimate}
			\begin{aligned}
				2 \iint_{T_{(t_0,x_0)}} |Q(u, \sh u)|\ud x\ud t &
				\leq\bigg(\int_{T_0}|v_0(x)-Du_0(x)| \ud x + \iint_{T_{(t_0,x_0)}} |h|\ud x\ud t\bigg)
				\\
				& \times \bigg(\int_{T_0} |\sh v_0(x)+D\sh u_0(x)| \ud x + \iint_{T_{(t_0,x_0)}} |\sh h|\ud x\ud t\bigg)\\
				& + \bigg(\int_{T_0}|v_0(x)+Du_0(x)| \ud x + \iint_{T_{(t_0,x_0)}} |h|\ud x\ud t\bigg)
				\\
				& \times \bigg(\int_{T_0} |\sh v_0(x)-D \sh u_0(x)| \ud x + \iint_{T_{(t_0,x_0)}} | \sh h|\ud x\ud t\bigg).
			\end{aligned}
		\end{equation}
	\end{proposition}
	\begin{proof}[Proof of Proposition \ref{prop:zhou-estimates}]
		Let us choose and fix a trapezoid $K$ and two locally integrable functions $u, \sh u:K \to \bR^n$ which are  mild solutions of linear wave equation \eqref{eqn-linear-wave-nonhom} on $K$ such that  $D u_0, D \sh u_0  \in L_{\loc}^1(K_0;\bR^n)$.	Let us take any arbitrary $(t_0,x_0) \in K$.  By Definition \ref{def-abstract trapezoid}, the closed triangle $T_{(t_0,x_0)}$ with vertices $(t_0,x_0), (0,x_0-t_0)$ and $(0,x_0+t_0)$ is a subset of $K$. Let us denote the base of $T_{(t_0,x_0)}$ by $T_0$.
		
		Then, from the representation formula \eqref{eq:wave-sol-mild} we infer that, on $T_{(t_0,x_0)}$, the following formulae hold in the weak sense,
		\begin{equation}\label{eqn-u_t}
			\begin{aligned}
				\partial_t u(t, x) &= \frac 12 (D u_0(x + t) - D u_0(x - t)) + \frac 12(v_0(x + t) + v_0(x - t)) \\
				&+ \frac 12 \int_0^t(h(\tau, x + t - \tau) + h(\tau, x - t + \tau))\ud\tau, \\
			\end{aligned}
		\end{equation}
		\begin{equation}\label{eqn-u_x}
			\begin{aligned}
				\partial_x u(t, x) &= \frac 12 (D u_0(x + t) + D u_0(x - t)) + \frac 12(v_0(x + t) - v_0(x - t)) \\
				&+ \frac 12 \int_0^t(h(\tau, x + t - \tau) - h(\tau, x - t + \tau))\ud\tau.
			\end{aligned}
		\end{equation}
Let us observe now, that the RHSs of the above two identities define respectively functions belonging to spaces
$C_b([0,T_0)]; L^1(K_\cdot))$, respectively, $C_b([0,T_0)]; L^1(K_\cdot))$. Hence we infer that $u \in \mathscr{H}(K)$.
	
Moreover,  form these last two identities we deduce that, in the weak sense,
		\begin{equation}
			\begin{aligned}
				\partial_t u(t, x) - \partial_x u(t, x) &= -D u_0(x - t) + v_0(x - t) + \int_0^t h(\tau, x - t + \tau)\ud\tau, \\
				\partial_t u(t, x) + \partial_x u(t, x) &= D u_0(x + t) + v_0(x + t) + \int_0^t h(\tau, x + t - \tau)\ud\tau.
			\end{aligned}
		\end{equation}
		In other words equations \eqref{eq:zhou-lemma-hyp} are satisfied with
		\begin{align*}
			v_{+} &\coloneqq  (\partial_t - \partial_x)u, \;\; v_{-} \coloneqq  (\partial_t + \partial_x)u,\\
			g_{+} &\coloneqq  - D u_0 + v_0, \;\;\; g_{-} \coloneqq  + D u_0 + v_0  \ \mbox{  and }\  f_+=f_- \coloneqq  h.
		\end{align*}
		Analogously, \eqref{eq:zhou-lemma-hyp} holds with
		\begin{align*}
			v_{+} &\coloneqq  (\partial_t - \partial_x)\sh u, \;\; v_{-} \coloneqq  (\partial_t + \partial_x)\sh u, \\
			g_{+} &\coloneqq  - D \sh u_0 + \sh v_0, \;\;\; g_{-} \coloneqq  + D \sh u_0 + \sh v_0  \ \mbox{  and }\  f_+=f_- \coloneqq  \sh h.
		\end{align*}
		Now we also choose and fix   $j,k \in \{1,\ldots,n\}$. Then we have
		\begin{equation}\label{eqn-rel-Q_jk}
			\begin{split}
				2Q_{jk}(u, \sh u) &= (\partial_t u_j - \partial_x u_j)(\partial_t \sh u_k + \partial_x \sh u_k) + (\partial_t u_j + \partial_x u_j)(\partial_t \sh u_k - \partial_x \sh u_k)
				\\ &=:  Q_{jk}^1(u, \sh u)+Q_{jk}^2(u, \sh u).
			\end{split}
		\end{equation}
		Then by applying the scalar  inequality \eqref{eq:zhou-lemma} with $t_0=L$ from Lemma \ref{lem:zhou-lemma}, we obtain
		\begin{align}
			\iint_{T_{(t_0,x_0)}} |Q^1(u, \sh u)| \ud x \ud t &= \sum_{i,j=1}^n \iint_{T_{(t_0,x_0)}} |Q_{jk}^1(u, \sh u)| \ud x \ud t
			\leq    \iint_{T_{(t_0,x_0)}}  \vert \partial_t u - \partial_x u \vert  \vert \partial_t \sh u + \partial_x \sh u \vert  \ud x \ud t
			\\ &\leq \bigg(\int_{T_0} |v_{0}(x)-Du_{0}(x)|  \ud x+\iint_{T_{(t_0,x_0)}}  |h|\ud x\ud t\bigg)  \\
			&\times \bigg(\int_{T_0} | \sh v_{0}(x) +D \sh u_{0}(x)| \ud x+ \iint_{T_{(t_0,x_0)}} |\sh h|\ud x\ud t\bigg).
		\end{align}
		Similarly, we get
		
		\begin{align}
			\iint_{T_{(t_0,x_0)}} |Q^2(u, \sh u)| \ud x \ud t &= \sum_{i,j=1}^n \iint_{T_{(t_0,x_0)}} |Q_{jk}^2(u, \sh u)| \ud x \ud t
			\leq    \iint_{T_{(t_0,x_0)}}  \vert \partial_t u + \partial_x u \vert  \vert \partial_t \sh u - \partial_x \sh u \vert  \ud x \ud t
			\\
			& \leq   \bigg(\int_{T_0} |v_{0}(x)+Du_{0}(x)|  \ud x+\iint_{T_{(t_0,x_0)}} |h|\ud x\ud t\bigg) \\
			& \times \bigg(\int_{T_0} | \sh v_{0}(x) -D \sh u_{0}(x)| \ud x+ \iint_{T_{(t_0,x_0)}} |\sh h|\ud x\ud t\bigg).
		\end{align}

		Consequently,   we have  \begin{align*}
			2  \iint_{T_{(t_0,x_0)}} |Q(u, \sh u)| \ud x \ud t &  \leq \bigg(\int_{T_0} |v_0(x)-Du_0(x)|  \ud x+\iint_{T_{(t_0,x_0)}} |h|\ud x\ud t\bigg) \\
			&\qquad \times \bigg(\int_{T_0} | \sh v_0(x) +D \sh u_0(x)| \ud x+ \iint_{T_{(t_0,x_0)}} |\sh h|\ud x\ud t\bigg) \\
			& +  \bigg(\int_{T_0} |v_0(x)+Du_0(x)|  \ud x+\iint_{T_{(t_0,x_0)}} |h|\ud x\ud t\bigg) \\
			&\qquad \times \bigg(\int_{T_0} | \sh v_0(x) -D \sh u_0(x)| \ud x+ \iint_{K} |\sh h|\ud x\ud t\bigg),
		\end{align*}
		and we complete the proof of Proposition \ref{prop:zhou-estimates}.
	\end{proof}
	
	\begin{remark}
		It is important to note that the results of Lemmata \ref{lem:transport}, \ref{lem:zhou-lemma} and Proposition \ref{prop:zhou-estimates} are valid for every norm,  e.g.  the Euclidean and the $\ell^1$,  on $\bR^n$.
	\end{remark}

	\section{Definitions and framework for wave map equation}\label{sec:framework}
	We consider  a compact  $m$-dimensional Riemannian manifold $\cM$. By the Nash embedding theorem \cite{Nash56} there exists an isometric embedding of $\cM$ into $\bR^n$ for some $n\in \bN$.
	In this note, we identify $\cM$ with the image of this embedding, so that from now on we assume that  $\cM \subset \bR^n$  and the Riemannian metric on $\cM$ is induced by the Euclidean  metric on $\bR^n$.
	
	We consider the following wave map equation with values in $\cM$ with a general forcing term
	\begin{equation}
		\label{eqn-wave map classical forced}
		\partial_t^2 u(t, x) - \partial_x^2 u(t, x) - f(t, x) \perp T_{u(t, x)}\cM, \qquad\text{for all }(t, x) \in \Omega,
	\end{equation}
	where $\Omega \subset \bR^{1+1}$ is an open set and $f: \Omega \to \bR^n$ is sufficiently regular for the equation above to make sense. A classical forced wave map is a sufficiently regular function $u: \Omega \to \cM$ such that \eqref{eqn-wave map classical forced} holds.
	
	The equation \eqref{eqn-wave map classical forced} can be written in the following explicit form
	\begin{equation}
		\label{eqn-wave map forced}
		\partial_t^2 u_i - \partial_x^2 u_i = \sum_{j,k=1}^n \Gamma_{ijk}(u)(\partial_t u_j\partial_t u_k - \partial_x u_j \partial_x u_k)
		+ \sum_{ j=1}^n P_{ij}(u)f_{j},\;\; 1 \leq i \leq n,
	\end{equation}
	for some   Lipschitz functions
	\begin{equation}
		\label{eqn-Gamma,P}
		\Gamma_{ijk}, P_{ij}: \cM \to \bR, \;\;\; 1 \leq i, j, k \leq n.
	\end{equation}
	We extend $\Gamma_{ijk}, P_{ij}$ as  bounded and  Lipschitz functions from $\bR^n$ to $\bR$.
	\begin{remark}\label{rem-v+ and v- for wave maps}
		If $u$ is a classical solution of equation \eqref{eqn-wave map classical forced} and  $\bR^n$-valued functions $v_+$ and $v_-$ are defined by formulae
		\eqref{eqn-v+} and \eqref{eqn-v-} respectively, then these functions satisfy
		\begin{align}
			\label{eqn- transport classical forced}
			\partial_t v_{\pm}(t, x) \pm \partial_x  v_{\pm}(t, x) - f(t, x) &\perp T_{u(t, x)}\cM, \qquad\text{for all }(t, x) \in \Omega.
		\end{align}
		
		Similarly, if $u$ is a classical solution of equation \eqref{eqn-wave map forced},  then the functions $v_+$ and $v_-$ satisfy the following non-linear transport equations,
		for $(t, x) \in \Omega$ and $1 \leq i \leq n$,
		\begin{align}\label{eqn-transport equation-nonlinear}
			\partial_t v_{\pm}^i& \pm \partial_x v_{\pm}^i=
			\sum_{1 \leq j,k\leq n}\Gamma_{ijk}(u) R_{jk}(v_+,v_-)
			+ \sum_{1\leq j\leq n}P_{ij}(u)f_{j},
		\end{align}
		where $R_{jk}$ is a bilinear map  defined by
		\begin{align}\label{eqn-R_jk}
			\bigl[ R_{jk}(v_+,v_-)\bigr](t, x)\coloneqq  \frac12 \bigl[ v_+^j(t, x)v_-^k(t, x)+ v_-^j(t, x)v_+^k(t, x) \bigr],\;\; (t, x) \in \Omega.
		\end{align}		
	\end{remark}
	
	Note that the system \eqref{eqn-wave map forced}  of $n$ equations can be equivalently written in the following vector form
	\begin{equation}\label{eq:fwm-vec}
		\partial_t^2 u - \partial_x^2 u = \sum_{j,k=1}^n\Gamma_{jk}(u)(\partial_t u_j\partial_t u_k - \partial_x u_j \partial_x u_k)
		+ P(u)f,
	\end{equation}
	with  Lipschitz   functions
	\begin{equation}\label{eqn-Gamma,P-vector}
		\Gamma_{jk}, P_{j}: \cM \to \bR^n, \;\;  j, k=1,\cdots, n, \quad \textrm{ and } \quad
		P: \cM \to \cL(\bR^n),
	\end{equation}
	where $\cL(\bR^n)$ is the set of all linear maps on $\bR^n$.

	Below  we summarize  the assumptions that we need to have. 

\begin{assumption}\label{assump-A1}
		The extensions of functions $\Gamma_{jk}$ and $P$  to $\bR^n$ are  bounded and globally Lipschitz functions  with values in  $\bR^n$ and $\cL(\bR^n)$ respectively.  We denote by $\gamma$ the supremum of all these functions. We denote by $L$ the common Lipschitz constant  of all these functions.\\
Let $\pi^\perp$ be the complimentary projection map of the map $\pi$ introduced in equation  \eqref{eqn-pi} of Lemma \ref{lem-projm-2}, i.e. 
for $p\in M$,    $\pi_p:\bR^n\to T_pM$ be the orthogonal projection.  Then the corresponding map
			\begin{equation}\label{eqn-pi^perp}
\pi_p^\perp=\id-\pi_p \in \cL(\bR^n). 
\end{equation}
By the same symbol we denote the extension of that map to a $C_0^\infty$-map defined on  the whole $\bR^n$.
	\end{assumption}

	In order to interpret equation \eqref{eq:fwm-vec} for functions $u$ which are not twice differentiable,
	we use the theory presented in Section~\ref{sec:linear-wave}.

	\begin{definition}[\textbf{Manifold valued weak solution}]
		\label{def-sol-weak}
		Let $K$ be  a trapezoid   with base $K_0$ with height $T \in (0,\infty]$.
		Let us assume that the initial data  $u_0 \in C(K_0;\cM)$ is an absolutely continuous function such that   $Du_0 \in L^1_{\loc}(K_0;\bR^n)$.
		Let us also assume  that $v_0 \in L^1_{\loc}(K_0;\bR^n)$ is such that
		\begin{align}\label{eqn-compatibility}
			v_0(x) \in T_{u_0(x)}\cM \mbox{  for almost all } x \in K_0.
		\end{align}
		Assume finally that $f\in L^1_{\loc}(K;\bR^n)$.
		
		A continuous function $u: K \to \cM$  is called a \emph{weak} solution of equation \eqref{eqn-wave map forced} on $K$,    with the initial condition
		\begin{align}\label{eqr-IC}
			u(0,\cdot)&=u_0, \;\;\;  \partial_t u(0,\cdot)=v_0
		\end{align}
		on $K_0$,  if and only if  the following conditions are satisfied.
		\begin{trivlist}
			\item[(i)]   the weak derivatives  $ \partial_x u, \partial_t u$ exist in the weak $L^1_{\loc}(K)$-sense;  in particular they are   measurable functions;
			\item[(ii)] for every  $t \in [0,T)]$,
                \begin{align}\label{eqn-compatibility-t}
			\partial_t u(t,x) \in T_{u(t,x)}\cM \mbox{  for almost every } x \in K_t.
		\end{align}
			\item[(iii)]  for every compact trapezoid $\widetilde{K} \subseteq K$ with base  $\widetilde{K}_0 $, the restriction of the right hand side of equation \eqref{eq:fwm-vec}  to  $ \widetilde{K}$ is integrable, and
			equation \eqref{eq:fwm-vec} holds in the weak sense on $\widetilde{K}$, i.e. if $t_0$ is the height of $\widetilde{K}$ then for every $T  \in (0,t_0)$ and  $\varphi \in  C_0^\infty((\widetilde{K}))$
			\begin{equation}\label{eqn-fwm-vector-in K-vector}
				\begin{split}
					&\int_{\widetilde{K}_T} u(T,x)\partial_t \varphi(T,x)\ud x - \int_{\widetilde{K}_T} \partial_t u(T,x)\varphi(T,x)\ud x
					\\		&+\iint_{\widetilde{K}}	u(\partial_t^2 \varphi - \partial_x^2 \varphi)\ud x\ud t =  \int_{\widetilde{K}_0} u_{0}(x)\partial_t \varphi(0,x)\ud x - \int_{\widetilde{K}_0} v_{0}(x)\varphi(0,x)\ud x\\
					&+  \iint_{\widetilde{K} } \bigl[ \sum_{j,k=1}^n \Gamma_{jk}(u)(\partial_t u_j\partial_t u_k - \partial_x u_j \partial_x u_k)
					+ \sum_{j=1}^nP_{j}(u)f_{j}\Bigr]  \varphi(t,x) \ud x\ud t .
			\end{split}\end{equation}
			
		\end{trivlist}
	\end{definition}

	Let us observe the following result.
	\begin{proposition}
		Condition \eqref{eqn-compatibility}, resp. \eqref{eqn-compatibility-t},   		 is equivalent to
		\begin{align}\label{eqn-compatibility-2}
			\Upsilon^M_{\varphi}(u_0,v_0) := \int_{K_0}  \pi^\perp_{u_0(x)}(v_0(x)) \;\;\varphi(x) \, dx=0 \mbox{  for every } \varphi \in C_0^\infty(\textrm{int}(K_0)).
		\end{align}
		respectively, to 
		\begin{align}\label{eqn-compatibility-t-2}
			\Upsilon^M_{\varphi}(u(t),\partial_t u(t)) := \int_{K_t}  \pi^\perp_{u(t,x)}(\partial_t u(t,x)) \;\;\varphi(x) \, dx=0 \mbox{  for every } \varphi \in C_0^\infty(\textrm{int}(K_0)).
		\end{align}
\end{proposition}
	
	\begin{remark}
		If $u$ is a $\cM$-valued classical  solution of our problem (for example if  $u$ belongs to  $C([0,T]; H_{\loc}^2(K_0;\cM))$ see e.g. \cite{BGOR22}),  then
		by employing the Stokes Theorem \cite[Theorem I.1.2]{Temam79}, one can show that $u$ is also a weak solution in the sense
		of Definition \ref{def-sol-weak}.
	\end{remark}

	\begin{remark}\label{rem-weak solution}
		Our definition of the weak solution is different than  \cite[Definition 1.1]{Zhou99} as it involves the ``final'' time  $T$. Hence apparently it is stronger. But we believe that both definitions are equivalent. We need this apparently stronger version of a definition of a weak solution because of the next result which we will use later on.
	\end{remark}

	\begin{definition}[\textbf{Manifold valued mild solution}]
		\label{def-sol-mild}
		Assume that	$K$,    $(u_0,v_0)$ and $f$ are as in   Definition \ref{def-sol-weak}.
		A continuous function $u: K \to \cM$   is called a \emph{mild} solution of equation \eqref{eq:fwm-vec} with the initial condition \eqref{eqr-IC}  if and only if
		the conditions (i)-(ii) from Definition \ref{def-sol-weak}  and the following condition (iv) hold.
		\begin{trivlist}
			\item[(iv)]
			equation \eqref{eq:fwm-vec} holds in the following mild (i.e. integral)   sense:  for every $(t,x) \in K$
			\begin{align}\label{eqn-fwm-mild-in K}
				u(t, x)& = \frac 12 \big(u_0(x + t) + u_0(x - t)\big) + \frac 12\int_{x-t}^{x+t}v_0(y)\ud y \nonumber\\
				& + \frac 12 \int_0^t\int_{x - t + \tau}^{x + t - \tau}  \bigg[\sum_{j,k=1}^n \Gamma_{jk}(u)(\partial_t u_j\partial_t u_k - \partial_x u_j \partial_x u_k)
				+ \sum_{j=1}^n P_{j}(u)f_{j} \bigg] \ud y\ud\tau.
			\end{align}
		\end{trivlist}
	\end{definition}
	
	\begin{remark}\label{rem-M-mild=weak}
		In view of Theorems \ref{thm-mild solution is weak} and \ref{thm-weak soluntions uniqueness}, a function $u : K \to M$ is a weak solution in the  sense of Definition \ref{def-sol-weak}  if and only if it is a mild solution in the sense of Definition \ref{def-sol-mild}.
	\end{remark}
	
	We finish this section with the following result which will be used later on.
	\begin{proposition}\label{prop-restriction-mildSoln}
		If a function $u: K \to \cM$  is a weak/mild solution of equation  \eqref{eq:fwm-vec} in $K$ with the initial data $(u_0,v_0)$ and external force $f$,
		and $\widetilde{K}$ is another trapezoid such that
		$\widetilde{K} \subset K$, then the restriction $\tilde{u}$ of $u$ to $\widetilde{K}$,
		is a weak/mild solution of equation  \eqref{eq:fwm-vec} in $\widetilde{K}$ with the initial data $(\tilde{u}_0,\tilde{v}_0)$ and external force $\tilde{f}$ defined appropriately.
	\end{proposition}
	\begin{proof}[Proof of Proposition \ref{prop-restriction-mildSoln}]
		The proof is obvious, see also   the proof  of the linear version, i.e. Proposition \ref{prop-weak solutions}.
	\end{proof}
	
	Before proving the existence of mild solutions to equation \eqref{eq:fwm-vec} with the initial condition \eqref{eqr-IC} which satisfy \eqref{eqn-compatibility}, in the next section we first develop the Cauchy theory for $\bR^n$-valued solution to the same problem.

	\section{Cauchy theory for $\bR^n$-valued solutions to wave map equation}\label{sec:Rn-valued-WE-Cauchy}
	In this section we develop a theory of $\bR^n$-valued solutions to equation \eqref{eq:fwm-vec} with suitable initial data $(u_0,v_0)$ and forcing $f$. The results of this section have nothing to do with  manifold $\cM$. In particular, we not assume that $u_0$ is $\cM$-valued,  neither that $(u_0,v_0)$ satisfy the compatibility condition \eqref{eqn-compatibility} nor that  the solution takes values in $\cM$.
	
	To shorten the notation and complexity in writing for interval $I \subseteq \bR$ we set
	\begin{equation}\label{eqn-space-L11}
		L^{1,1}(I;\bR^n) \coloneqq  (L^\infty(I;\bR^n) \cap \dot{W}^{1,1}(I;\bR^n)) \times L^1(I;\bR^n),
	\end{equation}
	and
	\begin{equation}\label{eqn-space-L11-norm}
		\| (f,g)\|_{L^{1,1}(I)}\coloneqq  \| f\|_{L^\infty(I)} + \| f' \|_{L^1(I)} + \| g\|_{L^1(I)}.
	\end{equation}
	
	First we give the definition of an $\bR^n$-valued mild solution to equation \eqref{eq:fwm-vec} with the initial condition \eqref{eqr-IC}.
	\begin{definition}[\textbf{$\bR^n$-valued mild solution}]
		\label{def-sol-mild-Rn}
		Let $K$ be  a trapezoid   with base $K_0$. Let us assume that the initial data  $(u_0,v_0) \in L^{1,1}(K_0;\bR^n)$ and forcing $f \in L^1(I;\bR^n)$.  A continuous function $u: K \to \bR^n$   is called a \emph{mild} solution of equation \eqref{eq:fwm-vec} with the initial condition \eqref{eqr-IC},  if and only if
		the condition (i)  from Definition \ref{def-sol-weak}  and $u$ satisfy  \eqref{eqn-fwm-mild-in K} for every $(t,x) \in K$.
	\end{definition}

	\subsection{Small data}\label{sec:small_data_cauchy}
	
	We begin our investigation with the small initial data and forcing global (in time) existence of $\bR^n$-valued solutions for equation \eqref{eq:fwm-vec}.
	\begin{remark}In the theorem below we assume that all the coefficients are  bounded and  Lipschitz. Had we assumed them to be only locally bounded and locally Lipschitz, we would have to replace the sentence ``Then there exists $\eta > 0$ such that the following is true.'' by 	``Then for any $\rho > 0$ there exists $\eta > 0$ such that the following is true.'' and add the following condition to conditions (\ref{eqn-smallness of initial condition}-\ref{bound-Forcing}):
		\begin{align}
			& \|u_0\|_{L^\infty} \leq \rho. \label{bound-InitPosn}
		\end{align}
		But such generalization will not be used in the present paper.
	\end{remark}
	
	\begin{theorem}[\textbf{$\bR^n$-valued global solution for small data on an arbitrary trapezoid}]
		\label{thm-global existence small data}
		Let $K$ be a trapezoid with base $K_0$ and let us assume that Assumption \ref{assump-A1} holds.
		Then there exists a number $\eta > 0$ such that the following is true.
		For every  initial data  $(u_0,v_0) \in L^{1,1}(K_0;\bR^n) $
		and for every external force $f \in L^1( K;\bR^n)$
		such that
		\begin{equation}\label{eqn-smallness of initial condition}
			\|D u_0 + v_0\|_{L^1(K_0)} \leq \eta,  \quad  \quad\|D u_0 - v_0\|_{L^1(K_0)} \leq \eta,
		\end{equation}
		
		and
		\begin{equation}\label{bound-Forcing}
			\| f\|_{L^1(K)} \leq \eta,
		\end{equation}
		there exists a function $u$ belonging to the Banach space $\mathscr{H}(K)$ introduced in  \eqref{def-space H(K)},
		such that $u$ is a mild solution, in the sense of Definition \ref{def-sol-mild-Rn}, to  equation \eqref{eq:fwm-vec} with  initial data $(u_0,v_0)$.
		\\
		Moreover, there exists a number $R>0$  such that the solution $u$ satisfying  the following condition
		\begin{equation}\label{eqn-u-space}
			\Vert \Gamma(u)(\partial_t u,\partial_t u) -  \Gamma(u) (\partial_x u, \partial_x u)
			+  P(u)f \Vert_{L^1(K)} \leq R,
		\end{equation}
		can be found  and this solution is  unique within the  class of elements of $ \mathscr{H}(K)$ satisfying  \eqref{eqn-u-space}.
		
		Finally, our problem is well posed in the Hadamard sense, i.e.  the map
		\begin{equation}\label{eqn-continuity}
			\mathscr{O}_{\eta}\ni (u_0,v_0,f) \mapsto u \in \mathscr{H}(K)
		\end{equation}
		is continuous, where $u$ satisfies \eqref{eqn-u-space} and $\mathscr{O}_{\eta}$ is a closed  subset of $L^{1,1}(K_0;\bR^n)  \times L^1( K; \bR^n)$ defined by
		conditions \eqref{eqn-smallness of initial condition} and \eqref{bound-Forcing}.
	\end{theorem}

	\begin{remark}\label{rem-global existence small data-uniqueness}
		The above result contains the following uniqueness result. If $u$ and $\sh u$ are two solutions to  equation \eqref{eq:fwm-vec} with  initial data $(u_0,v_0)$ and external force $f$ satisfying conditions \eqref{eqn-smallness of initial condition}-\ref{bound-Forcing}), then $u=\sh u$,  provided  condition \eqref{eqn-u-space} is satisfied by both $u $ and $\sh u$.
		
		In Theorem \ref{thm-unique-strongest} we will prove the following  stronger uniqueness result. If $u$ satisfies condition \eqref{eqn-u-space} and is  a solution to  equation \eqref{eq:fwm-vec} with  initial data $(u_0,v_0)$ and external force $f$ satisfying conditions \eqref{eqn-smallness of initial condition}-\ref{bound-Forcing}),  then $u=\sh u$,  provided $\sh u$ is a weak solution of equation \eqref{eq:fwm-vec} with the same initial data $(u_0,v_0)$
		external force $f$ such that
		\begin{equation}\label{eqn-u-space-tilde}
			\sh{u} \in L^\infty([0,\infty);L^\infty(K_\cdot) \cap \dot{W}^{1,1}(K_\cdot)) \mbox{ such that  } \partial_t \sh{u} \in L^\infty([0,\infty);L^1(K_\cdot)).
		\end{equation}
	\end{remark}

	\begin{remark}
		For large initial data, in general one can have blow-up in finite time.
		An example of an equation exhibiting this behaviour is $\partial_t^2 u - \partial_x^2 u = (\partial_t u)^2 - (\partial_x u)^2$.
		For wave maps with values in a compact Riemannian manifold, for initial data of finite energy there is no blow-up.
		Even for absolutely continuous initial data there is no blow-up, and in fact scattering holds.
		All this was proved by Keel and Tao in \cite{KT98}. Note however, that the uniqueness of solutions with finite energy was proved by Zhou in \cite{Zhou99}.
		We extend their findings to  \textbf{forced} wave maps with values in a compact Riemannian manifold.
	\end{remark}
	\begin{proof}[Proof of Theorem~\ref{thm-global existence small data}]
		We will use the Banach Fixed Point Theorem. Usually, one finds a suitable norm for the set of possible  solutions and proves that the natural map constructed from the equation in question, in such a way that its  fixed point is a solution,  turns out is a strict contraction. Here we follow a different path.
		We consider a suitable norm on the RHS of our equation \eqref{eq:fwm-vec} and construct a different  natural map from that space into itself so that its fixed point correspond to a solution of equation \eqref{eq:fwm-vec}. We also prove that this map is a strict contraction provided the initial data and the forcing satisfy conditions (\ref{eqn-smallness of initial condition}-\ref{bound-Forcing}) for suitably small positive  $\eta$.

		To pursue the idea, we choose and fix initial data $(u_0,v_0)\in L^{1,1}(K_0;\bR^n)$ and the external forcing $f\in L^1( \rK; \bR^n)$. We define a mapping
		\begin{equation}\label{eqn-Phi_u_0v_0f}
			\Phi=\Phi_{u_0,v_0,f} :L^1(\rK; \bR^n) \to L^1(\rK; \bR^n)
		\end{equation}
		as follows. For $h \in L^1(\rK;\bR^n)$ and the corresponding $u$ given by
		Theorem \ref{thm-weak soluntions uniqueness}, see formula \eqref{eq:wave-sol-mild},  we put
		\begin{align}
			\Phi(h)(t, x)  & = F(t, x, u, \partial_t u, \partial_x u) \\
			& \coloneqq  \sum_{jk}\Gamma_{jk}(u)(\partial_t u_j\partial_t u_k - \partial_x u_j \partial_x u_k)
			+ \sum_{j}P_{j}(u)f_{j}, \;\; (t, x) \in \rK.
			\label{eqn-Phi-def}
		\end{align}
		Note that once we get the unique fixed point $h$ of the map $\Phi$, a  mild  solution to \eqref{eq:fwm-vec} with the given initial data $(u_0,v_0)$ and forcing $f$ will be given by formula \eqref{eq:wave-sol-mild}.

		First, we show that the map $\Phi$ is well-defined. Let us choose and fix $h \in L^1(\rK;\bR^n)$.  Let $u$ be the  corresponding unique  mild  solution of \eqref{eqn-linear-wave-nonhom} in $\rK$, which is  given by formula \eqref{eq:wave-sol-mild}. Thus, $u_0,v_0,h$ and $u$ satisfy the assumptions of Proposition \ref{prop:zhou-estimates} and therefore, by applying inequality \eqref{eqn-Q(u,u)-estimate} from Proposition \ref{prop:zhou-estimates}, we get that
		\begin{align}\label{nonlin-scat-eq1}
			\| \Phi(h) \|_{L^1(\rK)}  & =  \iint_{K}  \bigg[\sum_{j,k=1}^n |\Gamma_{jk}(u) Q_{jk}(u, u)|
			+ \sum_{j=1}^n |P_{j}(u)f_{j}| \bigg] \ud x \ud t \nonumber\\
			& \quad \leq \gamma \bigg[\bigg(\int_{\bR}\big(|v_{0} -D u_{0}|\big)\ud x+ \iint_{K} |h|\ud x\ud t\bigg)
			\\
			& \quad \times \bigg(\int_{\bR}\big(|v_{0} + D u_{0}|\big)\ud x+\iint_{K} |h|\ud x\ud t\bigg)
			+ \iint_{K}\sum_{j=1}^n |f_{j}| \ud x\ud t \bigg],
		\end{align}
		where $\gamma$ is the bound of the extensions  of $\Gamma_{jk}, P_{j}$ from Assumption \ref{assump-A1}.
		Thus, we infer that  the map $\Phi$ satisfies  the following inequality,  for every suitable  data $(u_0,v_0;f)$,
		\begin{align}\label{nonlin-scat-eq2}
			\| \Phi(h) \|_{L^1(K)}
			& \leq \gamma \big[ ( \|v_0+D u_0\|_{L^1(K_0)} + \| h \|_{L^1(K)} )  ( \|v_0 - Du_0\|_{L^1(K_0)} \\
			& \quad + \| h \|_{L^1(K)} ) + \| f \|_{L^1(K)} \big]. \nonumber
		\end{align}
		Hence the map $\Phi$ is well-defined. It is important to note that we will use again   the computations from  \eqref{nonlin-scat-eq1} in the next paragraph, where we will prove the invariance of a closed ball in $L^1( K; \bR^n)$ of sufficiently small radius  under the map $\Phi$.%

		Let us set $\bB(0;R)\coloneqq  \bB_{L^1(K;\bR^n)}(0;R)$ as a closed ball of radius $R$, whose value will be set later, centered at the origin in $L^1( K;\bR^n)$. \\
		Assume that  the initial data and forcing satisfy conditions  (\ref{eqn-smallness of initial condition}-\ref{bound-Forcing}). Then for any $h \in \bB(0;R)$ by  inequality \eqref{nonlin-scat-eq2}  we have
		$$\| \Phi(h) \|_{L^1(K)}    \leq \gamma \big[ (  \eta + R )^2 + \eta \big] .$$
		Hence, if  $\eta>0$ and $R>0$ are such that
		\begin{equation}\label{eqn-eta and R}
			( \eta + R )^2 + \eta \leq   \frac{R}{\gamma},
		\end{equation}
		then  the ball $\bB(0;R)$ is invariant with respect to  the map $\Phi$.
		\begin{remark}\label{rem-eta and R exist}
			Let us first observe that it is sufficient to find $R, \eta$ such that
			\begin{equation}\label{nonlin-scat-rCond}
				2 \eta^2 +2R^2 +\eta \leq   \frac{R}{\gamma}.
			\end{equation}
			First we find $R>0$ small enough such that $4R^2 \leq   \frac{R}{\gamma}$. After that we find $\eta>0$ small enough such that $4\eta^2+2\eta \leq   \frac{R}{\gamma}$.
			We infer that the so found pair $(R,\eta)$
			satisfies condition \eqref{nonlin-scat-rCond} and therefore
			the function $\Phi$  maps the ball  $\bB(0;R)$ into itself.
		\end{remark}			
		As the next step we show that there exists $R>0$ and $\eta >0$ satisfying the condition \eqref{nonlin-scat-rCond} and such that
		the restriction of  map $\Phi$ on $\bB(0;R)$ is a  $\frac{1}{2}$-contraction. Let us recall that $f \in L^1(K;\bR^n)$ is given. Let us choose and fix  $h,g \in \bB(0;R)$  and let   $u$ and $v$ given by \eqref{eq:wave-sol-mild} with $h$ and $g$, respectively but with the same initial data, i.e.,
		\begin{equation}\label{nonlin-scat-initDataAssump}
			u_0 = v_0 \quad \textrm{ and } \quad   u_1 = v_1.
		\end{equation}
		Thus we observe that the Lipschitz and boundedness  properties of the extensions $\Gamma_{jk}$ and $P$ followed by inequality \eqref{eq:Qestimate} from Proposition \ref{prop:zhou-estimates} give
		\begin{align}\label{nonlin-scat-3}
			& \| \Phi(h) - \Phi(g)\|_{L^1( K)} \nonumber \\
			&\leq\iint_{K} \bigg|  \sum_{jk}   \{ \Gamma_{jk}(u) - \Gamma_{jk}(v) \}Q_{jk}(u, u)  \bigg| \ud x \ud t + \iint_{K} \bigg|  \sum_{jk}  \Gamma_{jk}(v)  Q_{jk}(u-v, u)   \bigg| \ud x \ud t
			\nonumber\\
			& \quad +  \iint_{K} \bigg|  \sum_{jk}  \Gamma_{jk}(v)  Q_{jk}(v, u-v)  \bigg| \ud x \ud t
			+ \iint_{K}\bigg|  \sum_{j} \{ P_{j}(u) - P_{j}(v) \} f_{j}  \bigg| \ud x \ud t \nonumber\\
			& \leq  L \iint_{K} \frac{1}{2} \bigg\vert \int_0^t \int_{x-t+\tau}^{x+t-\tau} h(\tau,y) -g(\tau,y) \ud y \ud \tau \bigg\vert \sum_{jk}   |Q_{jk}(u, u)|   \ud x \ud t \nonumber\\
			& + \gamma \iint_{K}  \sum_{jk}  |Q_{jk}(u-v, u) | \ud x \ud t   +  \gamma \iint_{K}  \sum_{jk} |Q_{jk}(v, u-v) | \ud x \ud t
			\nonumber \\
			& \quad + L \iint_{K} \frac{1}{2} \bigg\vert \int_0^t \int_{x-t+\tau}^{x+t-\tau} h(\tau,y) -g(\tau,y) \ud y \ud \tau \bigg\vert \sum_{j} |f_{j} | \ud x \ud t \nonumber\\
			& \leq  \frac{1}{2}\bigg[ 2 L \| h-g\|_{L^1(K)} \left\{ ( \| u_1 + Du_0 \|_{L^1(K_0)} +  \| h\|_{L^1(K)} )   ( \| u_1 - Du_0 \|_{L^1(K_0)} +  \| h\|_{L^1(K)} ) \right\}  \nonumber\\
			& \quad + \gamma  \|h-g\|_{L^1(K)} \left\{   \| u_1 + Du_0 \|_{L^1(K_0)} +  2 \| h\|_{L^1(K)}   +     \| u_1 - Du_0 \|_{L^1(K_0)}\right\}  \nonumber\\
			& \quad + \gamma  \|h-g\|_{L^1(K)} \left\{   \| v_1 - Dv_0 \|_{L^1(K_0)} +  2 \| g\|_{L^1(K)}   +     \| v_1 + Dv_0 \|_{L^1(K_0)}\right\}  \nonumber\\
			&  \quad + L \| h-g\|_{L^1(K)} \| f \|_{L^1(K)}  \bigg].
		\end{align}
		Hence,  if the initial data and forcing satisfy conditions (\ref{eqn-smallness of initial condition}-\ref{bound-Forcing}),  then  for any $h,g \in L^1(K; \bR^n)$
		\begin{align}
			\| \Phi(h) - \Phi(g)\|_{L^1( K)}
			& =  \| h-g\|_{L^1(K)}  [ L ( \eta + R )^2  + 5 \gamma \eta+4 \gamma R  ]. \nonumber
		\end{align}
		Thus, by choosing  numbers $\eta>0$  and $R>0$ sufficiently small such that \eqref{nonlin-scat-rCond} holds and
		\begin{equation}\label{nonlin-scat-r-eta-Cond}
			L ( \eta + R )^2  + 5 \gamma \eta+4 \gamma R    < \frac{1}{2},
		\end{equation}
		we infer that the map $\Phi_{|\bB(0;R)}$ maps the ball $\bB(0;R)$  into itself and moreover is a $\frac{1}{2}$-contraction. In view of the Banach Fixed Point Theorem, since the initial data and forcing $(u_0,v_0,f)$ satisfying conditions  (\ref{eqn-smallness of initial condition}) and (\ref{bound-Forcing}),  we can find a unique element    $h \in \bB(0;R)$  such that $\Phi_{(u_0,v_0,f)}(h)=h$.  As in the construction above of the map $\Phi_{(u_0,v_0,f)}$,
		let $u$ be the corresponding function given  by Theorem \ref{thm-weak soluntions uniqueness}. Because $h$  a fixed point of $\Phi_{u_0,v_0,f}$, by definition  \eqref{eqn-Phi-def} of $\Phi_{u_0,v_0,f}$, we infer that
		\begin{align*}
			h(t, x)  &  = \sum_{jk}\Gamma_{jk}(u)(\partial_t u_j\partial_t u_k - \partial_x u_j \partial_x u_k)
			+ \sum_{j}P_{j}(u)f_{j}, \;\; (t, x) \in \rK.
		\end{align*}
		Then, by Theorem \ref{thm-weak soluntions uniqueness}, in particular  formula \eqref{eq:wave-sol-mild}, we deduce that for every $(t, x) \in \rK$, $u(t,x)$ satisfies equality \eqref{eqn-fwm-mild-in K}.
		
		Hence, $u$  is a mild solution to problem \eqref{eq:fwm-vec}. Moreover, such a solution $u$  belongs to $\mathscr{H}(K)$ because of Proposition \ref{prop-W^1,1 estimates} or  Proposition \ref{prop:zhou-estimates}.   This completes the first part of the Theorem  \ref{thm-global existence small data}.

		\begin{remark}\label{rem-stronger statement}
			For the later reference, let us note that above in the we proved the following fact.
			Under the assumptions of Theorem \ref{thm-global existence small data}, there exists $\eta > 0$ and $R>0$ such that
			for every  $\bR^n$-valued initial data  $(u_0,v_0)$, defined on $K_0$ satisfying condition \eqref{eqn-smallness of initial condition} and every  $\bR^n$-valued  integrable function $f$ defined on $K$ satisfying  \eqref{bound-Forcing}, the map $\Phi_{(u_0,v_0,f)}$ defined in
			\eqref{eqn-Phi_u_0v_0f} is an $\frac12$-contraction.
		\end{remark}

		Let us note that in view of Proposition \ref{prop-W^1,1 estimates}  the  constructed   mild  solution   $u$  to problem   \eqref{eq:fwm-vec},   satisfies condition \eqref{eqn-u-space}. The uniqueness of such solution, i.e., which satisfy \eqref{eqn-u-space}, follows immediately as a consequence of  Banach fixed point theorem for the map $\Phi_{|\bB(0;R)}$.

		Now we move to the last part of this theorem, i.e., the continuity w.r.t. the initial data and the external forcing. For this aim let us consider  an $L^{1,1}(\oK;\bR^n)  \times L^1(\rK;\bR^n)$-valued  sequence $(u_{0,m},v_{0,m},f_m)_{m=1}^\infty $ and an element $(u_0,v_0,f) \in L^{1,1}(\oK;\bR^n)  \times L^1(\rK;\bR^n)$ such that
		$$(u_{0,m},v_{0,m}, f_m) \to (u_0,v_0, f) \textrm{ in } L^{1,1}(\oK;\bR^n)  \times L^1(\rK;\bR^n) \quad \textrm{ as } \quad m\to\infty.$$
		We further assume that $D u_{0,m}, D u_0, v_{0,m}, v_0 \in L^1(K_0;\bR^n)$ and $f_m,f \in L^1(K;\bR^n)$ are such that the bounds (\ref{eqn-smallness of initial condition}-\ref{bound-Forcing}) hold true for all $m$ with $\eta$ independent of $m$.

		Let the maps $\Phi_m, \Phi: L^1(K;\bR^n) \to L^1(K;\bR^n)$ defined, as \eqref{eqn-Phi_u_0v_0f}, based on $(u_{0,m},v_{0,m})$ and $(u_0,v_0)$, respectively,  and forcing $f_m$ and $f$, respectively.

		Thus, in view of the Remark \ref{rem-stronger statement}, we infer that there exists $R>0$,   a sequence $h_m \in L^1(K;\bR^n)$ and $h\in L^1(K;\bR^n)$ such that $\Phi_m(h_m) = h_m \in \bB(0;R)$ and $\Phi(h) = h \in \bB(0;R)$. We need to show  that $h_m \to h$ in $L^1(K;\bR^n)$ as $m \to \infty$.  But this will follow from Lemma \ref{lem-continuity} once we show that for every $h \in L^1(K; \bR^n)$,
		\begin{equation}\label{eqn-nonlin-scat-4}
			\| \Phi_m(h) - \Phi(h)\|_{L^1(K)}  \to 0 \quad \textrm{ as } \quad m \to \infty.
		\end{equation}
		\begin{proof}[Proof of \eqref{eqn-nonlin-scat-4}] Let us choose and fix $h \in L^1(K; \bR^n)$.  Recall that
			\begin{align}
				\Phi(h) & = \sum_{jk} \Gamma_{jk}(u) Q_{jk}(u, u) +  P(u) f, \nonumber
			\end{align}
			where $u$ is given by \eqref{eq:wave-sol-mild}.
			Similarly, we can define $\Phi_m(h)(t, x)$ and $u_m(t,x)$ by replacing $u_0, v_0,f $ by $u_{0,m}, v_{0,m}, f_m$, respectively.

			Consequently, by following the computation in \eqref{nonlin-scat-3}, we obtain
			\begin{align}
				& \| \Phi_m(h) - \Phi(h)\|_{L^1(K)}  \leq  L  \iint_{K}   \bigg| \frac{1}{2} \bigg\{ u_{0,m}(x + t)- u_{0}(x + t) + u_{0,m}(x - t) - u_{0}(x - t) \\
				&     +\int_{x - t + \tau}^{x + t - \tau}v_{0,m}(y) - v_0 (y) \ud y    \bigg\}  \bigg| \sum_{jk}  |Q_{jk}(u_{m}, u_{m})|   \ud x \ud t  \\
				&  + \gamma \iint_{K}  \sum_{jk}  |Q_{jk}(u_{m}-u, u_{m}) | \ud x \ud t +  \gamma \iint_{K}  \sum_{jk} |Q_{j,k}(u, u_{m}-u) | \ud x \ud t \\
				&  +L  \iint_{K} \bigg| \frac{1}{2} \Bigl\{ u_{0,m}(x + t)- u_{0}(x + t) + u_{0,m}(x - t) - u_{0}(x - t)
				\\
				& \hspace{2truecm} +\int_{x - t + \tau}^{x + t - \tau} \bigl(v_{0,m}(y) - v_0 (y)\bigr) \ud y    \Bigr\} \bigg| \\
				& \times \sum_{j} |f_{m,j} | \ud x \ud t  + \gamma \|f_m-f\|_{L^1}  \\
				& \leq  \frac{L}{2}  \{ 2 \| u_{0,m} - u_0 \|_{L^\infty(K_0)} + \| v_{0,m}-v_0\|_{L^1(K_0)} \}   \{ \| v_{0,m} - Du_{0,m} \|_{L^1(K_0)} + \| h\|_{L^1(K)} \}
				\\& \hspace{2truecm}\times
				\{ \|v_{0,m}+ D u_0 \|_{L^1(K_0)} + \| h\|_{L^1(K)} \}   \\
				&  +  \frac{\gamma}{2} \big[ \{\| v_{0,m} - v_0  - D (u_{0,m} -u_0) \|_{L^1(K_0)}  \} \{\| v_{0,m}  + D u_{0,m}  \|_{L^1} + \|h\|_{L^1(K_0)} \}    \\
				& \quad  + \{\| v_{0,m} - v_0  + D (u_{0,m} -u_0) \|_{L^1(K_0)}  \} \{\| v_{0,m}  - D u_{0,m}  \|_{L^1(K_0)} + \|h\|_{L^1(K)} \}  \big] \\
				&  +  \frac{\gamma}{2} \big[ \{\| v_0  - D u_0 \|_{L^1(K_0)} + \|h\|_{L^1(K)} \} \{\| v_{0,m} - v_0  + D (u_{0,m} -u_0) \|_{L^1(K_0)}  \}    \\
				& \quad  + \{\| v_0  + D u_0 \|_{L^1(K_0)} + \|h\|_{L^1(K)} \} \{\| v_{0,m} - v_0  - D (u_{0,m} -u_0) \|_{L^1(K_0)}  \}  \big] \\
				&  + \frac{L}{2}  \{ 2 \| u_{0,m} - u_0 \|_{L^\infty(K_0)} + \| v_{0,m}-v_0\|_{L^1(K_0)} \}   \| f_m \|_{L^1(K)} +  \gamma \|f_m-f\|_{L^1(K)}  .
			\end{align}
			Thus $\lim_{m \to \infty} \| \Phi_m(h) - \Phi(h)\|_{L^1(\rK)} = 0$, since $u_{0,m}\to u_0$ in $L^\infty(K_0;\bR^n) \cap \dot{W}^{1,1}(K_0;\bR^n)$, $v_{0,m} \to v_0$ in $L^1(K_0;\bR^n)$  and $f_m \to f$ in $L^1(\rK;\bR^n)$  as $m \to \infty$.
			
			Hence,  $h_m \to h$ in $L^1(K; \bR^n)$ as $m \to \infty$ due to Lemma \ref{lem-continuity}.  This proves \eqref{eqn-nonlin-scat-4}.
		\end{proof}
		Consequently, from \eqref{eqn-nonlin-scat-4}, due to \eqref{eq:wave-sol-mild},  we deduce that  $u_n \to u$ in $L^\infty(\rK;\bR^n)$. The proof of Theorem~\ref{thm-global existence small data} is complete.
	\end{proof}

	Let us now formulate the following  uniqueness result which is stronger than the uniqueness part of Theorem \ref{thm-global existence small data} because the smallness assumptions \eqref{eqn-smallness of initial condition}-\ref{bound-Forcing}) are
	not required.
	
	\begin{theorem}\label{thm-unique-strongest}
		Assume that $T_0 \in (0,\infty]$. Assume that $(u_0,v_0) \in L^{1,1}(\bR;\bR^n)$ and  $f \in L^1( [0,T_0)] \times \bR;\bR^n)$.
		Assume that  $u, \sh{u} \in \mathscr{H}([0,T_0)] \times \bR)$ satisfying
		\begin{equation}\label{eqn-u-space-2}
			\begin{split}
				&\Vert \Gamma(u)(\partial_t u,\partial_t u) -  \Gamma(u) (\partial_x u, \partial_x u)
				+  P(u)f \Vert_{L^1([0,T_0)] \times \bR)} <\infty,
				\\
				&\Vert \Gamma(\sh{u})(\partial_t \sh{u}\partial_t \sh{u} - \partial_x \sh{u} \partial_x \sh{u})
				+  P(\sh{u})f \Vert_{L^1([0,T_0)] \times \bR)} <\infty,
			\end{split}
		\end{equation}
		are   mild solutions to  equation \eqref{eq:fwm-vec} with  the same initial data $(u_0,v_0)$ and the same external force $f$. Then $u= \sh u$.
	\end{theorem}
	\begin{proof}[Proof of Theorem \ref{thm-unique-strongest}]
		The  proof  follows the lines of the proof of Theorem \ref{thm-uniqueness-wave-map}. Note that both Theorems \ref{thm-unique-strongest} and \ref{thm-uniqueness-wave-map} are concerned with the
		uniqueness result without the existence. But the proof of the latter result depends on the existence of solutions for small data as formulated in Theorem \ref{thm-unique-strongest}. In a similar fashion,
		the omitted  proof of the present Theorem \ref{thm-unique-strongest} depends on the existence of solutions for small data as formulated in Theorem \ref{thm-global existence small data}.	
	\end{proof}

	The following result is essential in our proof of  the finite speed of propagation for large data  formulated later in Corollary \ref{cor-finiteSOP}.
	
	\begin{corollary}\label{cor-finite speed of propagation}
		Let $K$ be a trapezoid with base $K_0$ and assume that Assumption \ref{assump-A1} holds.
		Assume that numbers $\eta>0$ and $R>0$  are as in Theorem \ref{thm-global existence small data}.
		Let us assume that  $f:K \to \bR^n$ is an integrable function, that  $(u_0,v_0):K_0 \to \bR^n \times \bR^n$  satisfy  the  smallness conditions
		\eqref{eqn-smallness of initial condition}-\ref{bound-Forcing}).
		Let $\widetilde{K} \subset K$ is another trapezoid, and $(\tilde{u}_0,\tilde{v}_0,\tilde{f})$ be another choice of the initial data  and the forcing which coincides on $\widetilde{K}_0$ and $\widetilde{K}$, respectively, with the initial data and the forcing $(u_0,v_0,f)$.  Then the corresponding  mild  solutions $u$ and $\tilde{u}$ to problem \eqref{eq:fwm-vec}, which exist by Theorem \ref{thm-global existence small data}, are equal on $\widetilde{K}$.
	\end{corollary}
	\begin{proof}[Proof of Corollary \ref{cor-finite speed of propagation}] Let us assume the assumptions of the Corollary.  It is obvious  that   $(\tilde{u}_0,\tilde{v}_0,\tilde{f})$ also satisfy the
		smallness conditions  \eqref{eqn-smallness of initial condition}-\ref{bound-Forcing}).

		Let $\Phi:L^1(K) \to L^1(K)$ as in the proof of Theorem \ref{thm-global existence small data} for the trapezoid $K$. Let  $\widetilde{\Phi}:L^1(\widetilde{K}) \to L^1(\widetilde{K})$ be the corresponding map for the trapezoid $\widetilde{K}$.
		
		The following important relationship between the maps $\Phi$
		and 			$\widetilde{\Phi}$ play a crucial role in our analysis.
		\begin{lemma}\label{lem-Phi-Phi`}
			If $h \in L^1(K;\bR^n)$, then
			\[
			\Phi(h)_{|\widetilde{K}}=\widetilde{\Phi}(h_{|\widetilde{K}}) \mbox{ in } L^1(\widetilde{K};\bR^n).
			\]
		\end{lemma}
		\begin{proof}[Proof of Lemma \ref{lem-Phi-Phi`}]
			Let us begin by recalling the definition of the map $\Phi:L^1(K) \to L^1(K)$ for fixed data $(u_0,v_0,f)$.
			For $h \in L^1(K)$ and the corresponding $u\in \mathscr{H}(K)$, which is given by Theorem \ref{thm-weak soluntions uniqueness} and Proposition \ref{prop-W^1,1 estimates},  we put
			\begin{align*}
				\Phi(h)(t, x)
				& = \sum_{jk}\Gamma_{jk}(u)(\partial_t u_j\partial_t u_k - \partial_x u_j \partial_x u_k)
				+ \sum_{j}P_{j}(u)f_{j}.
			\end{align*}
			The definition of $\widetilde{\Phi}$ is analogous but with data $(\tilde{u}_0,\tilde{v}_0,\tilde{f})$. \\
			Let us choose and fix  $h \in L^1(K)$ and put $\tilde{h}=f_{|\widetilde{K}} \in L^1(\widetilde{K})$. Since the restrictions of the functions
			$(u_0,v_0,f)$ to $\widetilde{K}_0 $ and $\widetilde{K}$ as appropriate, by the second part of
			Theorem \ref{thm-weak soluntions uniqueness} we infer that the corresponding solutions $u$ and $\tilde{u}$ are such that
			\[u_{|\widetilde{K}}=\tilde{u}.\]
			Hence, we infer that the restriction of the function $\sum_{jk}\Gamma_{jk}(u)(\partial_t u_j\partial_t u_k - \partial_x u_j \partial_x u_k)
			+ \sum_{j}P_{j}(u)f_{j}$, defined on $K$, to the set $\widetilde{K}$ is equal to
			$\sum_{jk}\Gamma_{jk}(\tilde{u})(\partial_t u_j^\prime\partial_t u_k^\prime - \partial_x u_j^\prime \partial_x u_k^\prime)
			+ \sum_{j}P_{j}(\tilde{u})f_{j}^\prime$. This implies that
			\[
			\Phi(h)_{|\widetilde{K}}=\widetilde{\Phi}(\tilde{h}).
			\]
			The proof of the Lemma \ref{lem-Phi-Phi`} is complete.
		\end{proof}
		Let $h$, respectively $\tilde{h}$ be the unique fixed point of the  map $\Phi$, respectively
		$\widetilde{\Phi}$. 
		By Lemma \ref{lem-Phi-Phi`} we infer easily that $h_{|\widetilde{K}} \in L^1(\widetilde{K})$ is also a fixed point of $\widetilde{\Phi}$. By the uniqueness of the fixed point we infer that $h_{|\widetilde{K}}=\tilde{h}$. Hence, by the second part of
		Theorem \ref{thm-weak soluntions uniqueness} we infer that the corresponding solutions $u$ and $\tilde{u}$
		to problem \eqref{eq:fwm-vec} satisfy $u_{|\widetilde{K}}= \tilde{u}$.
		The proof of Corollary \ref{cor-finite speed of propagation} is complete.		
	\end{proof}

	\subsection{Large data}\label{sec:big_data_cauchy}
	In the whole subsection we  assume that Assumption \ref{assump-A1} holds. This subsection is about the local existence and uniqueness of $\bR^n$-valued solutions to problem \eqref{eq:fwm-vec}  with arbitrary initial data and arbitrary external forcing.  We first formulate (and prove) the result for bounded trapezoids (in the sense of Definition \ref{def-trapezoid} together with Remark \ref{rem-max-trapezoid}).
	
	\begin{theorem}[\textbf{$\bR^n$-valued local solution for large data on a compact trapezoid}]
		\label{thm-nonlin-cauchy-bigData-bndTrap}
		Assume that $L >0$ and $\widetilde{K}$ is the compact trapezoid  of height $L$ and base $K_0$ of half-length $L$. Assume that the  initial data $(u_0, v_0) \in L^{1,1}(K_0;\bR^n)$ and the forcing $f \in L^1(\widetilde{K};\bR^n)$. Then  there exists $T_0  \in (0,L]$ such that problem \eqref{eq:fwm-vec} has a unique  mild  solution in class $\mathscr{H}(K)$ where $K$ is the compact trapezoid  base $K_0$ and height $T_0$.		
	\end{theorem}

	\begin{proof}[Proof of Theorem \ref{thm-nonlin-cauchy-bigData-bndTrap}] Let us  take $\eta>0$ and $R>0$ as in Theorem \ref{thm-global existence small data}.		
		Let us choose and fix a closed trapezoid $\tilde{K}$ and function $(u_0, v_0)$ and $f$ as in the statement of Theorem \ref{thm-nonlin-cauchy-bigData-bndTrap}. \\
		By \cite[Theorem 6.11]{Rudin74RCA}   we infer that there exist
		$\delta_1>0$ such that for every $x_0 \in K_0$,
		\begin{equation}\label{eqn-delta_1}
			\Vert D u_0+v_0\Vert_{L^1(K_0 \cap [x_0-\delta_1,x_0+\delta_1] )}  \leq \eta, \quad \Vert D u_0- v_0\Vert_{L^1(K_0  \cap [x_0-\delta_1,x_0+\delta_1])} \leq \eta.
		\end{equation}
		By the Lebesgue Dominated Convergence Theorem we can find $\delta_2>0$ such that
		\begin{equation}\label{eqn-delta_2}
			\Vert f\Vert_{L^1(\tilde{K} \cap( [0,\delta_2]\times K_0 )} \leq \eta.
		\end{equation}
		Set $\delta=\delta_1 \wedge \frac{2\sqrt{3}}{3}\delta_2 \wedge \frac{L}2$ and note that the conditions \eqref{eqn-delta_1} and \eqref{eqn-delta_2} are satisfied with $\delta$. Put temporarily $K_0=[a,b]$ and $K_0^\prime=[a+\delta,b-\delta]$. Since $\delta \leq \frac{L}2$, we infer that  $K_0^\prime\not= \emptyset$.
		Let us put
		\begin{equation}\label{eqn-K}
			K\coloneqq  \widetilde{K} \cap ([0,\delta]\times K_0)= \bigcup_{x_0 \in K_0^\prime } K_{x_0},
		\end{equation}
		where $K_{x_0}$ is a closed trapezoid  with base $ [x_0-\delta,x_0+\delta] $,  and height $\delta$. Note that $ [x_0-\delta,x_0+\delta] \subset K_0$, as $x_0 \in K_0^\prime$ and   $K$ is a closed trapezoid  with base $K_0$ and height $\delta$.
		Thus, due to  Theorem \ref{thm-global existence small data}, on each $K_{x_0}$ there exist a unique  mild  solution in class $\mathscr{H}(K_{x_0})$, denoted here by  $u_{x_0}$, to  problem \eqref{eq:fwm-vec} with data $(u_0|_{[x_0-\delta,x_0+\delta]},v_0|_{[x_0-\delta,x_0+\delta]}, f|_{K_{x_0}})$. Note that, due to  Corollary \ref{cor-finite speed of propagation},  $u_{x_1}=u_{x_2}$ in trapezoid $K_{x_1} \cap K_{x_2}$, for $x_1,x_2 \in K_0^\prime$.  This allow us to define a function $u$ by gluing together functions $u_{x_0}$, i.e.
		\begin{equation}
			u(t,x)=u_{x_0}(t,x) \mbox{ for }(t,x) \in K_{x_0}.
		\end{equation}
		Then, by employing the partition of unity we can show that  $u$ is the required mild solution of \eqref{eq:fwm-vec} in class $\mathscr{H}(K)$.
		Moreover, $u$ is unique by Theorem \ref{thm-unique-strongest}. Hence the theorem follows with $T_0=\delta$.
	\end{proof}

	Next we formulate a generalization of the previous local existence result to the closed and unbounded trapezoids.
	\begin{theorem}[\textbf{$\bR^n$-valued local solution for large data on an unbounded trapezoid}]
		\label{thm-nonlin-cauchy-bigData-UnBndTrap}
		Assume that $K$ is a closed and unbounded trapezoid with height $L \in (0,\infty]$ and base $K_0$. Assume that the  initial data $(u_0, v_0) \in L^{1,1}(K_0;\bR^n)$ and the forcing $f \in L^1(K; \bR^n)$. Then  there exists a $T_0  \in (0,L)]$ such that \eqref{eq:fwm-vec} has a unique  mild  solution in class $\mathscr{H}(\widetilde{K})$, where $\widetilde{K} := ([0,T_0] \times K_0) \cap K$.
	\end{theorem}
	
	\begin{proof}[Proof of Theorem \ref{thm-nonlin-cauchy-bigData-UnBndTrap}]	Let us  take $\eta>0$ and $R>0$ as in Theorem \ref{thm-global existence small data}. Let us choose and fix and functions $(u_0, v_0)$ and $f$   as in the statement of Theorem \ref{thm-nonlin-cauchy-bigData-UnBndTrap}. \\
		We begin by observing that there exists a bounded interval $[a,b] \subset K_0$ and $\bar{L} \in (0,L)$ such that $a+\bar{L} < b-\bar{L}$ and
		\begin{equation}\label{eqn-[a,b]-rays-small-data}
			\Vert D u_0+v_0\Vert_{L^1(K_{a,b}^c )}  \leq \eta  \mbox{ and } \Vert D u_0- v_0\Vert_{L^1(K_{a,b}^c)} \leq \eta,
		\end{equation}
		where
		$$K_{a,b}^c \coloneqq  K_0 \cap \left( (-\infty,a+ \bar{L}] \cup [b-\bar{L},\infty) \right).$$
		Moreover, w.l.o.g., we can choose $\bar{L}$ such that together with \eqref{eqn-[a,b]-rays-small-data} the following estimate  hold
		\begin{equation}\label{eqn-[a,b]-rays-small-f}
			\Vert f\Vert_{L^1(K^+_{b,\bar{L}} )}  \leq \eta  \mbox{ and } \Vert f \Vert_{L^1(K^-_{a,\bar{L}})} \leq \eta,
		\end{equation}
		where $K^+_{b,\bar{L}}\coloneqq  K^+([b-\bar{L},\infty) \cap K_0, \bar{L})$ and $K^-_{a,\bar{L}}\coloneqq K^-((-\infty,a+\bar{L}] \cap K_0, \bar{L})$ are semi-bounded trapezoids, see Definition \ref{def-trapezoid semi bounded},  with  height $\bar{L}$ and bases $[b-\bar{L},\infty) \cap K_0$ and $(-\infty, a+\bar{L}] \cap K_0$, respectively.

		Let $K_{a,b,\bar{L}}$ be a compact trapezoid with base $[a,b]$  and height $\bar{L}$.  By Theorem \ref{thm-nonlin-cauchy-bigData-bndTrap} there exists $T_1 \in (0,\bar{L}]$ such that equation \eqref{eq:fwm-vec} has a unique  mild  solution $u^0$, with initial data and forcing $(u_0|_{[a,b]},v_0|_{[a,b]},f|_{K_{a,b,\bar{L}}})$,  in the compact trapezoid $K_{a,b,T_1}$ with base $[a,b]$ and height $T_1$.
		
		Next, due to smallness condition \eqref{eqn-[a,b]-rays-small-data} and \eqref{eqn-[a,b]-rays-small-f}, by the global existence result for small data, see Theorem \ref{thm-global existence small data}, we can find the unique  mild  solutions $u^+$ and $u^-$  for problem \eqref{eq:fwm-vec} on the  semi-bounded trapezoids  $K^+_{b,\bar{L}}$ and $K^-_{a,\bar{L}}$, respectively, with initial data $(u_0|_{[b-\bar{L},\infty) \cap K_0}, v_0|_{[b-\bar{L},\infty) \cap K_0})$ and $(u_0|_{(-\infty,a+\bar{L}] \cap K_0}, v_0|_{(-\infty,a+\bar{L}] \cap K_0})$, and forcing $f|_{K^+_{b,\bar{L}}}$ and $f_{K^-_{a,\bar{L}}}$, respectively.
		By gluing the solutions $u^0$, $u^+$ and $u^-$ we construct a function  $u \in $ defined in the trapezoid with base $K_0$ and height $T_1$, i.e., on $\widetilde{K}\coloneqq  ([0,T_1] \times K_0) \cap K$.
		
		From here we conclude as in the proof of Theorem  \ref{thm-nonlin-cauchy-bigData-bndTrap} and get that  $u \in \mathscr{H}(\widetilde{K})$ is the required unique mild solution of \eqref{eq:fwm-vec} on the trapezoid $\widetilde{K}$.
	\end{proof}

	\subsection{Scattering result}
	As it is well-known, we don't expect that the forced wave map equation \eqref{eq:fwm-vec} have a $\bR^n$-valued global solution theory in the framework of  Theorems \ref{thm-nonlin-cauchy-bigData-bndTrap} and \ref{thm-nonlin-cauchy-bigData-UnBndTrap}. However, our next result shows that if such a global solution exists then the scattering of solutions hold true.
	\begin{theorem}[\textbf{Scattering for $\bR^n$-valued solution}]
		\label{thm:nonlin-scattering}
		Assume that $K=[0,\infty) \times \bR$ is a given trapezoid and $u \in \mathscr{H}(K)$ be a unique  mild  solution to  \eqref{eq:fwm-vec} with initial data $(u_0,v_0) \in L^{1,1}(\bR;\bR^n)$ and forcing $f \in L^1(K;\bR^n)$. Then, there exist $(\bar{u}_0,\bar{v}_0) \in L^{1,1}(\bR; \bR^n)$ such that the corresponding solution $u\lin$ of the linear homogeneous wave equation
		\eqref{eqn-linear-wave-nonhom} with $h=0$ and initial data $(\bar{u}_0,\bar{v}_0)$ satisfies
		\begin{equation}\label{nonlin-scattering}
			\lim_{t \to \infty}\big(\|u(t) - u\lin(t)\|_{L^\infty(\bR)} + \|\partial_x u(t) - \partial_x u\lin(t)\|_{L^1(\bR)}
			+ \|\partial_t u(t) - \partial_t u\lin(t)\|_{L^1(\bR)}\big) = 0.
		\end{equation}
	\end{theorem}
	\begin{proof}[Proof of Theorem \ref{thm:nonlin-scattering}]
		Let us  choose and fix initial data $(u_0,v_0) \in L^{1,1}(\bR;\bR^n) $.  Let us set, for $t \in [0,\infty)$,
		\begin{equation}\label{eqn-S(t)}
			(S(t)(u_0,v_0))(x) \coloneqq  \frac{1}{2} (u_0(x+t) + u_0(x-t)) + \frac{1}{2} \int_{x-t}^{x+t} v_0(y) \ud y,\;\; x\in \bR.
		\end{equation}
		Denote $\tilde{S}(t) \coloneqq  (S(t), \partial_t S(t))$ the linear group on Banach space $L^{1,1}(\bR;\bR^n)$ associated with the \eqref{eqn-linear-wave-nonhom} with $h=0$, that is the operator associating to given initial data the solution of the free wave equation and its time derivative  evaluated at time $t$.
		
		Let $u\in  L^\infty([0,\infty);L^\infty \cap \dot{W}^{1,1})$ be the  unique  mild  solution to  \eqref{eq:fwm-vec} with  data $(u_0,v_0)$ and such that $\partial_t u \in L^\infty([0,\infty);L^1(K_\cdot))$.
		
		Then by  the Duhamel formula, we have  the following identity
		\begin{equation}\label{duhamel}
			(u(t), \partial_t u(t)) = \tilde{S}(t) (u_0,v_0) + \int_{0}^{t} \tilde{S}(t-\tau) (0, f(\tau)) \ud\tau.
		\end{equation}
		By  the formula \eqref{eqn-S(t)} defining $S(t)$ and the assumption that $(u_0,v_0) \in L^{1,1}(\bR;\bR^n)$ we deduce that, for each $t$,
		\begin{equation}
			\label{est-duhamel-L^infty}
			\| S(t) (u_0,v_0)\|_{L^\infty(\bR)} \leq \| u_0 \|_{L^\infty(\bR)} + \frac{1}{2} \| v_0 \|_{L^1(\bR)}.
		\end{equation}
		Moreover, since $D u_0 \in L^1(\bR; \bR^n)$ and $v_0 \in L^1(\bR; \bR^n)$, we have
		\begin{equation}
			\label{est-duhamel-L^1}
			\begin{aligned}
				& \| \partial_x (S(t) (u_0,v_0))\|_{L^1(\bR)} \leq \| D u_0 \|_{L^1(\bR)} + \| v_0 \|_{L^1(\bR)}, \\
				& \| \partial_t (S(t) (u_0,v_0))\|_{L^1(\bR)} \leq \| D u_0 \|_{L^1(\bR)} + \| v_0 \|_{L^1(\bR)}.
			\end{aligned}
		\end{equation}		
		So, for every $t$, $\tilde{S}(t)$ is a bounded linear operator in  $L^{1,1}(\bR)$. By standard methods we can show that  for every $t$, $\tilde{S}(t)$ is an isomorphism of the Banach space  $L^{1,1}(\bR)$. In particular, $\tilde{S}(-t):=[\tilde{S}(t)]^{-1}$ exists.

		Next, applying $\tilde{S}(-t)$  to both sides of \eqref{duhamel}, we obtain
		\begin{equation}\label{duhamel-negTimes}
			\tilde{S}(-t)(u(t), \partial_t u(t)) =(u_0,v_0) + \int_{0}^{t} \tilde{S}(-\tau) (0, f(\tau)) \ud\tau.
		\end{equation}
		By the above bounds \eqref{est-duhamel-L^infty}-\eqref{est-duhamel-L^1} we infer that
		\begin{equation}
			\| \tilde{S}(-\tau) (0, f(\tau)) \|_{L^{1,1}(\bR)} \lesssim \| f(\tau)\|_{L^1(\bR)}.
		\end{equation}
		Since $f \in L^1([0,\infty) \times \bR)$, we can pass to $t \to \infty$ in \eqref{duhamel-negTimes}, \begin{equation}
			\lim\limits_{t \to \infty} \tilde{S}(-t)(u(t), \partial_t u(t))  = (u_0,v_0) + \int_{0}^{\infty} \tilde{S}(-\tau) (0, f(\tau)) \ud\tau,
		\end{equation}
		strongly  in $L^{1,1}(\bR)$,  where the integral  $\int_{0}^{\infty} \tilde{S}(-\tau) (0, f(\tau)) \ud\tau$ is understood in the Bochner sense. Let us call this limit $(\bar{u}_0, \bar{v}_0)$ and define $u\lin(t,x) \coloneqq  S(t)(\bar{u}_0, \bar{v}_0)(x)$. Using again the bounds \eqref{est-duhamel-L^infty}-\eqref{est-duhamel-L^1}, we obtain \eqref{nonlin-scattering}. Indeed,
		\begin{align*}
			& \| (u(t), \partial_t u(t)) - \tilde{S}(t)(\bar{u}_0, \bar{v}_0) \|_{L^{1,1}(\bR)}  = \| \tilde{S}(t)[\tilde{S}(-t)(u(t), \partial_t u(t)) - (\bar{u}_0, \bar{v}_0)] \|_{L^{1,1}(\bR)} \\
			&  \lesssim  \| \tilde{S}(-t)(u(t), \partial_t u(t)) - (\bar{u}_0, \bar{v}_0) \|_{L^{1,1}(\bR)} \to 0 \textrm{ as } t \to \infty.
		\end{align*}
	\end{proof}

	\section{Cauchy theory for manifold valued solutions to wave map equation}\label{sec:M valued cauchy}
	This section is about the global in time solution theory to wave map problem \eqref{eq:fwm-vec}, defined on an arbitrary trapezoid, taking values on manifold $\cM$ with suitable initial data and forcing. As a consequence we also prove a scattering result when the domain trapezoid is $[0,\infty) \times \bR$.
	
	\subsection{Existence of global solution}\label{sec:wave-map-M-global-soln-existence}
	Two main results of this section  are Theorems \ref{thm-M-globalwave-BndTrap} and  \ref{thm-M-global wave-unBndTrap} about the global existence of $\cM$-valued mild solutions to the wave map equation \eqref{eq:fwm-vec}
	with external data defined respectively on bounded or general closed trapezoids. We always consider initial data $(u_0,v_0)$ taking values in the ``tangent bundle'' $T\cM$, i.e. satisfying the compatibility condition \eqref{eqn-compatibility}.

	We begin this section by proving the existence of global solutions to problem \eqref{eq:fwm-vec} for small data $(u_0,v_0,f)$. To shorten the notation we set, for any interval $I \subset \bR$,
	\begin{equation}\label{eqn-space-L11-M}
		L^{1,1}(I; \cM) \coloneqq  \bigl(C(I;\cM)\cap \dot{W}^{1,1}(I;\bR^n)\bigr)\times L^1(I;\bR^n),
	\end{equation}
	and
	\begin{equation}\label{eqn-space-L11-M-norm}
		\| (f,g)\|_{L^{1,1}(I)}\coloneqq  \| f\|_{L^\infty(I)} + \| f' \|_{L^1(I)} + \| g\|_{L^1(I)}.
	\end{equation}
	\begin{theorem}[\textbf{$\cM$-valued global solution for small data on an arbitrary trapezoid}]
		\label{thm-nonlin-cauchy-L1}
		There exists $\eta > 0$ such that the following is true. For every trapezoid $K$ with base $K_0$, for  every  initial  data $(u_0,v_0) \in L^{1,1}(K_0; \cM)$
		satisfying the compatibility condition \eqref{eqn-compatibility},
		and every external forcing $f \in L^1(K;\bR^n)$, if
		the  smallness conditions  \eqref{eqn-smallness of initial condition}-\ref{bound-Forcing}) hold, then there exists a unique  mild  $\cM$-valued solution of \eqref{eq:fwm-vec} on $K$ in the sense of Definition \ref{def-sol-mild}.
		This solution is well-posed in the Hadamard sense as in Theorem \ref{thm-global existence small data}.
	\end{theorem}
	\begin{proof}[Proof of Theorem \ref{thm-nonlin-cauchy-L1}] We assume that assumptions of Theorem \ref{thm-nonlin-cauchy-L1} hold. By Theorem \ref{thm-global existence small data}
		there exists a unique $\bR^n$-valued solution $u$ of  \eqref{eq:fwm-vec} on $K$ with initial  data
		$(u_0,v_0)$ and external forcing $f$. Put $v=\partial_t u$. We need to prove that  $(u(t),v(t))$ satisfies \eqref{eqn-compatibility} for every $t\geq 0$.
		
		By Lemma \ref{lem-projm-3} we can find a sequence $(u_0^k,v_0^k)$ taking values in $\bigl[ C(K_0;M) \cap  \dot{W}^{2,2} (K_0;\bR^n) \bigr] \times \bigl[ C_0(K_0;TM) \cap \dot{W}^{1,2} (K_0;\bR^n) \bigr]$ such that
		\begin{equation}\label{eqn-convergence-2}
			\lim_{k\to\infty}\vert (u_0^k,v_0^k) -(u_0,v_0)\vert_{ \bigl( C(K_0) \cap \dot{W}^{1,1}(K_0) \bigr)  \times L^1(K_0)}=0.
		\end{equation}
		In particular, every $(u_0^k,v_0^k)$ satisfies  the compatibility condition \eqref{eqn-compatibility}.
		
		Moreover, we consider a natural extension of $f$, which we denote by $\hat{f}$, to $[0,\infty) \times \bR$ by setting $\hat{f} =0$ on $([0,\infty) \times \bR) \setminus K$. Then we can find a sequence $\hat{f}^k \in C_0^\infty ([0,\infty) \times \bR; \bR^n) $  such that
		$$\lim_{k \to \infty} \Vert \hat{f}^k - \hat{f} \Vert_{L^1([0,\infty) \times \bR)} =0.$$
		Thus, the sequence  $f^k\coloneqq  \hat{f}^k|_K $ approximates $f$ in $L^1(K;\bR^n)$.

		Next,  by \cite[Theorem 4.1]{BGOR22}, see also \cite{BO07}, for each $k$ there exists a unique global  $\cTM$-valued  solution $u^k$
		to the geometric wave map equation \eqref{eq:fwm-vec}  with initial  data
		$(u_0^k,v_0^k)$ and the external forcing $f^k$.  Let us denote $v^k=\partial_t u^k$. In particular, for every $t\in [0,\infty)$, each $(u^k(t),v^k(t))$ satisfies the compatibility condition \eqref{eqn-compatibility}. By the continuous dependence of solutions on the initial data, see last part of Theorem \ref{thm-global existence small data},
		$u^k$ converges to $u$ in $\mathscr{H}(K)$. Since $\cM$ is a closed  subspace of $\bR^n$ and $u^k$ is $\cM$-valued,   we infer that $u$ is $\cM$ valued as well. Moreover, the compatibility condition \eqref{eqn-compatibility-t} at every $t\geq 0$ is also satisfied by the last part of Theorem \ref{thm-trace}. Indeed,
		let $\phi \in C_0^\infty( \textrm{int}(K_0))$ and $t \geq 0$. Since  $(u^k(t),v^k(t))$ satisfies \eqref{eqn-compatibility-t}, from \eqref{eqn-compatibility-2} we infer that
		$\Upsilon^M_{\phi}(u^k(t),v^k(t))=0$,  where $\Upsilon^M_{\phi}$  is defined in \eqref{eqn-compatibility-t-2}. Since $\Upsilon^M_{\phi}$ is continuous in appropriate topologies, see
		Theorem \ref{thm-trace}, and $(u^k,v^k) \to (u,v)$ in the right topology, we infer that
		$\Upsilon^M_{\phi}(u^k(t),v^k(t))=0$. Because $\phi \in C_0^\infty( \textrm{int}(K_0))$  was chosen in an arbitrary way, we deduce  that $(u(t),v(t))$ satisfies \eqref{eqn-compatibility} for every $t\geq 0$. The proof of Theorem \ref{thm-nonlin-cauchy-L1} is complete.
	\end{proof}

	If  in the above proof  instead of using Theorem \ref{thm-global existence small data} we used Theorems \ref{thm-nonlin-cauchy-bigData-bndTrap} and \ref{thm-nonlin-cauchy-bigData-UnBndTrap}
	then we would get the following local existence result for the wave map equation \eqref{eq:fwm-vec} for arbitrary data.

	\begin{theorem}[\textbf{$\cM$-valued local solution for arbitrary data on an arbitrary trapezoid}]
		\label{thm-nonlin-cauchy-local}
		If $K$ is a  trapezoid  with base $K_0$, the   initial  data
		$(u_0,v_0) \in L^{1,1}(K_0; \cM) $   satisfy the compatibility condition \eqref{eqn-compatibility},
		and the external forcing $f \in L^1(K;\bR^n)$,
		then there exists   $T_0  \in (0,L)]$, where  $L \in (0,\infty]$ is the  height of $K$, such that \eqref{eq:fwm-vec} has a unique  mild  solution in the  trapezoid   with   base $K_0$ and height $T_0$.
	\end{theorem}
	
	\begin{remark}\label{rem-relaxed assumptions on F}
		If we were only interested in ther local existence and uniqueness, we could somewhat relax the conditions on non-linearity in equation \eqref{eq:fwm-vec}.
		For instance if the non-linearity, as defined in \eqref{eqn-Phi-def},  could be written in the form
		\begin{align*}
			F(t, x, u, \partial_t u, \partial_x u) & = g\big(t, x, u, (\partial_t - \partial_x) u, (\partial_t + \partial_x) u\big)\\
			& =: g\big(t, x, u, v_+, v_-\big)
		\end{align*}
		with $g$ satisfying the following condition, then for every $C > 0$ there exist positive constants $a, b_+, b_-, c$ such that
		if $|u|, |\sh u| \leq C$, $|v_+|, |v_-|$ have bounded integrals on the rays transversal to the appropriate light rays
		(bounds depending on $C$)
		and $\|v_+ v_-\|_{L^1} \leq C$, then
		\begin{equation}
			|g(t, x, u, v_+, v_-)| \leq a(t, x) + b_+(t, x)|v_+| + b_-(t, x)|v_-| + c(t, x)|v_+v_-|,
		\end{equation}
		and
		\begin{equation}
			\begin{aligned}
				&|g(t, x, u, v_+, v_-) - g(t, x, \sh u, \sh v_+, \sh v_-)| \\
				&\leq a(t, x)|u - \sh u| + b_+(t, x)|v_+ - \sh v_+| + b_-(t, x)|v_- - \sh v_-| + c(t, x)|v_+v_- - \sh v_+\sh v_-|,
			\end{aligned}
		\end{equation}
		with $a, b_+, b_-, c$ satisfying $a \in L^1$, $c \in L^\infty$, $b_+, b_-$ of bounded integrals on the light rays.	
	\end{remark}
	
	The following two results are concerned with the existence of global $\cM$-valued solutions. The first one, see Theorem \ref{thm-M-globalwave-BndTrap}, is for closed  bounded trapezoid and the second one, see Theorem \ref{thm-M-global wave-unBndTrap},  is for closed unbounded trapezoids.
	\begin{theorem}[\textbf{$\cM$-valued global solution for arbitrary data on a compact trapezoid}]
		\label{thm-M-globalwave-BndTrap}
		Let $L>0$, $x_0 \in \bR$ and assume that  $K$ is a compact trapezoid with base $K_0=[x_0-L,x_0+L]$ and height $T \in (0,L]$. We assume that $u$ is a local in time mild solution to  the wave map equation \eqref{eq:fwm-vec} with initial data $(u_0, v_0) \in L^{1,1}(K_0;\cM)$, which satisfies the compatibility condition \eqref{eqn-compatibility} and forcing $f \in L^1(K_0;\bR^n)$.
		Then $u$ is a global mild solution,  i.e. defined on the whole $K$, to the wave map equation \eqref{eq:fwm-vec}.
	\end{theorem}

	To prove Theorem \ref{thm-M-globalwave-BndTrap} we need the following fundamental result about $L^1$-estimates for solutions to the wave map equation \eqref{eq:fwm-vec}.
	\begin{proposition}\label{prop-noncon}
		Let the assumptions of Theorem \ref{thm-M-globalwave-BndTrap} are satisfied.
		Then,  for every $t_0 \in [0,T]$  we have
		
		\begin{equation}\label{prop-noncon-est1}
			\begin{split}
				\int_{x_0-L+t_0}^{x_0+L-t_0} |(\partial_t-\partial_x) u(t_0,x)| \ud x &\leq \int_{x_0-L}^{x_0+L-2t_0} |(\partial_t-\partial_x) u(0,x)| \ud x \\
				&+ \int_{0}^{t_0} \int_{x_0-L+t}^{x_0+L-2t_0+t} |f(t,x)| \ud x \ud t,
			\end{split}
		\end{equation}
		and
		\begin{equation}\label{prop-noncon-est}
			\begin{split}
				\int_{x_0-L+t_0}^{x_0+L-t_0} |(\partial_t+\partial_x) u(t_0,x)| \ud x &\leq \int_{x_0-L+2t_0}^{x_0+L} |(\partial_t+\partial_x) u(0,x)| \ud x
				\\ &+ \int_{0}^{t_0} \int_{x_0-L+2t_0-t}^{x_0+L-t} |f(t,x)| \ud x \ud t. 
			\end{split}
		\end{equation}
	\end{proposition}

	\begin{proof}[Proof of Proposition \ref{prop-noncon}]
		We will only prove inequality \eqref{prop-noncon-est1} as the proof of \eqref{prop-noncon-est} is similar.
		
		By employing Lemma \ref{lem-projm-3} as we did in the proof of the invariance of $\cM$, see the proof of the last part of Theorem \ref{thm-nonlin-cauchy-L1},
		we can assume that $u$ and $f$ are of $C^1$-class.
		
		Note that it is sufficient to show that for all $y \in [x_0-L, x_0+L-2t_0]$
		\begin{equation}\label{prop-noncon-est2}
			(\partial_t-\partial_x) u(t_0,y+t_0)| \leq (\partial_t-\partial_x) u(0,x)| + \int_{0}^{t_0} |f(t,y+t)| \ud t,
		\end{equation}
		because by integrating above in $y$, with the change of variable $x=y+t_0$,  we get \eqref{prop-noncon-est1}. Let us fix $y \in [x_0-L, x_0+L-2t_0]$ and set
		\begin{equation}\label{prop-noncon-exp1}
			v_+(t) \coloneqq  (\partial_t-\partial_x) u(t,y+t), \textrm{ and } f_+ (t)\coloneqq  f(t,y+t).
		\end{equation}
		We then have $v_+'(t) = (\partial_t^2-\partial_x^2) u(t,y+t)$. Recall that $(\partial_t^2-\partial_x^2) u(t,y+t) - f(t,y+t) \perp v_+(t)$, since $v_+(t) \in T_{u(t,y+t)}\cM$.  Thus, we have
		$$v_+'(t) - f_+(t) \perp v_+(t),$$
		which yields
		\begin{equation}\label{prop-noncon-eq1}
			\frac{\ud}{\ud \ut} |v_+(t)\vert^2 = 2 v_+'(t) \cdot v_+(t) = 2 f_+(t) \cdot v_+(t).
		\end{equation}
		Due to notation \eqref{prop-noncon-exp1}, our aim \eqref{prop-noncon-est2} is to prove that
		\begin{equation}\label{prop-noncon-est3}
			|v_+(t_0)| \leq |v_+(0)| + \int_{0}^{t_0} |f_+(t)|\ud t.
		\end{equation}
		If $v_+(t_0)=0$, then we are done. So, assume that $v_+(t_0) \neq 0$ and let $\tilde{T} \geq 0$ be the smallest positive number such that $v_+(t) \neq 0$ for all $t \in (\tilde{T}, t_0]$. Such $\tilde{T}$ exists due to continuity of $v_+$. Thus, either $v_+(\tilde{T}) =0$ or $\tilde{T}=0$, in any case we get $|v_+(\tilde{T})| \leq |v_+(0)|$. Next, since $v_+(t) \neq 0$ for all $t \in (\tilde{T}, t_0]$, as an application of the Cauchy-Schwarz inequality the expression  \eqref{prop-noncon-eq1} yields
		\begin{equation*}
			\frac{\ud}{\ud \ut} |v_+(t)| \leq   f_+(t), \qquad  \forall t \in (\tilde{T}, t_0].
		\end{equation*}
		Consequently, the fundamental theorem gives, since $|v_+(\tilde{T})| \leq |v_+(0)|$,
		\begin{equation*}
			|v_+(t_0)| \leq |v_+(\tilde{T})| + \int_{\tilde{T}}^{t_0} |f_+(t)|\ud t \leq |v_+(0)| + \int_{0}^{t_0} |f_+(t)|\ud t .
		\end{equation*}
	\end{proof}

	\begin{proof}[Proof of Theorem \ref{thm-M-globalwave-BndTrap}]
		Let us introduce a family $\mathscr{U}$ consisting of solutions to the wave map equation \eqref{eq:fwm-vec}.  By Theorem \ref{thm-nonlin-cauchy-local} this family is not empty. For  each $u \in \mathscr{U}$, let us denote its compact domain trapezoid by $\tilde{K}(u)$ with base $K_0=[x_0-L,x_0+L]$ and height $\tau=\tau(u) \leq T$. Moreover, by the finite speed of propagation proven in Corollary \ref{cor-finite speed of propagation}, we deduce that if $u,v \in \mathscr{U}$, either $u \subset v$ or $v \subset u$, where by $u \subset v$ we mean that  $\tilde{K}(u) \subset \tilde{K}(v)$ and  $v|_{\tilde{K}(u)} = u$. Let $T_+$ be the supremum of the bounded set $\{ \tau(u): u \in \mathscr{U}\}$.
		We will prove that $T_+=T$.
		Suppose by contradiction that $T_+<T$.

		The method of obtaining a contradiction here is taken from \cite{KT98}. Let $K_+$ be the compact trapezoid with base $K_0$ and height $T_+$, see  figure \ref{K-K_+}.
		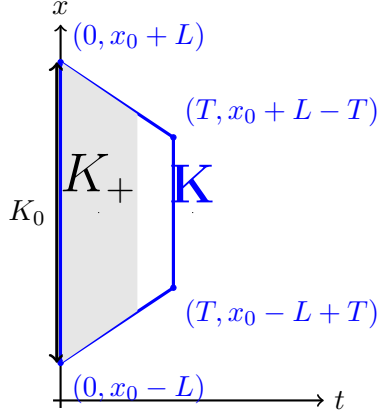
\begin{figure}
			\begin{tikzpicture}[scale=0.50]
				\draw[->,  thick] (-0.2,0)--(7,0) node[right]{$t$};	
				\draw[->, thick] (0,-0.2)--(0,10) node[above]{$x$};
				\draw[blue,very thick] (0,1) node[anchor=north west]{$(0,x_0-L)$}
				-- (0,9) node[anchor=south west]{$(0,x_0+L)$}
				-- (3,7) node[anchor=south west]{$(T,x_0+L-T)$}
				-- (3,3) node[anchor=north west]{$(T,x_0-L+T)$}
				-- cycle;
				\draw[dashed]  (3.5,5)  -- (3.5,5)node[anchor=south]{\textbf{\textcolor{blue}{\huge{K}}}};
				\filldraw[blue] (0,1) circle (2pt);
				\filldraw[blue] (0,9) circle (2pt);
				\filldraw[blue] (3,7) circle (2pt);
				\filldraw[blue] (3,3) circle (2pt);
				\coordinate (v1) at (-0.1,1);
				\coordinate (v2) at (-0.1,9); \draw[<->,very thick] (v1) -- node[left] {$K_0$} (v2);
				\filldraw[gray!20, very thick] (0.1,1.1) node[anchor=north west]{}
				-- (0.1,8.9) node[anchor=south west]{}
				-- (2,7.6) node[anchor=south west]{}
				-- (2,2.4) node[anchor=north west]{}
				-- cycle;
				\draw[dashed]  (1,5)  -- (1,5)node[anchor=south]{\textbf{\textcolor{black}{\huge{$K_+$}}}};
			\end{tikzpicture}
			\caption{Bounded  trapezoids $K$  and $K_+$.}
			\label{K-K_+}
		\end{figure}
		
		Let us fix $\eta$ as in Theorem \ref{thm-global existence small data}. Then clearly by the Lebesgue Dominated Convergence Theorem  there exists $\eps_0 >0$ such that, $\eps_0 \leq T_0 < T_+ + \eps_0 \leq T$  and
		\begin{equation}\label{cor-no-blowup-1}
			\| f\|_{L^1\left( K_{[T_+-\eps_0, T_++\eps_0]} \right) } \leq \eta,
		\end{equation}
		where $K_{[T_+-\eps_0, T_++\eps_0]}=\bigl\{ (t,x) \in K: t \in [T_+-\eps_0, T_++\eps_0]\bigr\}$.

		Note that $K_{[T_+-\eps_0, T_++\eps_0]}$ is a trapezoid with vertices $(T_+-\eps_0, x_0-L+T_+-\eps_0)$,   $(T_+-\eps_0, x_0+L-T_++\eps_0)$, $(T_++\eps_0, x_0+L-T_+-\eps_0)$  and $(T_++\eps_0, x_0-L+T_++\eps_0)$.
		
		Next, we define a continuous and increasing function $\phi: K_0=[x_0-L, x_0+L] \to \bR$ by
		\begin{equation*}
			\phi: [x_0-L, x_0+L] \ni w \mapsto 		  \int\int_{A_w} |f(s,y)| \, \ud s \, \ud y \in \bR,
		\end{equation*}
		where $A_w =\{ (t,x) \in K: x-t \leq w\}$. 		
		Since the set $K_0$ is compact, $\phi$ is uniformly continuous and thus, we can  choose $\eps_1 >0$ such that
		\begin{equation*}
			\phi(w+2\eps_1) - \phi(w) \leq \frac{\eta}{4}, \qquad \forall w \in [x_0-L, x_0+L-2\eps_1].
		\end{equation*}
		\begin{figure}
			\begin{tikzpicture}[scale=0.50]
				\draw[->,  thick] (-0.2,0)--(7,0) node[right]{$t$};	
				\draw[->, thick] (0,-0.2)--(0,10) node[above]{$x$};
				\draw[red,very thick] (0,1) node[anchor=north west]{$(0,x_0-L)$}
				-- (0,9) node[anchor=south west]{$(0,x_0+L)$}
				-- (3,7) node[anchor=south west]{$(T_+,x_0+L-T_+)$}
				-- (3,3) node[anchor=north west]{$(T_+,x_0-L+T_+)$}
				-- cycle;
				\draw[dashed]  (4,5)  -- (4,5)node[anchor=south]{\textbf{\textcolor{red}{\huge{$K_+$}}}};
				\filldraw[red] (0,1) circle (2pt);
				\filldraw[red] (0,9) circle (2pt);
				\filldraw[red] (3,7) circle (2pt);
				\filldraw[red] (3,3) circle (2pt);
				\filldraw[teal] (0,7) circle (2pt);
				\filldraw[teal!70, very thick] (0.05,1.1) node[anchor=north west]{}
				-- (0.05,7) node[anchor=north east]{$w$}
				-- (1.46,7.97) node[anchor=south west]{}
				-- (2.95,6.97) node[anchor=south west]{}
				-- (2.95,3.03) node[anchor=north west]{}
				-- cycle;
				\draw[dashed]  (1,5)  -- (1,5)node[anchor=south]{\textbf{\textcolor{black}{\huge{$A_w$}}}};
			\end{tikzpicture}
			\caption{$A_w$.}
			\label{K_+w}
		\end{figure}
		Further, by \cite[Theorem 6.11]{Rudin74RCA} we can choose $\eps_2>0$ such that
		\begin{equation}
			\begin{split}
				\int_{y-\eps_2}^{y+\eps_2} |(\partial_t-\partial_x)u(0,x)| \ud t &\leq \frac{\eta}{4}, \qquad \forall y \in (x_0-L+\eps_2, x_0+L-\eps_2),
				\\
				\int_{y-\eps_2}^{y+\eps_2} |(\partial_t+\partial_x)u(0,x)| \ud t &\leq \frac{\eta}{4}, \qquad \forall y \in (x_0-L+\eps_2, x_0+L-\eps_2).
			\end{split}	
		\end{equation}
		To move further, we set $\eps_3 = \min\{\eps_1,\eps_2\}$.  Thus, by using \eqref{prop-noncon-est1} and  \eqref{prop-noncon-est} from
		Proposition \ref{prop-noncon}, we deduce that  for every $T_0 < T_+$ we get
		\begin{align}\label{cor-no-blowup-3}
			\int_{y-\eps_3+T_0}^{y+\eps_3+T_0} |(\partial_t-\partial_x) u(T_0,x)| \ud x  &\leq \frac{\eta}{2}, \qquad \forall y \in (x_0-L+\eps_3, x_0+L-2T_0-\eps_3),
			\\		
			\label{cor-no-blowup-6}
			\int_{y-\eps_3-T_0}^{y+\eps_3-T_0} |(\partial_t+\partial_x) u(T_0,x)| \ud x  &\leq \frac{\eta}{2}, \qquad \forall y \in (x_0-L+2T_0 +\eps_3, x_0+L-\eps_3).
		\end{align}
		Now we set $\eps = \min\{ \eps_0, \eps_3 \}$ and get that
		\begin{equation}\label{cor-no-blowup-7}
			\int_{y-\eps-T_0}^{y+\eps+T_0} |\partial_t u(T_0,x)\pm\partial_x u(T_0,x)| \ud x  \leq \eta,
		\end{equation}
		for every $T_0 <T_+$  and $y \in (x_0-L+2T_0 + \eps , x_0+L-2 T_0-\eps)$. This implies that
		\begin{equation}\label{cor-no-blowup-8}
			\int_{x_0-L+T_0}^{x_0+L-T_0} |\partial_t u(T_0,x)\pm\partial_x u(T_0,x)| \ud x  \leq \eta, \quad \forall T_0 <T_+.
		\end{equation}
		Next, let us take an arbitrary $T_0 \in (T_+-\eps,T_+)$ and consider the trapezoid $K_\eps$ with vertices $(T_0, x_0-L+T_0)$,   $(T_0, x_0+L-T_0)$, $(T_0+\eps_0, x_0+L-T_+-\eps_0)$  and $(T_0+\eps_0, x_0-L+T_++\eps_0)$, i.e.  $K_\eps=K_{[T_0, T_0+\eps_0]}$.
		
		Since  \eqref{cor-no-blowup-8} holds and \eqref{cor-no-blowup-1} implies that $\|f\|_{L^1(K_\eps)} \leq \eta$, by Theorem  \ref{thm-global existence small data} there exists  a unique  mild  solution $u_{\eps}$ to  problem \eqref{eq:fwm-vec}  on  trapezoid $K_\eps$, with initial data $(u(T_0),\partial_t u(T_0))$.
		
		Since $u$ is also a unique  mild  solution to problem \eqref{eq:fwm-vec} on $K_+$, by the uniqueness of mild solutions we infer that $u=u_\eps$ on $K_+ \cap K_\eps$. But, since $T_0+\eps_0>T_+$, this implies that the solution $u$ can be extended beyond $T_+$ what contradicts the definition of $T_+$. Hence Theorem \ref{thm-M-globalwave-BndTrap} follows.
	\end{proof}

	\begin{theorem}[\textbf{$\cM$-valued global solution for arbitrary data on an unbounded trapezoid}]
		\label{thm-M-global wave-unBndTrap}
		Let us assume that $\widetilde{K}$ is any closed and unbounded trapezoid with height $L \in (0,\infty]$ and base $\widetilde{K}_0$ which is a connected subset of $\bR$. Assume that the  initial data $(u_0, v_0) \in L^{1,1}(\widetilde{K}_0;\bR^n)$ and satisfy the compatibility condition \eqref{eqn-compatibility}. Let us further assume that the forcing $f \in L^1(\widetilde{K}; \bR^n)$.  Then there exists a global, i.e. on the whole $\widetilde{K}$,    solution $u$ is global and satisfy the wave map equation \eqref{eq:fwm-vec}.
		Finally, the solution $u$ depends continuously on the initial data, where $u$ is considered in the space $\mathscr{H}(\widetilde{K})$.
	\end{theorem}
	\begin{proof}[Proof of Theorem \ref{thm-M-global wave-unBndTrap}] The proof of this result is based on Theorem \ref{thm-M-globalwave-BndTrap} in a similar way the proof of
		Theorem \ref{thm-nonlin-cauchy-bigData-UnBndTrap} was based on the proof of Theorem \ref{thm-nonlin-cauchy-bigData-bndTrap}.
		
	\end{proof}
	
	\begin{remark}\label{rem-KT98-global-soln}
		As noticed by Keel and Tao in \cite{KT98}, one can prove global existence using the \emph{pointwise conservation laws}.
		It seems it suffices to assume that there exists an integrable function $a$ and functions $b_\pm$ with bounded integrals over the light rays such that
		\begin{equation}
			|g(t, x, u, v_+, v_-)^\perp| \leq a(t, x)(1 + |u|) + b_+(t, x)|v_+| + b_-(t, x)|v_-|,
		\end{equation}
		where $\perp$ denotes the projection on the plane spanned by $v_+$ and $v_-$.
		
		Global existence can also be deduced from boundedness of the energy,
		but this would probably require stronger assumptions on the forcing term.
	\end{remark}

	\subsection{Uniqueness of global solutions}
	\label{sec:wave-map-M-global-soln-uniq}
	The idea is to prove the uniqueness result first for compact trapezoid domain, then extend the result to unbounded trapezoid case along the lines of proof of Theorem \ref{thm-nonlin-cauchy-bigData-UnBndTrap}.
	\begin{theorem}
		\label{thm-uniqueness-wave-map}
		Assume that $L>0$ and $K$ is a compact trapezoid with base $K_0$ and height $T_0 \in (0, L]$.  Assume that the
		initial data $(u_0, v_0) \in L^{1,1}(K_0;\cM)$, satisfies the compatibility condition \eqref{eqn-compatibility},   and the forcing $f \in L^1(K;\bR^n)$.
		Let $u^1,u^2 \in \mathscr{H}(K)$ be two $\cM$-valued  mild  solutions in $K$  to problem \eqref{eq:fwm-vec}, in the sense of Definition \ref{def-sol-mild}, with the same forcing $f$ and the initial condition \eqref{eqr-IC} with the same data $(u_0,v_0)$.
		Then $u^1=u^2$.
	\end{theorem}
	
	\begin{proof}[Proof of Theorem \ref{thm-uniqueness-wave-map}]
		\textbf{Step 1.} 		
		Let us fix $K_0\subset \bR$, $T_0$ and $L>0$ as in the assumptions. Let us take $\eta>0$ as in Theorem \ref{thm-nonlin-cauchy-L1} and first observe that by the Lebesgue Dominated Convergence Theorem we can find $T_1  \leq T_0$ such that if $K^\ast$ is the  trapezoid with base $K_0$ and the height $T_1$, i.e.
		\[
		K^\ast=K \cap ([0,T_1]\times \bR)
		\]
		then
		\begin{equation}\label{eqn-f in L^1}
			\Vert f \Vert_{L^1(K^\ast)} < \eta.
		\end{equation}

		Since by assumptions $Du_0, v_0  \in L^1(K_0; \bR^n)$,
		by \cite[Theorem 6.11]{Rudin74RCA}   we can find $L_1>0$ such that  for every closed interval $I \subset K_0$, the following holds
		\begin{equation}\label{eqn-small initial data}
			\Leb(I)\leq 2 L_1 \then 		\|D u_0 + v_0\|_{L^1(I)} \leq \eta, \;\; \|D u_0 - v_0\|_{L^1(I)} \leq \eta.
		\end{equation}
		Put
		\begin{equation}\label{eqn-delta}
			\begin{split}
				\delta&\coloneqq \min\{ L_1, T_1\} \mbox{ and }
				\\
				K^\natural&\coloneqq  K^\ast \cap ([0, \delta/2 ]\times \bR)=K \cap ([0, \delta/2 ]\times \bR).
			\end{split}
		\end{equation}
		Let us note that $K^\natural$ is simply the  trapezoid with base $K_0$ and the height $\delta/2$.
		
		Let us denote by $ K(y_0,\delta)$ the maximal compact trapezoid with base $K_0(y_0,\delta)=[y_0-\delta,y_0+\delta] \cap K_0$ and the height $ \delta $. Then we have
		\begin{equation}\label{eqn-K natural}
			K^\natural= \bigcup_{y_0 \in K_0} K(y_0,\delta).
		\end{equation}

		\textbf{Step 2.}
		First we will show that  $u^1=u^2$ in $K^\natural$. In view of equality \eqref{eqn-K natural} is  sufficient to prove that
		$u^1=u^2$ in $K(y_0,\delta)$, for every $y_0\in K_0$.
		
		For this aim let us choose and fix $y_0\in K_0$. We observe that in view of \eqref{eqn-f in L^1}, \eqref{eqn-delta} and \eqref{eqn-small initial data}, due to Theorem  \ref{thm-nonlin-cauchy-L1}, there exists a \textbf{unique} $\cM$-valued mild solutions $u \in \mathscr{H}(K(y_0,\delta);\bR^n)$
		to problem \eqref{eq:fwm-vec}-\eqref{eqr-IC} in  the trapezoid  $K(y_0,\delta)$ with the forcing $f|_{K(y_0,\delta)}$ and  the initial data $(u_0|_{K_0(y_0,\delta)},v_0|_{K_0(y_0,\delta)})$. But, since the restrictions of $u^i$, $i=1,2$ on $K(y_0,\delta)$ are also mild solutions to problem \eqref{eq:fwm-vec}-\eqref{eqr-IC} with  forcing $f$ and  initial condition
		$(u_0,v_0)$  restricted to  $K(y_0,\delta)$ and $K_0(y_0,\delta)$ respectively,  we infer that
		$u^1=u^2 = u$ on $K(y_0,\delta)$. Hence, $u^1=u^2$ on $K^\natural$ as required.

		\textbf{Step 3.}
		In this step we define
		\begin{align}\label{eqn-T_ast}
			\begin{split}
				& \mathbf{S}  \coloneqq  \bigl\{T \in [0,T_0]: u^1=u^2 \mbox{ in } K \cap ([0,T]\times \bR)\bigr\}, \\
				& T_\ast \coloneqq \sup \mathbf{S}.				
			\end{split}
		\end{align}
		
		In \textbf{Step 2}  we proved that the set $\mathbf{S}$ defined in \eqref{eqn-T_ast} is non-empty and that $\delta/2 \in S$. Thus
		$T_\ast \in [\delta/2,L]$. It's obvious that either $\mathbf{S}=[0,T_\ast)$ or  $\mathbf{S}=[0,T_\ast]$.\\
	We aim now to prove that $\mathbf{S}=[0,T_\ast]$.
		Suppose by contradiction that $\mathbf{S}=[0,T_\ast)$.	Therefore,
		\[
		K \cap ([0,T_\ast)\times \bR)= \bigcup_{T \in \mathbf{S}} K \cap ([0,T]\times \bR).
		\]
		By the definition of the set $\mathbf{S}$  we infer that
		\begin{align}\label{eqn-u^1=u^2 on T_ast}
			u^1&=u^2 \mbox{ on } K \cap ([0,T_\ast)\times \bR).
		\end{align}
		Let us recall that $T_\ast \in [\delta/2,L]$. Hence, since by assumptions
		$u^1,u^2 \in \mathscr{H}(K)$,
		by the Trace Theorem  \ref{thm-trace} the functions
		$\bigl( \partial_x u^i(T_\ast, \cdot),\partial_t u^i(T_\ast,\cdot)\bigr)$, $i=1,2$,   belong to   $(L^\infty \cap \dot{W}^{1,1}(K_{T_\ast}))\times L^1(K_{T_\ast})$.
		Moreover, by the definition of the space $\mathscr{H}(K)$,
		if  $(T_n)$ is an sequence  such that $T_n \toup T_\ast$, then
		these two functions are limits in the space $(L^\infty \cap \dot{W}^{1,1}(K_{T_\ast}))\times L^1(K_{T_\ast})$ of
		the sequences  $\bigl( \partial_x u^i(T_n, \cdot),\partial_t u^i(T_n,\cdot)\bigr)$.
		Hence, by \eqref{eqn-u^1=u^2 on T_ast}, we infer that
		\begin{align}\label{eqn-u at T_ast}
			\bigl( \partial_x u^1(T_\ast, \cdot),\partial_t u^1(T_\ast,\cdot)\bigr)=  \bigl( \partial_x u^2(T_\ast, \cdot),\partial_t u^2(T_\ast,\cdot)\bigr).
		\end{align}
		Therefore, $T_\ast \in \mathbf{S}$, contradiction. Thus we infer that   $\mathbf{S}=[0,T_\ast]$.

		\textbf{Step 4.} Finally we are going to prove that $T_\ast=T_0$. Suppose by contradiction that $T_\ast<T_0$.
		Then, both functions $u^i$, $i=1,2$, or more precisely their appropriate restrictions,  belong to the space
		$u^1,u^2 \in \mathscr{H}(K\cap  ([T_\ast,T_0]\times \bR); \bR^n)$. Moreover, by \eqref{eqn-u at T_ast}, the traces of these functions at the initial time $T_\ast$ are equal. Furthermore, due to Remark \ref{rem-M-mild=weak} and \eqref{eqn-fwm-vector-in K-vector}, it is easy to see that each $u^i,i=1,2$ satisfy
		\begin{equation}
			\begin{split}
				&\int_{K_{T_0}} u^i(T_0,x)\partial_t \varphi(T_0,x)\ud x - \int_{K_{T_0}} \partial_t u^i(T_0,x)\varphi(T_0,x)\ud x
				\\		&+\iint_{K \cap ([T^\ast,T_0] \times \bR)}	u^i(\partial_t^2 \varphi - \partial_x^2 \varphi)\ud x\ud t =
				\int_{K_{T^\ast}} u^i(T^\ast,x)\partial_t \varphi(T^\ast,x)\ud x - \int_{K_{T^\ast}} \partial_t u^i(T^\ast,x)\varphi(T^\ast,x)\ud x
				\\
				&+  \iint_{K \cap ([T^\ast,T_0] \times \bR) } \bigl[ \sum_{j,k=1}^n \Gamma_{jk}(u^i)(\partial_t u_j^i\partial_t u_k^i - \partial_x u_j^i \partial_x u_k^i)
				+ \sum_{j=1}^nP_{j}(u^i)f_{j}\Bigr]  \varphi(t,x) \ud x\ud t .
		\end{split}\end{equation}
		Thus, we infer that $u^1,u^2$ are the weak (and so mild) solutions  to problem \eqref{eq:fwm-vec} with the same forcing $f_{|K\cap  ([T_\ast,T_0]\times \bR)}\in L^1(K\cap  ([T_\ast,T_0]\times \bR); \bR^n)$ and the initial data \eqref{eqr-IC} at time $T_\ast$   with the same data initial data \eqref{eqn-u at T_ast}  in the compact trapezoid  $K\cap  ([T_\ast,T_0] \times \bR)$.

		Then, by \textbf{Step 1} above, there exists $\delta_\ast \in (0,T_0-T_\ast]$, such that
		\begin{align}\label{eqn-u^1=u^2 on T_ast T_ast+delta}
			u^1&=u^2 \mbox{ on } K \cap ([T_\ast,T_\ast+\delta_\ast])\times \bR).
		\end{align}
		Combining the last equality \eqref{eqn-u^1=u^2 on T_ast T_ast+delta} with equality \eqref{eqn-u^1=u^2 on T_ast} we deduce that
		\begin{align}\label{eqn-u^1=u^2 on 0 T_ast+delta}
			u^1&=u^2 \mbox{ on } K \cap ([0,T_\ast+\delta_\ast])\times \bR).
		\end{align}
		This contradicts the definition of $T_\ast$.
		The proof of Theorem \ref{thm-uniqueness-wave-map} is now complete.
	\end{proof}
	
	Now we extend the above uniqueness result to a closed and unbounded trapezoid.
	\begin{theorem}
		\label{thm-uniqueness-wave-map-unBndTrap}
		Let $K$ be any closed and unbounded trapezoid with base $K_0$ and height $L \in (0,\infty]$. Let the
		initial data $(u_0, v_0) \in L^{1,1}(K_0;\cM)$ and the forcing $f \in L^1(K;\bR^n)$ satisfy the assumption as in Theorem \ref{thm-uniqueness-wave-map}.
		Let $u^1,u^2 \in \mathscr{H}(K)$ be two $\cM$-valued  mild  solutions in $K$  to problem \eqref{eq:fwm-vec}, in the sense Definition \ref{def-sol-mild} and whose existence is shown in Theorem\ref{thm-M-global wave-unBndTrap}, with the same forcing $f$ and the initial data $(u_0,v_0)$.
		Then $u^1=u^2$.
	\end{theorem}
	\begin{proof}[Proof of Theorem \ref{thm-uniqueness-wave-map-unBndTrap}]
		The idea of the proof here is similar to the proof of Theorem \ref{thm-nonlin-cauchy-bigData-UnBndTrap}. But we include the details here for completion. Let $\eta>0$ be as in Theorem \ref{thm-nonlin-cauchy-L1}.
		Since $(u_0,v_0) \in L^{1,1}(K_0;M)$ and $f \in L^1(K_0;\bR^n)$, there exists a bounded interval $[a,b] \subset K_0$  and $\bar{L} \in (0,L)$ such that $a+\bar{L} < b-\bar{L}$ and
		\begin{equation}\label{eqn-[a,b]-rays-small-data-uniq}
			\Vert D u_0+v_0\Vert_{L^1(K_{a,b}^c )}  \leq \eta  \mbox{ and } \Vert D u_0- v_0\Vert_{L^1(K_{a,b}^c)} \leq \eta,
		\end{equation}
		and
		\begin{equation}\label{eqn-[a,b]-rays-small-f-uniq}
			\Vert f\Vert_{L^1(K^+_{b,\bar{L}} )}  \leq \eta  \mbox{ and } \Vert f \Vert_{L^1(K^-_{a,\bar{L}})} \leq \eta,
		\end{equation}
		where $K_{a,b}^c $ and the semi-bounded trapezoids $K^+_{b,\bar{L}}$ and $K^-_{a,\bar{L}}$ are defined in the proof of Theorem \ref{thm-nonlin-cauchy-bigData-UnBndTrap}.

		Let $K_{a,b,\bar{L}}$ is the maximal compact trapezoid $K_{0,\bar{L}}$ with base $[a,b]$  and height $\bar{L}$.  Then due to Theorems \ref{thm-nonlin-cauchy-local}, \ref{thm-M-globalwave-BndTrap} and \ref{thm-uniqueness-wave-map} there exists a \textbf{unique} $\cM$-valued mild global solution $u^0$ on $K_{a,b,\bar{L}}$ to equation \eqref{eq:fwm-vec} with initial data and forcing $(u_0|_{[a,b]},v_0|_{[a,b]},f|_{K_{a,b,\bar{L}}})$. But since the restrictions of $u^1,u^2 \in \mathscr{H}(K)$ to $K_{0,\bar{L}}$ are also solution to problem \eqref{eq:fwm-vec} with initial data and forcing $(u_0|_{[a,b]},v_0|_{[a,b]},f|_{K_{a,b,\bar{L}}})$, by we get that $u^0=u^1=u^2$ on $K_{a,b,\bar{L}}$.
		
		Next, due to the smallness conditions \eqref{eqn-[a,b]-rays-small-data-uniq} and \eqref{eqn-[a,b]-rays-small-f-uniq}, by the existence of a \textbf{unique} global result for small data, see Theorem \ref{thm-nonlin-cauchy-L1}, we can find the unique  mild  solutions $u^+$ and $u^-$  for problem \eqref{eq:fwm-vec} on the  semi-bounded trapezoids  $K^+_{b,\bar{L}}$ and $K^-_{a,\bar{L}}$, respectively, with initial data $(u_0|_{[b-\bar{L},\infty) \cap  K_0}, v_0|_{[b-\bar{L},\infty) \cap  K_0})$ and $(u_0|_{(-\infty,a+\bar{L}] \cap  K_0}, v_0|_{(-\infty,a+\bar{L}] \cap  K_0})$, and forcing $f|_{K^+_{b,\bar{L}}}$ and $f_{K^-_{a,\bar{L}}}$, respectively. Similarly, as in last paragraph, we conclude that $u^1 = u^2=u^+$ on $K^+_{b,\bar{L}}$ and $u^1 = u^2=u^-$ on $K^-_{a,\bar{L}}$.
		
		Hence, the last two paragraph together imply that $u^1=u^2$ on $([0,\bar{L}] \times  K_0) \cap K$.
		
		By employing the arguments very similar to \textbf{Step 3} and \textbf{Step 4} in the proof of Theorem \ref{thm-uniqueness-wave-map}, we infer that $u^1=u^2$ on $K$. Hence we complete the proof of Theorem \ref{thm-uniqueness-wave-map-unBndTrap}.
	\end{proof}

	\subsection{Finite speed of propagation and continuous dependence}
	We have the following two corollaries. 	The first result is about the finite speed of propagation for large data. The second one is about continuity of the solution with respect to large data.
	\begin{corollary}\label{cor-finiteSOP}
		Let the assumption of Theorem \ref{thm-M-globalwave-BndTrap} or Theorem \ref{thm-nonlin-cauchy-bigData-UnBndTrap} be satisfied. Let $\widetilde{K} \subset K$ is another trapezoid, and $(\tilde{u}_0,\tilde{v}_0)$ is another choice of initial data in $L^{1,1}(\widetilde{K}_0)$  and let forcing $\tilde{f} \in L^1(\widetilde{K})$, which coincides on $\widetilde{K}_0$ and $\widetilde{K}$, respectively, with the initial data and forcing $(u_0,v_0,f)$. Then the corresponding  mild  solutions $u$ and $\tilde{u}$ to problem \eqref{eq:fwm-vec}, which exist by  Theorem \ref{thm-nonlin-cauchy-bigData-bndTrap}, are equal on $\widetilde{K}$.
	\end{corollary}
	\begin{proof}[Proof of Corollary \ref{cor-finiteSOP}] Using the notation from this Corollary we notice that by Proposition \ref{prop-restriction-mildSoln}  not only  $\tilde{u}$ but also  the restriction of $u$ to $\widetilde{K}$,  are  solutions   to problem \eqref{eq:fwm-vec} on $\widetilde{K}$ with data  $(\tilde{u}_0,\tilde{v}_0)$ and $\tilde{f}$ and then apply previous uniqueness result, i.e. Theorem \ref{thm-uniqueness-wave-map} or Theorem \ref{thm-uniqueness-wave-map-unBndTrap}.
	\end{proof}
	
	We  prove the following continuous dependence on initial data and forcing result only for compact trapezoid which can be extended to unbounded trapezoid setting by working on similar lines as the proof of Theorem \ref{thm-uniqueness-wave-map-unBndTrap}.
	\begin{corollary}
		\label{cor-continuous dependence-nonlin-we}
		Let us assume that $L>0$, $T_0 \in (0,L]$, trapezoid $K$ with base $K_0$ and height $T_0$ and initial data  $(u_0, v_0)$ and forcing $f$ satisfy the assumptions of Theorem \ref{thm-uniqueness-wave-map}.
		Assume also that $(u_0^n, v_0^n) \in L^{1,1}(K_0;\cM)$ is a sequence of compatible initial data which satisfy \eqref{eqn-compatibility}
		and  $f^n \in L^1(K;\bR^n)$ is a sequence of the forces.
		Let  $u \in \mathscr{H}(K)$, resp. $u^n \in \mathscr{H}(K)$ be $\cM$-valued unique global mild  solutions in   $K$  to problem \eqref{eq:fwm-vec} with the same forcing $f$ and the initial  data $(u_0,v_0)$, resp. $f^n$ and $(u_0^n, v_0^n)$. Such mild solutions exist because of previous results proved in Sections \ref{sec:wave-map-M-global-soln-existence} and \ref{sec:wave-map-M-global-soln-uniq}.
		Assume that
		\begin{equation}\label{eqn-ID-n}
			u_0^n \to u_0 \in C(K_0;\cM) \quad \mbox{ and   } \quad Du_0^n \to D u_0,  v_0^n \to v_0  \mbox{ in } L^1(K_0;\bR^n),
		\end{equation}
		and
		\begin{equation}\label{eqn-forcing}
			f^n \to f \mbox{ in } L^1(K;\bR^n).
		\end{equation}
		Then $u^n \to u$ in $\mathscr{H}(K)$.
	\end{corollary}
	\begin{proof}[Proof of Corollary \ref{cor-continuous dependence-nonlin-we}]
		Since the approach here is similar to the proof of Theorem \ref{thm-uniqueness-wave-map}, we will not define the notation again.
		We define the set $K^\natural$ as in  proof of Theorem \ref{thm-uniqueness-wave-map} using only the limiting objects, i.e. $u_0$, $v_0$, $f$ and $u$.
		
		\textbf{Step 1.}		First we claim that  $u^n\to u$, as $n \to \infty$,  in $\mathscr{H}(K^\natural)$. For this purpose let us observe that we can strengthen equality \eqref{eqn-K natural} by observing that the sum can be taken finite.  To be precise, we claim that there exists a finite set $\{y_1,\ldots, y_N\} \subset K_0$ such that
		\begin{equation}\label{eqn-K natural-2}
			K^\natural= \bigcup_{i=1}^N K(y_i,\delta).
		\end{equation}
		Indeed, the length of top side of the   trapezoid $K(y,\delta)$ is equal to $\delta>0$. This is the reason why we defined the   trapezoid $K(y,\delta)$  in our way. In the proof of the previous result Theorem \ref{thm-uniqueness-wave-map} we could have taken as $K(y,\delta)$  the trapezoid of the same base but of height $\delta$ instead of $ \frac\delta2 $.

		Because of the above equality \eqref{eqn-K natural-2} it is  sufficient to prove that, for every $y_0\in K_0$, $u^n\to u$, as $n \to \infty$,  in $\mathscr{H}(K(y_0,\delta))$.
		For this aim let us choose and fix $y_0\in K_0$. We observe that in view of the assumptions here the functions $u_0$, $v_0$, $u_0^n$'s and $v_0^n$'s restricted to $K_0(y_0,\delta)$, satisfy
		\begin{align}\label{eqn-ID-n-res}
			& u_0^n|_{K_0(y_0,\delta)} \to u_0|_{K_0(y_0,\delta)} \in C(K_0(y_0,\delta),M) \\\nonumber
			& \mbox{ and   }Du_0^n|_{K_0(y_0,\delta)} \to D u_0|_{K_0(y_0,\delta)},  v_0^n|_{K_0(y_0,\delta)} \to v_0|_{K_0(y_0,\delta)}  \mbox{ in } L^1(K_0(y_0,\delta)),
		\end{align}
		and
		the functions  functions $f$ and $f^n$'s restricted to  $K(y_0,\delta)$, satisfy
		\begin{equation}\label{eqn-forcing-res}
			f^n|_{K(y_0,\delta)} \to f|_{K(y_0,\delta)} \mbox{ in } L^1(K(y_0,\delta)).
		\end{equation}
		Therefore, by the continuous dependence part of the Theorem  \ref{thm-global existence small data}, $u^n\to u$, as $n \to \infty$,  in $\mathscr{H}(K(y_0,\delta))$. Since the point $y_0$ is arbitrary, the claim in Step 1 follows.
		
		\textbf{Step 2.}		To prove that the claim holds on whole $K$, let us define
		\begin{align}\label{eqn-T_ast-cor}
			\begin{split}
				T_\ast&\coloneqq \sup \mathbf{S}, \mbox{ where }\\
				\mathbf{S} &\coloneqq \bigl\{T \in [0,T_0]: u^n \to u \mbox{ in } \mathscr{H}(K \cap ([0,T]\times \bR) ) \bigr\}.
			\end{split}
		\end{align}
		
		From our first claim of this proof we know that the set $\mathbf{S}$ is non-empty because $ \frac\delta2 $ is an element of it. Thus $T_\ast \in [ \frac\delta2 ,L]$. It's obvious that either $\mathbf{S}=[0,T_\ast)$ or  $\mathbf{S}=[0,T_\ast]$.

		We aim now to prove that $\mathbf{S}=[0,T_\ast]$.
		Suppose by contradiction that $\mathbf{S}= [0,T_\ast)$.
		But since
		\[
		K \cap ([0,T_\ast)\times \bR)= \bigcup_{T \in \mathbf{S}} K \cap ([0,T]\times \bR),
		\]
		by the definition of the set $\mathbf{S}$  we infer that
		\begin{align}\label{eqn-un-to-u on T_ast}
			u^n \to u \mbox{ as } n \to \infty \mbox{ in } \mathscr{H}(K \cap ([0,T_\ast)\times \bR) ) .
		\end{align}
		Thus, since by assumptions
		$u^n,  u \in \mathscr{H}(K)$,
		by the Trace Theorem  \ref{thm-trace} the functions
		$$ \bigl( \partial_x u^n(T_\ast, \cdot),\partial_t u^n(T_\ast,\cdot)\bigr), \bigl( \partial_x u(T_\ast, \cdot),\partial_t u(T_\ast,\cdot)\bigr) \in (L^\infty \cap \dot{W}^{1,1}(K_{T_\ast}))\times L^1(K_{T_\ast}), $$
		for every $n$.
		Moreover, by the Definition of the space $\mathscr{H}(K)$,
		if  $(T_m)$ is an sequence  such that $T_m \toup T_\ast$, then
		these all these functions are limits in the space $(L^\infty \cap \dot{W}^{1,1}(K_{T_\ast}))\times L^1(K_{T_\ast})$ of
		the sequences (in $m$) $\bigl( \partial_x u^n(T_m, \cdot),\partial_t u^n(T_m,\cdot)\bigr), \bigl( \partial_x u(T_m, \cdot),\partial_t u(T_m,\cdot)\bigr). $
		Hence, by \eqref{eqn-un-to-u on T_ast}, we infer that
		\begin{align}\label{eqn-n^n-to-u at T_ast}
			\partial_x u^n(T_\ast, \cdot) \to  \partial_x u(T_\ast, \cdot) \mbox{ and } \partial_t u^n(T_\ast,\cdot)  \to \partial_t u(T_\ast,\cdot).
		\end{align}
		Therefore, $T_\ast \in \mathbf{S}$, contradiction. Thus we infer that   $\mathbf{S}=[0,T_\ast]$.

		It only remains to prove that $T_\ast=T_0$.  Suppose by contradiction that $T_\ast<T_0$.
		Then, all functions $u^n, u$, $n \in \mathbb{N}$, or more precisely their appropriate restrictions,  belong to the space
		$\mathscr{H}(K\cap  ([T_\ast,T_0]\times \bR))$. By \eqref{eqn-n^n-to-u at T_ast}, the
		traces of these functions at the initial time $T_\ast$ are equal.
		
		Finally, as we did in \textbf{Step 4} in the proof of Theorem \ref{thm-uniqueness-wave-map} we can prove that
		$u^n,u$ are the weak solutions  to problem \eqref{eq:fwm-vec}, respectively with the forcing $f^n_{|K\cap  ([T_\ast,T_0]\times \bR)}, f_{|K\cap  ([T_\ast,T_0]\times \bR)}\in L^1(K\cap  ([T_\ast,T_0]\times \bR))$ and the initial condition \eqref{eqr-IC} at time $T_\ast$   with the same data initial data \eqref{eqn-n^n-to-u at T_ast}  in the compact trapezoid  $K\cap  ([T_\ast,T_0] \times \bR)$.

		Then, by Step 1 of the current proof, there exists $\delta_\ast \in (0,T_0-T_\ast]$, such that
		\begin{align}\label{eqn-u^n-to-u on T_ast T_ast+delta}
			u^n\to  u \mbox{ as } n \to \infty \mbox{ in } \mathscr{H}(K \cap ([T_\ast,T_\ast+\delta_\ast]\times \bR) ) .
		\end{align}
		Combining the last equality with equality \eqref{eqn-un-to-u on T_ast} we deduce that
		\begin{align}\label{eqn-u^n-to-u on 0 T_ast+delta}
			u^n \to u  \mbox{ as } n \to \infty  \mbox{ in } \mathscr{H}(K \cap ([0,T_\ast+\delta_\ast]\times \bR) ).
		\end{align}
		This contradicts the definition of $T_\ast$.
		The proof of Corollary \ref{cor-continuous dependence-nonlin-we} is now complete.		
	\end{proof}

	\section{Scattering result}
	Following is the main result of this section.
	\begin{theorem}[Scattering]
		\label{thm:scat-largedata}
		Assume that $K = [0, \infty) \times \bR$ with base $K_0=\bR$.  Assume that  $(u_0, v_0) \in L^{1,1}(K_0;\cM)$ and $f \in L^1(K;\bR^n)$.  For every $\cM$-valued global mild solution $u$ of problem \eqref{eqn-wave map forced}   there exists a new initial data $(\bar{u}_0,\bar{v}_0) \in L^{1,1}(\bR; \bR^n)$ such that if $u\lin$ is the corresponding unique solution to homogeneous linear wave \eqref{eqn-linear-wave-nonhom} with $h=0$ and initial data $(\bar{u}_0,\bar{v}_0)$,  then
		\begin{equation}
			\label{scattering-L11}
			\lim_{t \to \infty}\big(\|u(t) - u\lin(t)\|_{L^\infty(\bR)} + \|\partial_t u(t) - \partial_t u\lin(t)\|_{L^1(\bR)}
			+ \|\partial_x u(t) - \partial_x u\lin(t)\|_{L^1(\bR)}\big) = 0.
		\end{equation}
	\end{theorem}
	The proof of Theorem \ref{thm:scat-largedata} is based on the following observation, which is somehow a combination
	of inequality \eqref{prop-noncon-est3} with the ``Zhou estimates'' from subsection \ref{sec:Zhou-trick}.
	\begin{lemma}
		\label{lem:spacetime-int}
		Assume that $K$ is a trapezoid with base $[x_0 - L, x_0 + L]$ and height $T$, where $0 \leq T \leq L \leq +\infty$. Assume that $u$ is an $\cM$-valued mild  solution of problem \eqref{eqn-wave map forced}
		with initial data $(u_0, v_0)$ and forcing term $f$ satisfying the assumptions of Theorem \ref{thm:scat-largedata}. Then
		\begin{equation}\label{eq:lem-spacetime-int-ineq}
			\int_K |(\partial_t u(t, x))^2 - (\partial_x u(t, x))^2|\ud x \ud t \leq
			\bigg(\int_{x_0-L}^{x_0+L}(|D u_0(x)| + |v_0(x)|)\ud x + \int_K |f(t, x)|\ud x\ud t \bigg)^2.
		\end{equation}
	\end{lemma}
	\begin{proof}[Proof of Lemma \ref{lem:spacetime-int}]
		It suffices to consider finite $L$ and $T$, because the constant in inequality \eqref{eq:lem-spacetime-int-ineq} is $1$,  and eventually pass to the limit.
		
		Changing $T$ to $t$, $t$ to $s$ and $y+T$ to $x$ in inequality \eqref{prop-noncon-est3}, and recalling the definitions
		of $v_+$ and $f_+$ in \eqref{prop-noncon-exp1}, we obtain, for all $(t, x) \in K$,
		\begin{equation}
			\label{eq:pointwise-1}
			|(\partial_t - \partial_x)u(t, x)| \leq |v_0(x-t) - D u_0(x-t)| + \int_0^t|f(s, x-t + s)|\ud s.
		\end{equation}
		Analogously, for all $(t, x) \in K$,
		\begin{equation}
			\label{eq:pointwise-2}
			|(\partial_t + \partial_x)u(t, x)| \leq |v_0(x+t) + D u_0(x+t)| + \int_0^t|f(s, x+t - s)|\ud s.
		\end{equation}
		Then it suffices to apply Lemma~\ref{lem:zhou-lemma} with $g_+ = |v_0 - D u_0|$,
		$g_- = |v_0 + D u_0|$ and $f_+ = f_- = |f|$.
	\end{proof}

	\begin{proof}[Proof of Theorem~\ref{thm:scat-largedata}] The idea of the proof is borrowed from \cite[Section 9.4]{KT98}. As in the proof of Proposition \ref{prop-noncon} we do not directly not weak solutions with weak solutions but with suitable approximations. By applying Lemma \ref{lem-projm-3},  our present argument implies its validity also in the realm of weak solutions.
		
		Let us choose and fix 	the initial data  $ u_0 \in L^\infty(K_0)$,  $v_0, D u_0 \in L^1(K_0)$   and the forcing $f \in L^1(K)$.
		Let us choose and fix the mild solution $u: \bR_+ \times \bR \to M$ of problem \eqref{eqn-wave map forced}.
		Let us define the null coordinates
		\[
		a\coloneqq x+t,\;\;\; b\coloneqq x-t. \]
		Note that if $(t,x) \in [0,\infty) \times \bR$, then  $(a,b)\in K^\ast_+$, where  $K^\ast_+$ is the half-plane
		\[
		K^\ast_+\coloneqq \bigl\{(a,b) \in \bR^{2}: a \geq b  \bigr\}.
		\]
		Then, we denote the function  $u$ in the null coordinates
		by symbol $u^\ast$, i.e., $u^\ast$ is defined by
		\begin{align}\label{eqn-u^ast}
			u^\ast(a,b)&\coloneqq u(x+t,x-t), \quad \textrm{ or equivalently, } \\
			u(t,x) & = u^\ast((a-b)/2,(a+b)/2).
		\end{align}
		Similarly, we set
		\begin{align}\label{eqn-f^ast}
			f^\ast(a,b)\coloneqq f(x+t,x-t).
		\end{align}
		Let us first recall the definition of  the ``positive'' Einstein diamond
		\begin{align}\label{eqn-Einstein diamond}
			\bD_+^\ast\coloneqq  \{(A,B)\in \bR^2:  |A|, |B| < \pi/2, A \geq B\}.
		\end{align}
		Note that $\bD_+^*$ is a triangle  with the following three vertices $(\frac{\pi}2,\frac{\pi}2)$, $(\frac{\pi}2,-\frac{\pi}2)$ and $(-\frac{\pi}2,-\frac{\pi}2)$ and the function
		\[
		\bD_+^\ast \ni (A,B) \mapsto (a,b)\coloneqq (\tan A, \tan B) \in K^\ast_+
		\]	
		is a bijection.
		Let us also introduce the   ``conformal compactification transformation'' $Lu^\ast$ of the function $u^\ast$ by the following formula
		\begin{align}\label{eqn-conformal compactification}
			(Lu^\ast)(A,B) \coloneqq  u^\ast(\tan A, \tan B), \;\; (A,B) \in \bD_+^*.
		\end{align}
		Note that the transformation $L$ maps functions defined on $K_+^\ast$ to functions defined on $\bD_+^\ast$.
		
		Let the ``new'' physical variables $(T, X)$ for $L u^\ast$ be defined by
		\begin{align} A&=X+T, \quad \;\; B=X-T.
		\end{align}
		Observe that the set $\bK \coloneqq \{ (T,X): (A,B) \in \bD_+^*\}$ is also a triangle with vertices:
		\[
		\mbox{ $(0,\frac{\pi}2)$, $(\frac{\pi}2,0)$ and $(0,-\frac{\pi}2).$}
		\]
		So the set $\bK$ can be described  by
		\[
		\bK = \left\{(T,X): T \in [0,\frac{\pi}2), \;\; -\frac{\pi}2+T \leq X \leq \frac{\pi}2-T \right\}.
		\]
		This allows us to define  a transformation, denoted by $\lL$,  by the following formula, for $(T,X) \in \bK$,
		\begin{align}
			(\lL u)(T,X)\coloneqq  (Lu^\ast)(A,B)=u\left( \frac{\tan(X+T) - \tan(X-T)}{2}, \frac{\tan(X+T) + \tan(X-T)}{2} \right).
		\end{align}
		Let us observe that if $(t,x) \in \{0\}\times \bR$, then $X=\tan^{-1}x \in (-\frac{\pi}{2},\frac{\pi}{2})$ and  $T=0$.
		Thus
		\begin{align}
			(\lL u)[0](X) & \coloneqq  ((\lL u)(0,X),( \partial_T (\lL u))(0,X) ) \nonumber\\
			& = \left( u_0(\tan X) , \sec^2(X) v_0(\tan X)  \right).
		\end{align}
		It will be convenient to introduce the following additional auxiliary notation
		\begin{align*}
			U_0(X)&\coloneqq  u_0(\tan X), \;\; X \in (-\frac{\pi}{2},\frac{\pi}{2}),  \nonumber \\
			V_0(X)&\coloneqq  \sec^2(X) v_0(\tan X),  \;\; X \in (-\frac{\pi}{2},\frac{\pi}{2}), \nonumber
		\end{align*}
		and
		\begin{equation*}
			F(T,X)  \coloneqq  \sec^2(X+T)\sec^2(X-T) f\left( \frac{\tan(X+T) - \tan(X-T)}{2},  \frac{\tan(X+T) + \tan(X-T)}{2} \right).
		\end{equation*}
		A straightforward computation gives that
		\begin{align}\label{eqn-L^1 norm of LL u}
			\|(\lL u)[0]\|_{L_X^{1,1}}= \|u[0]\|_{L_x^{1,1}},
		\end{align}
		and
		\begin{align}\label{eqn-L^1 norm of F}
			\|F\|_{L^1(\bK)}= \|f\|_{L^1(K)},
		\end{align}
		where $\|\cdot\|_{L^{1,1}}$-norm is defined in \eqref{eqn-space-L11-norm}.
		Moreover,  we easily see that the function $U$ defined by
		\[U\coloneqq \lL u\]
		satisfies the wave map equation \eqref{eqn-wave map forced}, on $(T,X) \in [0,\infty) \times (-\frac{\pi}{2},\frac{\pi}{2})$,  with the forcing $F$, i.e.
		\begin{equation}
			\label{eq:fwm-capital}
			\partial_T^2 U - \partial_X^2 U = \sum_{1 \leq j,k\leq n} \Gamma_{jk}(U)(\partial_T U_j\partial_T U_k - \partial_X U_j \partial_X U_k)
			+ \sum_{1\leq j\leq n} P_{j}(U)F_{j}
		\end{equation}
		with the  initial data $(U_0,V_0)$.

		Let us now define a continuous   extension  of the initial data $(\lL u)[0] = (U_0,V_0)$, which is initially defined on the open interval $(-\frac{\pi}{2},\frac{\pi}{2})$.
		For this purpose we observe that, since by \eqref{eqn-L^1 norm of LL u} the weak derivative of $U_0$, i.e. the  first component of $(\lL u)[0]$, belongs to $L^1(-\frac{\pi}{2},\frac{\pi}{2})$, function $U_0$  extends to a continuous, in fact absolutely continuous, see \cite[Theorem 6.11]{Rudin74RCA},  function  defined on the closed interval
		$[-\frac{\pi}{2},\frac{\pi}{2}]$.  Then we can extend the initial position $U_0$ to whole $\bR$, we will continue to denote this extension as $\bU_0$, by
		\begin{equation}
			\bU_0(X) = \threepartdef
			{U_0(X),}      {X \in [-\frac{\pi}{2},\frac{\pi}{2}], }
			{U_0(\frac{\pi}{2}), }      {X \geq \frac{\pi}{2}, }
			{U_0(-\frac{\pi}{2}), }      {X \leq -\frac{\pi}{2}. }
		\end{equation}
		Analogously, we extend the initial velocity $V_0$, call the extension $\bV_0$, by zero on the intervals $(-\infty,-\pi/2)$ and $(\pi/2,\infty)$.
		Note that the weak derivative of the former function is the latter, i.e.
		\[
		D \mathbf{U}_0=\mathbf{V}_0 \quad \mathrm{ on } \quad \bR.
		\]
		We also extend $F$ to   $(T,X) \in [0,\infty) \times \bR$ by $0$ and this extension we will denote by $\mathbf{F}$.
		Thus, by Theorems \ref{thm-M-global wave-unBndTrap} and \ref{thm-uniqueness-wave-map} there exists a unique global solution $\mathbf{U}$ to equation
		\begin{equation}
			\label{eq:fwm-capital-bf}
			\partial_T^2 \mathbf{U} - \partial_X^2 \mathbf{U} = \sum_{1 \leq j,k\leq n}\Gamma_{ijk}(\mathbf{U})(\partial_T \mathbf{U}_j\partial_T \mathbf{U}_k - \partial_X \mathbf{U}_j \partial_X \mathbf{U}_k)
			+ \sum_{1\leq j\leq n}P_{j}(\mathbf{U})\mathbf{F}_{j},
		\end{equation}
		on $(T,X) \in [0,\infty) \times \bR$,  with the  initial data $(\bU_0,\bV_0)$.
		By Corollary \ref{cor-finiteSOP}, we infer that  $\mathbf{U}$ is an extension of  $U$.  Denoting by $\mathbf{U}^\ast$ the map $\mathbf{U}$ in the $(A,B)$-coordinates, we observe that  $\mathbf{U}^\ast$ satisfies
		\begin{equation}
			\label{eq:fwm-ast}
			\partial_A \partial_B \bU^*  = \sum_{1 \leq j,k\leq n}\Gamma_{jk}(\bU^\ast)(\partial_A \bU_j^\ast \partial_B \bU_k^\ast) + \sum_{1\leq j\leq n}P_{j}(\bU^\ast) \bF_{j}^\ast,
		\end{equation}
		where $\bF^\ast = (\bF_{1}^\ast, \ldots, \bF_{n}^\ast): (A,B) \to \bR^n$.

		Next, by Lemma \ref{lem-scatter-aux} below for $S=\pi/2$, we see that map $\bU$ also solves the free wave equation with initial data $(\bU_0, \bV_0) \in L^{1,1}(\bR)$ and forcing $\bF$ such that its null coordinate representation $\bU^\ast$ satisfies the following three conditions.  \begin{enumerate}
			\item $\bU^\ast$ is constant in the region
			$ \{ A \geq B\} \cap  \cR_C$ where
			$$ \cR_C \coloneqq  \bigl( [\pi/2,\infty) \times [\pi/2,\infty)  \bigr) \cup  \bigl( [\pi/2,\infty) \times (-\infty,-\pi/2]  \bigr) \cup \bigl( (-\infty,-\pi/2] \times (-\infty,-\pi/2] \bigr),$$
			\item $\bU^\ast$ is constant in the $B$ direction in the region $\{ A \geq B\} \cap \bigl( [-\pi/2,\pi/2] \times (-\infty,-\pi/2]  \bigr) $,
			\item $\bU^\ast$ is constant in the $A$ direction in the region $\{ A \geq B\} \cap \bigl( [\pi/2,\infty) \times [-\pi/2,\pi/2]  \bigr)$.
		\end{enumerate}
		Note that $\bU$ also solves the free wave equation with initial data $(\bU_0, \bV_0)$ in all the corresponding regions in $(T,X)$ mentioned above in points (1), (2) and (3).  Let us denote by $\bU_+^{*}$, resp. $\bU_{-}^{*}$,    the restriction of the function $\bU^*$  to the sets
		\begin{align*}
			& \bE_+^\ast \coloneqq  \{ A \geq B\} \cap \bigl( \{A \geq \pi/2\} \cup \{B \geq \pi/2\} \bigr), \\
			\textrm{resp.} \quad  & \bE_-^\ast \coloneqq  \{ A \geq B\} \cap \bigl( \{A \leq -\pi/2\} \cup \{B \leq -\pi/2\} \bigr).
		\end{align*}
		By using the above properties (1), (2) and (3), we can extend $\bU_{\pm}^*$, still denoting these extensions by the same symbol to be the solution of the free wave solution on $\{(A,B) \in \bR^2: A \geq B\}$ with initial data $(\bU_0, \bV_0)$. We denote the corresponding extensions to $\{(T,X) \in [0,\infty) \times \bR \}$ by $\bU_{\pm}$.  Thus the whole above construction gives us the following well-defined invertible four maps
		\begin{equation}
			M_{\pm}: u \mapsto \bU_{\pm},  \qquad M_{\pm}^\ast: u^\ast \mapsto \bU_{\pm}^\ast.
		\end{equation}
		We now can define the scattering maps $W_\pm: u \to u_{\pm}$ by defining
		\begin{equation}\label{def_W}
			u_{\pm}(t,x) \coloneqq   \bU_\pm \left( \frac{\tan^{-1}(x+t) - \tan^{-1}(x-t)}{2}, \frac{\tan^{-1}(x+t) + \tan^{-1}(x-t)}{2}  \right), \, (t,x) \in [0,\infty) \times \bR.
		\end{equation}
		To complete the proof of  property \eqref{scattering-L11},   it is sufficient to show that
		\begin{equation}\label{eq:L11-scatter}
			\|u(T^*) - u_- (T^*)\|_{L_x^{1,1}} \to 0 \quad \textrm{as} \quad T^* \to  \infty.
		\end{equation}
		Indeed, then  $u\lin \coloneqq  u_-$  is a solution to homogeneous linear wave \eqref{eqn-linear-wave-nonhom} with $h=0$ and initial data $(\bar{u}_0(x),\bar{v}_0(x))\coloneqq  (\bU_0(\tan^{-1}(x)), \bV_0(\tan^{-1}(x)))$,  	satisfies  \eqref{scattering-L11}.

		Note that, since $\bU$ and  $\bU_-$ are continuous functions,  for any $x \in \bR$,  \eqref{def_W} implies
		\begin{align}
			\lim_{T^* \to \infty}&  | u(T^*,x) - u_-(T^*,x)|  = | \bU (\pi/2,0) - \bU_- (\pi/2,0)| \nonumber\\
			& = | \bU^* (\pi/2,-\pi/2) - \bU_-^* (\pi/2,-\pi/2)|   =0, \nonumber
		\end{align}
		where the last equality follow because $\bU_-^*$ and $\bU^*$ agree on the set $\bE_-^\ast$.  This implies that  $\lim_{T^* \to \infty} \| u(T^*) - u_-(T^*)\|_{L_x^\infty}=0$.
		Thus, to get \eqref{eq:L11-scatter} it only remains to show that
		$$\lim\limits_{T^* \to \infty} \| \partial_x u(T^*) - \partial_x u_-(T^*) \|_{L_x^1} +  \| \partial_t u(T^*) - \partial_t u_-(T^*) \|_{L_x^1} = 0. $$
		But since $\partial_x u= \frac{1}{2} (\partial_a u + \partial_b u)$ and $\partial_t u= \frac{1}{2} (\partial_a u - \partial_b u)$, it is enough to show that
		\begin{equation}\label{eq:L11-1}
			\lim\limits_{T^* \to \infty} \left[\| \partial_a u(T^*) - \partial_a u_-(T^*) \|_{L_x^1} +  \| \partial_b u(T^*) - \partial_b u_-(T^*) \|_{L_x^1} \right] = 0.
		\end{equation}
		We show  only $\lim\limits_{T^* \to \infty} \| \partial_a u(T^*) - \partial_a u_-(T^*) \|_{L_x^1} $, as the argument for the second
		is analogous.  Note that we have
		\begin{align}\label{twopage}
			& \| \partial_a u(T^*) - \partial_a u_-(T^*) \|_{L^1_x} =  \int_{\bR} | \partial_a u^*(a,a-2T^*) - \partial_a u_-^*(a,a-2T^*)| \, \ud a  ,
		\end{align}
		where we have used that for fixed $a$ and $T^*$, we have $ b=a-2T^*$.

		Next, by substituting $a =\tan A$ along with using $\partial_A \bU^* (A,B) = \partial_a u^* (\tan A,\tan B) \sec^2 (A)$,  we rewrite  \eqref{twopage} in  $(A,B)$-coordinates as
		\begin{align}\label{twopage-1}
			& \| \partial_a u(T^*) - \partial_a u_-(T^*) \|_{L^1_x}  = \int_{-\pi/2}^{\pi/2}  |  (\partial_A \bU^* - \partial_A \bU_-^*)(A,\tan^{-1} (\tan (A)-2T^*))|  \,  \ud A .
		\end{align}
		Note that, for fix $A \in [-\pi/2, \pi/2]$,  by the Fundamental Theorem of Calculus, for any suitable $f$ we have
		$$ \int_{-\pi/2}^{\tan^{-1} (\tan (A)-2T^*)} \partial_B f(A,B) \, dB =   f(A, \tan^{-1} (\tan (A)-2T^*))-f(A, -\pi/2). $$
		Using above for $f=\partial_A\bU^*-\partial_A\bU_-^*$,  we infer that $\| \partial_a u(T^*) - \partial_a u_-(T^*) \|_{L^1_x}$ is bounded from above by
		\begin{align}\label{twopage-2}
			\int_{A=-\pi/2}^{\pi/2}  \int_{B=-\pi/2}^{\tan^{-1} (\tan (A)-2T^*)}   |  (\partial_A \partial_B \bU^*)(A,B) - (\partial_A \partial_B \bU_-^*)(A, B)|  \,  dB \, dA.
		\end{align}
		Indeed, $\bU^*-\bU_-^*$ vanishes at the lower boundary of the $\bD_+^\ast$.
		Hence, since $A \geq \tan^{-1} (\tan (A)-2T^*)$,  if we prove that
		\begin{align}\label{twopage-3}
			\int\int_{|A|,|B| \leq \pi/2, A \geq B} |(\partial_A \partial_B \bU^*-\partial_A \partial_B \bU_-^*)(A,B)| \,  dA \, dB < \infty,
		\end{align}
		then, due to the Lebesgue monotone convergence theorem, the expression \eqref{twopage-2} will converge to zero as $T^* \to \infty$ and we will be done.

		But this is true because of Lemma \ref{lem:spacetime-int},  in the null coordinates $(A,B)$ with $L=\infty$,  since $\bU_-^*$ is a solution to free wave equation, i.e, $\partial_A \partial_B \bU_-^* = 0$. Hence the proof of Theorem \ref{thm:scat-largedata} is complete.
	\end{proof}

	\begin{lemma}\label{lem-scatter-aux}  Suppose that $L^{1,1}$ data $(u_0, v_0)$  and forcing $f$ are given such that $D u_0,v_0$ and  $f$ are supported, respectively,  in the interval
		$[-S,S]$ and in the set
		$$\bK_S \coloneqq  \{ (T,X) \in [0,S] \times  \{-S+T \leq X \leq S-T \} \}. $$
		Let $\Phi(t,x)$ be the unique global solution to \eqref{eq:fwm-vec}  with initial data $(u_0,v_0)$ and forcing $f$. Let $\Phi^\ast(a,b)$ be the representation of $\Phi(t,x)$ in null coordinates. Then function $\Phi^\ast$ is
		\begin{enumerate}
			\item constant in the region
			$ \{ a \geq b\} \cap  \cR_C^\ast$ where
			$$ \cR_C^\ast \coloneqq  \bigl( [S,\infty) \times [S,\infty)  \bigr) \cup  \bigl( [S,\infty) \times (-\infty,-S]  \bigr) \cup \bigl( (-\infty,-S] \times (-\infty,-S] \bigr),$$
			\item constant in the $b$ direction in the region $\{ a \geq b \} \cap \bigl( [-S,S] \times (-\infty,-S]  \bigr) $,
			\item constant in the $a$ direction in the region $\{ a \geq b\} \cap \bigl( [S,\infty) \times [-S,S]  \bigr)$.
		\end{enumerate}
	\end{lemma}
	\begin{proof}[Proof of Lemma \ref{lem-scatter-aux}]
		Let us fix an $S>0$ and take $(t_0,x_0) \in ([0,\infty) \times \bR) \setminus \bK_S$. Let $\eps>0$ and set $L\coloneqq t_0+\eps$ such that $x_0-t_0 -\eps\geq S$. Proposition \ref{prop-noncon}, with $T\coloneqq t_0$ and $x_0\coloneqq x_0$, gives
		\begin{equation}\label{lem-scatter-est1}
			\int_{x_0-\eps}^{x_0+\eps} |(\partial_t-\partial_x) u(t_0,x)| \ud x \leq \int_{x_0-t_0-\eps}^{x_0-t_0+\eps} | v_0(x) - D u_0(x)| \ud x + \int_{0}^{t_0} \int_{x_0-t_0-\eps+t}^{x_0-t_0+\eps+t} |f(t,x)| \ud x \ud t.
		\end{equation}
		Due to the assumptions on the supports of $(u_0, v_0)$  and  $f$,  \eqref{lem-scatter-est1} implies that  
        \[\int_{x_0-\eps}^{x_0+\eps} |(\partial_t-\partial_x) u(t_0,x)| \ud x=0. \]
        Thus, $(\partial_t-\partial_x) \Phi(t_0,x)=0$ for almost every $x \in [x_0-\eps, x_0+\eps]$. Similarly, we can prove that $(\partial_t+\partial_x) \Phi(t_0,x)=0$ for almost every $x \in [x_0-\eps, x_0+\eps]$.
		
		Moreover, since $(\partial_t\pm\partial_x) \Phi(t_0,\cdot)$ are integrable, $\Phi(t_0,\cdot)$ is an absolutely continuous function and hence $\Phi(t_0,\cdot)$ is constant for  every $x \in [x_0-\eps, x_0+\eps]$.   Since $(t_0,x_0) \in ([0,\infty) \times \bR) \setminus \bK_S$ and $\eps>0$ are arbitrary and the solution $\Phi$ is unique, we can perform the computation on overlapping regions to deduce that $(\partial_t\pm\partial_x) \Phi(t_0,x_0)=0$ for $(t_0,x_0) \in ([0,\infty) \times \bR) \setminus \bK_S$. In particular, $\Phi(t,x)$ is constant in the region $(t,x) \in ([0,\infty) \times \bR) \setminus \bK_S$ and thus assertion (1) is proved.
		
		In a similar way one can prove assertions (2) and (3).
	\end{proof}
	In particular, we see that $\Phi$ is a unique solution to free wave equation when $t > S$  and we denote this solution by $\Phi_+$. Similarly, $\Phi$ is also a  solution to  free wave equation, when $t <-S$, and we denote this by $\Phi_-$, see the regions in \cite[Fig. 1]{KT98}.

	\section{Generalisations to the $L_{\loc}^1$ data}\label{sec:generalisation-L1loc}
	
	In our global existence results, see Theorems \ref{thm-M-global wave-unBndTrap} and \ref{thm-uniqueness-wave-map},  we have assumed that $Du_0$ and $v_0$ belong to the space $L^1(\bR;\bR^n)$. It is natural to ask if  these results remain true under the following weaker assumptions of the initial data
	\begin{align}\label{eqn-weak-init-data}
		u_0 \in C(\bR;\cM) \mbox{ and }  Du_0, v_0\in  L^1_{\loc}(\bR;\bR^n) \mbox{ and } f \in L^1_{\loc}(K;\bR^n),
	\end{align}
	where $K = [0,\infty) \times \bR$.
	In the  following  paragraph we answer positively to this question.
	
	Let us begin with the sequence $(K^n)_{n \in \bN}$ where each $K^n$ is a triangle with vertices $(0,-n)$, $(0,n)$ and $(n,0)$. Then, by using the theory developed in Theorem \ref{thm-M-globalwave-BndTrap}, we obtain a sequence $(u^n)_{n \in \bN}$, where, for each $n \in \bN$,
	\begin{equation}\label{eqn-u^n}
		u^n:K^n \to M
	\end{equation}
	is a global solution to wave map equation \eqref{eq:fwm-vec}	with initial data $u_{0,n}$ and $v_{0,n}$ being the restrictions of $u_0$ and $v_0$ respectively, to the base $[-n,n]$ of $K^n$ and external forcing being the restriction of $f$ to $K^n$. By the uniqueness result, Theorem \ref{thm-uniqueness-wave-map} and Corollary \ref{cor-finiteSOP}, $u^n$ is equal to the restriction of $u^{n+1}$ to $K^n$. Hence we can define a function
	\begin{equation}\label{eqn-u}
		u:K \to M
	\end{equation}
	by concatenating of functions $u^n$, see \cite[Theorem 4.4]{BR20} for similar construction.  Since the weak solution, see Definition \ref{def-sol-weak}, to problem \eqref{eq:fwm-vec} is defined in a local sense,  the function $u$  is a weak solution to the equation \eqref{eq:fwm-vec} on $[0,\infty) \times \bR$ with initial data $u_0$ and $v_0$ and external forcing $f$ as in \eqref{eqn-weak-init-data}.
	
	Next, we sketch the proof of uniqueness. Let us take $u^1$ and $u^2$ as two $\cM$-valued weak/mild solutions to equation \eqref{eq:fwm-vec} on $[0,\infty) \times \bR$ with  data $(u_0,v_0,f)$  satisfying \eqref{eqn-weak-init-data}. Since, due to Corollary \ref{cor-finiteSOP},  the restriction of $u^1, u^2$ to $K^n$ are solutions to \eqref{eq:fwm-vec} with data $(u_0|_{K_0^n}, v_0|_{K_0^n}, f|_{K_0^n})$. But then Theorem \ref{thm-uniqueness-wave-map} implies that $u^1 = u^2$ on $K^n$. Sine this holds for every $n \in \bN$, we infer $u^1=u^2$ on $K$.

	\section{Discussion of the uniqueness result by Zhou}\label{sec-Zhou}	
	In this section we deduce the Zhou uniqueness result \cite[Theorem 1.3]{Zhou99} from Theorems~ \ref{thm-M-global wave-unBndTrap} and \ref{thm-uniqueness-wave-map-unBndTrap} in the following way.
	
	Let us recall the Zhou definition of weak solution and the related uniqueness result first.
	\begin{definition}\cite[Definition 1.1]{Zhou99}\label{defn:zhou-weak-soln}
		Let $K= [0,T] \times \bR$ be a trapezoid   with base $K_0$.		Let us assume that the initial data  $(u_0, v_0) \in H^1(\bR) \times L^2(\bR)$. A function $u: K \to \bR^n$  is called a \emph{weak} solution of equation \eqref{eq:fwm-vec},    with the initial data $(u_0,v_0)$ if and only if  the following conditions are satisfied.
		\begin{trivlist}
			\item[(i)]   $u \in L^\infty([0,T]; H^1(\bR)) \cap W^{1,\infty}([0,T]; L^2(\bR))$;
			\item[(ii)]  for every  $\varphi \in  C_0^\infty((K))$, we have
			\begin{equation}\label{eqn-fwm-vector-zhou}
				\begin{split}
					& \iint_{K}	u(\partial_t^2 \varphi - \partial_x^2 \varphi)\ud x\ud t =  \int_{K_0} u_{0}(x)\partial_t \varphi(0,x)\ud x - \int_{K_0} v_{0}(x)\varphi(0,x)\ud x\\
					&+  \iint_{K } \bigl[ \sum_{j,k=1}^n \Gamma_{jk}(u)(\partial_t u_j\partial_t u_k - \partial_x u_j \partial_x u_k)
					\bigr]  \varphi(t,x) \ud x\ud t .
			\end{split}\end{equation}			
		\end{trivlist}
	\end{definition}
	The statement of main result of \cite{Zhou99} also require the following notion of a strong solution.
	\begin{definition}\cite[Definition 1.2]{Zhou99}\label{defn:zhou-strong-soln}
		Let $K= [0,T] \times \bR$ be a trapezoid   with base $K_0$.		Let us assume that the initial data  $(u_0, v_0) \in H^1(\bR) \times L^2(\bR)$. We say a function $u: K \to \bR^n$  is called a \emph{strong} solution of equation \eqref{eq:fwm-vec},    with the initial data $(u_0,v_0)$ if and only if  there exists a sequence of solution $\{u^n\}_{n \in \bN} \subset C^\infty(K; \bR^n)$ to equation  \eqref{eq:fwm-vec} on the time interval $[0,T]$ such that $u^n \to u$ strongly in $C([0,T]; H^1(\bR)) \cap C^1([0,T]; L^2(\bR))$ and if $(u^n(0,x), \partial_t u^n(0,x)) \to (u_0,v_0)$ strongly in $H^1(\bR) \times L^2(\bR)$.
	\end{definition}
	\begin{theorem}\cite[Theorem 1.3]{Zhou99}\label{thm-zhou-uniq}
		A weak solution to the Cauchy problem \ref{eq:fwm-vec} with initial data $(u_0,v_0) \in H^1(\bR) \times L^2(\bR)$ is unique and is in fact a strong solution.
	\end{theorem}
	Observe that the definition of a weak solution used by Zhou, i.e., Definition \ref{defn:zhou-weak-soln}, requires  \emph{apriori} that $u, \partial_x u, \partial_t u \in L^\infty([0,\infty); L^2(\bR))$. This, in turn, implies that a weak solution in the sense of Zhou is a mild  solution in the sense of Definition~\ref{def-sol-mild}
	and therefore Theorem \ref{thm-zhou-uniq}  follows from Theorem~\ref{thm-uniqueness-wave-map-unBndTrap}.
	
	Note that Theorem~\ref{thm-M-global wave-unBndTrap} does not imply that the integral solutions in the sense of Definition~\ref{def-sol-mild}
	are strong solutions in the sense of Definition \ref{defn:zhou-strong-soln}, where strong convergence in the energy space is required.

	\vspace{1 cm}
	\textbf{Acknowledgement:} The authors acknowledge email discussion with Prof. Yi Zhou regarding  inequality (2.40) on page 716 in the proof of Proposition 2.4 of his paper \cite{Zhou99}.
	JJ is supported by the ANR grant 23-ERCB-0002-01 ``INSOLIT''.
	NR acknowledges the financial support by the Royal Society research professorship of Prof. Martin Hairer, RP$\backslash$R1$\backslash$191065 which allowed him to visit the University of York to have discussions on this work during his post-doctoral tenure at Imperial College London. ZB is partially supported by National Natural Science Foundation of China (No. 12071433).

	\begin{appendices}
		
		\section{Continuous dependence of fixed points}\label{sec:continuity-lem}
		\begin{lemma}\label{lem-continuity}
			Let $(X,d)$ be a complete metric space. Let $\alpha < 1$ and $T_n: X \to X$ for $n \in \bN \cup \{\infty\}$ be maps satisfying
			\begin{align}
				& d(T_n x_1, T_n x_2) \leq \alpha d(x_1,x_2), \quad \forall n \in \bN, \label{continuity1}\\
				& \lim\limits_{n \to \infty} T_n x = T_\infty x \textrm{ in } X, \quad \forall x \in X. \label{continuity3}
				&
			\end{align}
			Then for every $n \in \bN \cup \{\infty\}$ there exists a unique $a_n \in X$ such that $T_n a_n = a_n$, and
			\begin{equation}
				\lim\limits_{n \to \infty} a_n  =a_\infty \textrm{ in } X.
			\end{equation}
		\end{lemma}
		\begin{proof}[Proof of Lemma \ref{lem-continuity}] From	\eqref{continuity1} and \eqref{continuity3}, we obtain
			\begin{equation*}
				d(T_\infty x_1, T_\infty x_2) \leq \alpha d(x_1,x_2).
			\end{equation*}
			Thus, existence and uniqueness of $a_n$ for all $n \in \bN \cup \{\infty\}$ follows from the Banach Fixed Point Theorem.
			We have \begin{align*}
				d(a_n,a_\infty ) & = d(T_n a_n, T_\infty a_\infty) \leq d(T_n a_n, T_na_\infty)+ d(T_n a_\infty, T_\infty a_\infty) \\
				& \leq \alpha d(a_n,a_\infty) + d(T_n a_\infty, T_\infty a_\infty).
			\end{align*}
			Thus together with  \eqref{continuity3} we get
			\begin{equation*}
				d(a_n,a_\infty)  \leq  \frac{1}{1-\alpha} d(T_n a_\infty, T_\infty a_\infty) \to 0 \textrm{ as } n \to 0.
			\end{equation*}
		\end{proof}
		
		\section{Some useful facts about Riemannian geometry}\label{app-geom}
		This section is a slightly expanded Appendix A from \cite{BO11}. In what follows we assume that $\cM$ is a compact $m$-dimensional Riemannian manifold embedded isometrically in $\bR^n$.
		
		\begin{lemma}\label{lem-projm} There exist a  $C^\infty_0$-class function
			\[P:\bR^n\to\bR^n\]
			and a neighbourhood $\UP$ of $\cM$ such that $P(\UP)=M$ and, for $p\in\UP$,
			\[
			P(p)=p \mbox{  if and only if  } p\in M.
			\]
			
		\end{lemma}
		\begin{proof}[Proof of Lemma \ref{lem-projm}] We will use a suitable smooth projection of the ambient space $\bR^n$ on the manifold. By Proposition 7.26, p. 200 in \cite{ONeill83}, there exists an open neighbourhood $\tilde V$ of the set $\{(p,0):p\in M\}$ in the normal bundle $NM$ and an open set $O\subseteq\bR^n $  such that the function $\mathcal E:\tilde V\ni (p,\xi) \mapsto p+\xi $ is a diffeomorphism. Let us  define a smooth map $\tilde P:O\to M$ as the composition of the natural projection  map $NM\to M:(p,\xi)\mapsto p\in M$ and $\mathcal E^{-1}$. Employing a partition of unity we can find  a  $C^\infty_0$-class function $P:\bR^n\to\bR^n$ and a  neighbourhood $V\subseteq O$ of $\cM$ such that
			$P$ has the claimed properties.
		\end{proof}

		\begin{lemma}\label{lem-projm-2}
			Let, for $p\in M$,    $\pi_p:\bR^n\to T_pM$ be the orthogonal projection.  Then the corresponding map
			\begin{equation}\label{eqn-pi}
			\pi: M \ni p \mapsto \pi_p \in \cL(\bR^n)  
			\end{equation}			
			is of $C^\infty$ class. In particular, it is Lipschitz continuous and bounded, i.e. there exists $C_2=C_2(M)$:
			\begin{equation}\label{eqn-pi_p-uniform}
				\Vert \pi_p \Vert_{\cL(\bR^n)} \leq C_2, \;\; p \in M.
			\end{equation}
		\end{lemma}
		The proof of this Lemma is omitted.

		By $C_{b,uc}(\bR;\bR^n)$ we denote a separable Banach space of all functions $u:\bR \to \bR^n$ which are uniformly continuous and bounded, see Example 1 in \cite[section IX.1]{Yosida95}, endowed with usual ``sup'' norm. By $C_{b,uc}(\bR;M)$ we denote the closed subset of $C_{b,uc}(\bR;\bR^n)$ consisting of all elements $u$ of the latter such that $u(x) \in M$, for all $x\in \bR$, endowed with the trace distance.

		\begin{lemma}\label{lem-density} There exist  Borel measurable mappings
			\begin{align}\label{eqn-Theta_k}
				\Theta_k &: \bigl[ C_{b,uc}(\bR;M) \cap  \dot{W}^{1,1} (\bR;\bR^n) \bigr] \times L^1(\bR; TM)
				\\& \to
				\bigl[ C_{00}(\bR;M) \cap  \dot{W}^{3,2} (\bR;\bR^n) \bigr] \times \bigl[ C_0(\bR;TM) \cap \dot{W}^{1,2} (\bR;\bR^n) \bigr]
			\end{align}
			such that for every $z\in H^1_{\textrm{loc}}\times L^2_{\textrm{loc}}(TM)$,
			\begin{equation}\label{eqn-convergence}
				\lim_{k\to\infty}\vert \Theta_k(z)-z\vert_{\dot{W}^{1,1} \times L^1}=0.
			\end{equation}
		\end{lemma}
		In the above, by $C_{00}(\bR;M)$ we understand the subset of functions which are elements of the space  $C_{b,uc}(\bR;M)$ and which are constants in the two disjoint
		intervals whose union is a complement of some compact interval.

		\begin{proof}[Proof of Lemma \ref{lem-density}]   Let $\UP$ be the neighbourhood of $\cM$  introduced  Lemma \ref{lem-projm}. Let us choose and fix  $\eps_0 >0$ such that
			\begin{equation}\label{eqn-tubular neighbourhood}
				M+\overline{B_{\eps_0}} \subset V.
			\end{equation}
			Let $(b_j)_{j=1}^\infty$ be an approximation of identity on $\bR^n$ of the form
			\begin{equation}
				\label{eqn-aoi}
				\begin{split}
					b_j(x)&\coloneqq  j^{-\frac{n}2}b(jx), \;\; x\in \bR^n, \\
					b&\in C_0^\infty(\bR^n), 1_{B(0,\frac12)} \leq b \leq 1_{B(0,1)}, \;\; \int_{\bR^n} b(x)\, \ud x=1.
				\end{split}
			\end{equation}
			Note that $b_1=b$ and $\supp (b_j) \subset B(0,\frac1{j})$.

			Let us choose and fix $k\in \mathbb{N}$. We will construct
			the map $\Theta_k$ as follows. \\
			Let $z=(u,v)\in  \bigl[ C_{b,uc}(\bR;M) \cap  \dot{W}^{1,1} (\bR;\bR^n) \bigr] \times L^1(\bR; TM) $.
			Note that this means in particular, that $v(x) \in T_{u(x)} M$ for almost all $x \in \bR$.
			Also, that the following limits (in $\cM$) exists
			\begin{align}
				u(-\infty)&\coloneqq  \lim_{x \to -\infty} u(x), \label{eqn-u(-infty)}\\
				u(+\infty)&\coloneqq  \lim_{x \to \infty} u(x). \label{eqn-u(infty)}
			\end{align}
			We put
			\begin{align}\label{eqn-u_k bar}
				\bar u_k(x)&=u(\operatorname{sgn}(x)\min\,\{k,|x|\}), \;\; x \in \bR,
				\\
				\label{eqn-v_k bar}
				\bar v_k(x)&=\mathbf 1_{(-k,k)}(x) v(x), \;\; x \in \bR.
			\end{align}
			Obviously,  $\bar u_k\in C_{b,uc}(\bR;M) \cap  \dot{W}^{1,1} (\bR;\bR^n)$ and
			\begin{equation}\label{eqn-Du_k bar-Du} D \bar u_k=\mathbf 1_{(-k,k)} D u
				\mbox{ in the weak sense}.
			\end{equation}
			Thus,  since $\bar{u}_k$ is uniformly continuous, for every $\eps_0>0$, we can find $\delta_0>0$ such that
			\begin{equation}\label{eqn-u_k abs cont}
				\vert  \bar u_k(x)-\bar u_k(y) \vert  \leq  \frac{\eps_0}{3} \mbox{ provided that } x,y\in \bR, \;\; \vert y-x\vert \leq \delta_0.
			\end{equation}
			Set \begin{equation}\label{eqn-m_k}
				m_k=\min\,\bigl\{j \in \mathbb{N}: j \geq k \mbox{ and }  \supp (b_j) \subset [-\frac{\delta_0}{2}, \frac{\delta_0}{2}] \bigr\}.
			\end{equation}
			Next, we define
			\begin{equation}\label{eqn-w_k}
				\bar{\bar{u}}_{k}=\bar u_k\ast b_{m_k} \mbox{ and }   \bar{\bar{v}}_{k}= \bar v_k\ast b_{m_k}.
			\end{equation}
			Note that $\bar u_k$ is $\cM$-valued but $\bar{\bar{u}}_{k}$ is $V$-valued.
			By \eqref{eqn-u_k abs cont}, \eqref{eqn-m_k}  and  \eqref{eqn-w_k}
			we infer that
			\[ \sup_{x\in\bR}|\bar{\bar{u}}_{k}(x)-\bar u_k(x)|\leq \frac{\eps_0}{3}< \eps_0,
			\]
			and $\bar{\bar{u}}_{k}$ is a  $C^\infty$-class function.
			Therefore  we deduce  that
			\begin{trivlist}
				\item[(i)] taking values in a compact  set $M+\overline{B_{\eps_0}}\subseteq \UP$,
				\item[(ii)] and such that it is constant on both interval $[k+\delta_0,\infty)$ and
				$(-\infty, -k-\delta_0]$. More precisely,
				\begin{equation}
					\bar{\bar{u}}_{k}(x) =
					\begin{cases}
						u(-k),  & \mbox{ for } x\in (-\infty, -k-\delta_0], \\
						u(k),  & \mbox{  for } x\in [k+\delta_0,\infty).
					\end{cases}
				\end{equation}
			\end{trivlist}
			Similarly, $\bar{\bar{v}}_{k}$ is a  $C^\infty_0$-class function taking values in $\bR^n$
			and  $ \supp(\bar{\bar{v}}_{k}) \subset [-k-\delta_0,k+\delta_0]$.
			In particular,
			\begin{equation}\label{eqn-w_k-regularity}
				\bar{\bar{u}}_{k} \in \dot{W}^{3,2} (\bR;\bR^n).
			\end{equation}
			We set
			\begin{align}
				u_k&=P \circ \bar{\bar{u}}_{k}, \\
				v_k(x)&=\pi_{u_k(x)}\big( \bar{\bar{v}}_{k}(x)\big), \;\; x\in\bR^n,
			\end{align}
			and
			\begin{align}
				\Theta_k(u)&=(u_k,v_k).
			\end{align}
			It is easy to  see that
			\begin{equation}\label{eqn-u_k-regularity}
				u_k \in \dot{W}^{3,2} (\bR;\bR^n) \cap  C_{00}(\bR;M),
			\end{equation}
			\begin{equation}\label{eqn-v_k-regularity}
				v_k \in C_0(\bR;TM) \cap \dot{W}^{1,2} (\bR;\bR^n),
			\end{equation}
			and
			\begin{equation}\label{eqn-v_k in T_u_k}
				v_k(x) \in T_{u_k(x)}M, \;\; x\in\bR^n.
			\end{equation}
			Thus   $\Theta_k$ maps the space
			\[ \bigl[ C_{b,uc}(\bR;M) \cap  \dot{W}^{1,1} (\bR;\bR^n) \bigr] \times L^1(\bR; TM)
			\]
			to
			\[
			\bigl[ C_{00}(\bR;M) \cap  \dot{W}^{3,2} (\bR;\bR^n) \bigr] \times \bigl[ C_0(\bR;TM) \cap \dot{W}^{1,2} (\bR;\bR^n) \bigr]
			\]
			as requested.
			Moreover, it  is easy to  show that the map $\Theta_k$ is measurable in the above sense. This completes the first part of the proof.
			
			Now we can consider sequences $\{u_k\}_{k \in \bN}$ and $\{v_k\}_{k \in \bN}$ constructed above.  Note that by the construction $m_k \to \infty$ as $k \to \infty$.
			To conclude the proof let us observe that the convergence \eqref{eqn-convergence} is equivalent to
			\begin{align}
				& u_k \to  u \mbox{ in } C_{b,uc}(\bR;M) \mbox{ or, equivalently,  in } C_{b,uc}(\bR;\bR^n),
				\label{eqn-convergence-u_k-L^infty}\\
				&\lim_{k\to\infty}\vert Du_k- Du\vert_{L^1}=0, \label{eqn-convergence-u_k-L^1}\\
				&\lim_{k\to\infty}\vert v_k- v\vert_{L^1}=0. \label{eqn-convergence-v_k}
			\end{align}

			\begin{proof}[Proof of \eqref{eqn-convergence-u_k-L^infty}]
				Since the limits \eqref{eqn-u(-infty)} and \eqref{eqn-u(infty)} exists,
				we infer that $\bar{u}_k \to u$  in $C_{b,uc}(\bR;\bR^n)$. Since by standard arguments $u \ast b_{m_k} \to u $   in  $C_{b,uc}(\bR;\bR^n)$ and
				\[
				\vert \bar u_k \ast b_{m_k} - u \ast b_{m_k} \vert_{C_{b,uc}}\leq  \vert \bar u_k  - u  \vert_{C_{b,uc}} \vert  b_{m_k}\vert_{L^1}= \vert \bar u_k  - u  \vert_{C_{b,uc}} \to 0,
				\]
				we deduce that $u_k \to  u$  in   $C_{b,uc}(\bR;\bR^n)$.
			\end{proof}

			\begin{proof}[Proof of \eqref{eqn-convergence-v_k}.]
				
				Firstly we can prove by a standard argument that    $\vert \bar{\bar{v}}_{k}  - v\vert_{L^1} \to 0$.
				Secondly, by the assumptions of $u$ and $v$ we infer that $\pi_{u(x)}v(x)=v(x)$ for almost all $x \in \bR$. Hence
				\begin{align}
					\vert v_k-  v\vert_{L^1}&=\int_{\bR} \vert v_k(x)-  v \vert \, \ud  x
					=\int_{\bR} \vert \pi_{u_k(x)}\big( \bar{\bar{v}}_{k}(x)\big) - \pi_{u(x)}v(x)  \vert \, \ud  x
					\\
					&\leq \int_{\bR} \vert \pi_{u_k(x)}\big( \bar{\bar{v}}_{k}(x)-v(x) \big)\vert \, \ud  x   + \int_{\bR} \vert \bigl(\pi_{u_k(x)}v(x)  - \pi_{u(x)}v(x) \bigr) \vert \, \ud  x
					\\
					&\leq C_2 \int_{\bR} \vert  \bar{\bar{v}}_{k}(x)-v(x) \vert \, \ud  x +
					\int_{\bR} \vert \bigl(\pi_{u_k(x)}v(x)  - \pi_{u(x)}v(x) \bigr) \vert \, \ud  x,
				\end{align}
				where the last inequality is a consequence of bound \eqref{eqn-pi_p-uniform} in Lemma \ref{lem-projm-2}. \\
				By the first claim made in the current proof we infer that the 1st term on the RHS above converges to $0$.
				The second term on the RHS above also converges to $0$ because the map $\pi$ is Lipschitz by Lemma \ref{lem-projm-2}.
				This completes the proof of \eqref{eqn-convergence-v_k}.
			\end{proof}

			\begin{proof}[Proof of \eqref{eqn-convergence-u_k-L^1}]
				Since $u(x) \in \cM$, $u(x)=P(u(x))$. Consequently we have
				\begin{align}\label{eqn-convergence-u_k-L^1-1}
					|Du_k - Du|_{L^1} & = |(d_{\bar{\bar u}_k(x)} P) (D \bar{\bar u}_k) - (d_{u(x)} P) (D u_k)|_{L^1}\\
					& \leq C_1 | D \bar{\bar u}_k - D u|_{L^1} + |Du|_{L^1} \sup_{x \in \bR}|d_{\bar{\bar u}_k} P  - d_{u(x)} P|,
				\end{align}
				where $C_1$ is the bound on $\sup_{u \in \bR^n}|d_u P|$ that we get from Lemma \ref{lem-projm}. The second term in the RHS above goes to $0$ as $k \to \infty$ because due to Lemma \ref{lem-projm} there exists $\tilde{C}_1>0$ such that
				\begin{align*}
					\sup_{x \in \bR}|d_{\bar{\bar u}_k} P  - d_{u(x)} P| \leq \tilde{C}_1 \sup_{x \in \bR} |\bar{\bar u}_k(x) - u(x)|,
				\end{align*}
				but this converges to $0$, as $k \to \infty$, see the proof of \eqref{eqn-convergence-u_k-L^infty}. To prove the convergence of the first term in the RHS of \eqref{eqn-convergence-u_k-L^1-1} we observe that the Young convolution inequality gives
				\begin{align}
					| D \bar{\bar u}_k - D u|_{L^1} & = | (D \bar u_k) \ast b_{m_k} - D u |_{L^1} \leq | (D \bar u_k) \ast b_{m_k} - (D u) \ast b_{m_k} |_{L^1} + | (D u) \ast b_{m_k} - D u |_{L^1}\\
					& \leq | D \bar u_k - D u |_{L^1}+ | (D u) \ast b_{m_k} - D u |_{L^1}.
				\end{align}
				Here the second term in RHS converges to $0$, as $k \to \infty$, since $Du \in L^1$ and $b_{m_k}$ is the approximation of identity. The first term in RHS above also converges to $0$, as $k \to \infty$, because $D u \in L^1$ and by definition of $\bar u_k$
				$$ | D \bar u_k - D u |_{L^1} = \int_{|x|\geq k} |Du (x)| \, \ud  x.$$
			\end{proof}
			Hence, the proof of Lemma \ref{lem-density} is complete.
		\end{proof}

		By an inspection of the proof of Lemma \ref{lem-density} we deduce that the following extension/modification of it holds as well.
		\begin{lemma}\label{lem-projm-3} Assume that  $K_0$ is a compact interval.
			There exist a  Borel measurable map
			\begin{align}\label{eqn-Theta_k-bnd}
				\Theta_k &: \bigl[ C(K_0;M) \cap  \dot{W}^{1,1} (K_0;\bR^n) \bigr] \times L^1(K_0; TM)
				\\& \to
				\bigl[ C(K_0;M) \cap  \dot{W}^{3,2} (K_0;\bR^n) \bigr] \times \bigl[ C_0(K_0;TM) \cap \dot{W}^{1,2} (K_0;\bR^n) \bigr]
			\end{align}
			such that for every $z\in  \bigl[ C(K_0;M) \cap  \dot{W}^{1,1} (K_0;\bR^n) \bigr] \times L^1(K_0;TM)$,
			\begin{equation}\label{eqn-convergence-12}
				\lim_{k\to\infty}\vert \Theta_k(z)-z\vert_{ \bigl( C(K_0;M) \cap \dot{W}^{1,1}(K_0) \bigr)  \times L^1(K_0)}=0.
			\end{equation}
		\end{lemma}

	\end{appendices}

	\bibliographystyle{plain}
	\bibliography{forced-wave-maps}

\end{document}